\documentclass[11pt, letterpaper, oneside]{amsart}

\usepackage[margin=2.54cm]{geometry}% big spacing for "draft mode"

\usepackage{latexsym, amsmath, amssymb, amsfonts, amscd,stmaryrd}
\usepackage{amsthm}
\usepackage{t1enc}
\usepackage[mathscr]{eucal}
\usepackage{indentfirst}
\usepackage[all]{xy}
\usepackage{graphicx}
\usepackage{fancyhdr}
\usepackage{fancybox}
\usepackage{enumerate}
\usepackage[all]{xy}

\usepackage{tikz}
\usetikzlibrary{matrix,arrows}
\usetikzlibrary{calc}

\usepackage{url}
\usepackage{color}
\usepackage{hyperref}
\usepackage[normalem]{ulem}

\theoremstyle{plain}
\newtheorem{thm}{Theorem}[section]

\newtheorem{prop}[thm]{Proposition}
\newtheorem{lemma}[thm]{Lemma}
\newtheorem{cor}[thm]{Corollary}

\newtheorem{defprop}[thm]{Definition/Proposition}

\theoremstyle{definition}
\newtheorem{defi}[thm]{Definition}

\theoremstyle{definition}
\newtheorem{remark}[thm]{Remark}
\newtheorem{ep}[thm]{Example}

\newtheorem*{ack}{Acknowledgements}

\newcommand{\ZZ}{\ensuremath{\mathbb Z}}

\newcommand{\RR}{\ensuremath{\mathbb R}}
\newcommand{\su}{\ensuremath{\frak{su}}}
\newcommand{\g}{\ensuremath{\frak{g}}}
\renewcommand{\k}{\ensuremath{\frak{k}}}

\newcommand{\y}{\ensuremath{\frak{y}}}
\newcommand{\s}{\ensuremath{\frak{s}}}
\newcommand{\m}{\ensuremath{\frak{m}}}

\newcommand{\bs}{\mathbf{s}}                  %source

\definecolor{forest}{rgb}{0,0.5,0} 
\definecolor{orange}{rgb}{1,0.4,0}

\newcommand{\cG}{\mathcal{G}}

\newcommand{\cL}{\mathcal{L}}

\newcommand{\cC}{\mathcal{C}}
\newcommand{\cA}{\mathcal{A}}

\newcommand{\li}{\ensuremath{L_{\infty}}}
\newcommand{\pd}[1]{\frac{\partial}{\partial #1}} %\pd{x}

\newcommand{\bd}{\mathbf{d}}
\newcommand{\del}{\partial}
\newcommand{\cinf}{C^{\infty}}
\newcommand{\LX}{\mathfrak{X}^{\wedge {\bullet}}}
\renewcommand{\L}{\mathcal{L}}
\newcommand{\X}{\mathfrak{X}}

\newcommand{\Sh}{\mathrm{Sh}}
\newcommand{\Sn}{\mathcal{S}}
\newcommand{\maps}{\colon}
\newcommand{\tensor}{\otimes}
\renewcommand{\deg}[1]{\left \lvert #1 \right \rvert}

\newcommand{\ds}{\mathbf{s}^{-1}}                  
\newcommand{\innerprod}[2]{\bigl \langle #1  ,  #2 \bigr \rangle}
\newcommand{\pr}{\mathrm{pr}}
\newcommand{\epi}{\twoheadrightarrow}
\renewcommand{\S}{\bar{S}}
\newcommand{\rDelta}{\bar{\Delta}}
\newcommand{\rdDelta}[1]{\bar{\Delta}^{(#1)}}
\newcommand{\vs}{\varsigma}
\newcommand{\sQ}{Q^{1}}
\newcommand{\sF}{F^{1}}

\newcommand{\poi}{\text{$\mathsf{Ham}_{\infty}$}}

\newcommand{\Linf}{\mathsf{Lie}_{\infty}}
\newcommand{\str}{\mathfrak{string}}

\newcommand{\Xham}{\mathfrak{X}_{\mathrm{Ham}}}
\newcommand{\Xlham}{\mathfrak{X}_{\mathrm{LHam}}}
\newcommand{\ham}[1]{\Omega^{#1}_{\mathrm{Ham}}\left(M\right)}
\newcommand{\ip}[1]{\iota_{v_{#1}}}
\newcommand{\brac}[2]{\left \{ #1,#2 \right\}}

\newcommand{\alphak}[1]{\alpha_{1} \tensor \cdots \tensor \alpha_{#1}}
\newcommand{\alphadk}[1]{\alpha_{1},\hdots,\alpha_{#1}}

\newcommand{\vk}[1]{v_{\alpha_{1}} \wedge \cdots \wedge  v_{\alpha_{#1}}}

\newcommand{\xto}[1]{\xrightarrow{#1}}

\newcommand{\gl}{\ensuremath{\frak{gl}}}

\newcommand{\CE}{\mathrm{CE}}
\newcommand{\U}{\mathrm{U}}

\newcommand{\R}{\mathbb{R}}

\newcommand{\Z}{\mathbb{Z}}
\newcommand{\cAf}{\mathcal{A}_{\mathrm{flat}}}

\newcommand{\fat}[1]{\bigl \lVert #1 \bigr \rVert}

\newcommand{\bu}{\bullet}

\newcommand{\EGM}{E_{\bu}G \times M}
\newcommand{\bOmega}{\Omega^{\ast}}
\newcommand{\bX}{X_{\bullet}}

\DeclareMathOperator{\hor}{\mathrm{hor}}

\DeclareMathOperator{\Hom}{\mathrm{Hom}}

\DeclareMathOperator{\id}{\mathrm{id}}
\DeclareMathOperator{\im}{\mathrm{im}}

\DeclareMathOperator{\Ad}{\mathrm{Ad}}
\DeclareMathOperator{\ad}{\mathrm{ad}}
\DeclareMathOperator{\Tr}{\mathrm{Tr}}
\DeclareMathOperator{\Lie}{\mathrm{Lie}}
\DeclareMathOperator{\supp}{\mathrm{supp}}
\DeclareMathOperator{\curv}{\mathrm{curv}}
\DeclareMathOperator{\Ham}{\mathrm{Ham}}
\DeclareMathOperator{\dR}{\mathrm{dR}}
\DeclareMathOperator{\GL}{\mathrm{GL}}
\DeclareMathOperator{\SO}{\mathrm{SO}}
\DeclareMathOperator{\SU}{\mathrm{SU}}

\DeclareMathOperator{\Tot}{\mathrm{Tot}}

\DeclareMathOperator{\Car}{\mathrm{Car}}
\DeclareMathOperator{\Alt}{\mathrm{Alt}}
\DeclareMathOperator{\Sym}{\mathrm{Sym}}
\DeclareMathOperator{\pt}{\mathrm{pt}}
\DeclareMathOperator{\Map}{\mathrm{Map}}
\DeclareMathOperator{\spl}{\mathrm{spl}}

\begin{document}

\title{Homotopy moment maps}

\author{Martin Callies} 
\address{Mathematisches Institut, Georg-August Universit\"at G\"ottingen, Bunsenstrasse 3-5, D-37073 G\"ottingen, Germany. Current address: Fakult\"{a}t f\"{u}r Mathematik, SFB 701,
Universit\"{a}t Bielefeld, Postfach 100131, 33501 Bielefeld,
Germany } \email{mcallies@math.uni-bielefeld.de}

\author{Ya\"el Fr\'egier}
\address{MIT/University of Zurich/Universit\'e d'Artois} \email{yael.fregier@math.mit.edu} 

\author{Christopher L.\ Rogers} 
\address{Department of Mathematics and Statistics, University of
  Nevada, Reno, 1664 N. Virginia Street Reno, NV 89557-0084 USA
 } \email{chrisrogers@unr.edu, chris.rogers.math@gmail.com}

\author{Marco Zambon}
\address{KU Leuven, Department of Mathematics, Celestijnenlaan 200B box 2400, BE-3001 Leuven, Belgium}
\email{marco.zambon@wis.kuleuven.be}  


\begin{abstract}
  Associated to any manifold equipped with a closed form of degree
  $>1$ is an `$L_{\infty}$-algebra of observables' which acts as a higher/homotopy
  analog of the Poisson algebra of functions on a symplectic
  manifold.  In order to study Lie group actions on these manifolds,
  we introduce a theory of homotopy moment maps. Such a map is a
  $L_{\infty}$-morphism from the Lie algebra of the group into the
  observables which lifts the infinitesimal action.  We
  establish the relationship between homotopy moment maps and
  equivariant de Rham cohomology, and analyze the obstruction theory
  for the existence of such maps. This allows us to easily and
  explicitly construct a large number of examples. These include
  results concerning group actions on loop spaces and moduli spaces of
  flat connections. Relationships are also established with previous
  work by others in classical field theory, algebroid theory, and dg
  geometry. Furthermore, we use our theory to geometrically construct
  various $L_{\infty}$-algebras as higher central extensions of Lie
  algebras, in analogy with Kostant's quantization theory. In
  particular, the so-called `string Lie 2-algebra' arises this way.
\end{abstract}

\maketitle
\setcounter{tocdepth}{1} %doesn't display subsections in TOC
\tableofcontents

\section{Introduction}\label{introitus} 

This paper represents part of a larger project which involves studying the
symmetries of manifolds equipped with a closed differential
form. The motivation for this work stems from the desire to have a
more conceptual understanding of the role these manifolds play in
differential cohomology, generalized geometry, and 
field theory. 
In our approach, we view such manifolds as generalizations of
symplectic manifolds.

As a first step, we consider symmetries arising from a Lie group
acting on a manifold by diffeomorphisms which preserve a closed
differential form. A key component of our formalism is the 
`homotopy moment map'. This is a natural generalization of the moment
map used to study symmetries in symplectic geometry. However,
unlike their symplectic counterparts, our moment maps do not
correspond to morphisms between Lie algebras. Instead, they are morphisms
between objects called `$L_{\infty}$-algebras', which can
be thought of as homotopy-theoretic upgrades of Lie algebras. More
precisely, an $L_{\infty}$-algebra is a graded vector
space equipped with a skew-symmetric bracket which satisfies the
Jacobi identity up to coherent homotopy. The coherent homotopy is given as part
of the data by an infinite sequence of higher degree multi-linear
brackets which satisfy additional Jacobi-like identities.
$L_{\infty}$-algebras with underlying vector spaces concentrated in the first
non-positive $n$-degrees are often called `Lie $n$-algebras'. In
particular, a Lie 1-algebra is an ordinary Lie algebra.

Morphisms between $L_{\infty}$-algebras are not just linear maps which
preserve the brackets. This definition is too strict. Rather, a
morphism is an infinite collection of multi-linear maps which preserve
the brackets up to, again, coherent homotopy. We emphasize that this
notion of morphism between $L_{\infty}$-algebras plays a crucial role.

Perhaps it seems strange that these higher homotopical structures should appear
when studying something as classical as actions of Lie groups on manifolds. 
To understand why they are needed, we have to first recall some facts concerning
symmetries in symplectic geometry. 

\subsection{Symplectic geometry}
The important infinitesimal symmetries of a symplectic manifold
correspond to the Hamiltonian vector fields. These form a
Lie algebra whose bracket is the usual commutator of vector
fields. The space of smooth functions is also a Lie algebra, whose
bracket is specified by the symplectic 2-form. {This is the underlying
Lie algebra of the Poisson algebra.}
If the manifold is connected (we always assume
this is the case), then Kostant \cite{Kostant:1970} showed that the
Poisson algebra is characterized as a particular extension of the
Lie algebra of Hamiltonian vector fields by $\R$. The 2-cocycle
representing this central extension is determined by the symplectic
form.

Now suppose we have a Lie group $G$, with Lie algebra $\g$, acting on
the manifold via diffeomorphisms which preserve the symplectic
form. Assume further that the associated infinitesimal action is given
by a Lie algebra morphism from $\g$ to the Hamiltonian vector
fields. A `moment map' for the action corresponds to a lift of this Lie algebra morphism to the
central extension given by the Poisson algebra\footnote{Technically, this lift is not the moment map, but rather
the co-moment map.}. Whether a moment map exists or not for a
particular $G$-action is an important question in symplectic geometry.
It can be thought of as the symplectic analog of determining when a
projective representation of $G$ lifts to a linear one.

The relationship between symmetries in symplectic geometry and
representation theory is made explicit via `geometric quantization'. If
the symplectic form represents an integral cohomology class, then it
corresponds to the curvature of a principal $U(1)$-bundle equipped with
a connection. In this case, the Poisson algebra is
isomorphic to a Lie algebra consisting of the $U(1)$-invariant vector fields on the
bundle whose flows preserve the connection. This
Lie algebra acts naturally as differential operators on sections of
the associated Hermitian line bundle. Hence, if there is a $G$-action
on the symplectic manifold, then a moment map for this action gives a
representation of the Lie algebra $\g$ on the space of sections. In
certain cases, this action integrates to a global $G$-action.

If no moment map exists, then Kostant's construction produces
non-trivial central extensions of both $\g$ and $G$ which naturally act on
the space of sections of the Hermitian line bundle. Many important Lie
groups can be constructed this way e.g.\ central extensions of loop groups,
as well as the Heisenberg and Bott-Virasoro groups \cite[Sec.\ 2.4.]{Brylinski_book}.

\subsection{``Higher'' symplectic geometry}
Let us return to the more general case, and consider a manifold
equipped with a closed form of degree $n+1 > 2$. 
Such a manifold also has Hamiltonian vector fields, and these form
a Lie algebra just as they do in symplectic geometry.
To pursue the analogy further, one might try to construct a central
extension of the Hamiltonian vector fields using the closed $(n+1)$-form.
Unlike the symplectic case, the form does not
induce a skew-symmetric bracket on functions. But it does on a
particular subspace of $(n-1)$-forms, called
Hamiltonian $(n-1)$-forms. This bracket, however, fails to satisfy the Jacobi
identity. 

This lack of a genuine Lie bracket is the reason why
$L_{\infty}$-algebras appear.  In previous work \cite{RogersL}, the third author
considered such manifolds when the closed $(n+1)$-form satisfied a
mild non-degeneracy condition.  These are called `$n$-plectic' or
`multisymplectic' manifolds. He associated to such a manifold a Lie
$n$-algebra whose underlying vector space consists of the Hamiltonian $(n-1)$-forms and
all other forms of lower degree. Its brackets are completely
determined by the $(n+1)$-form and the de Rham differential. 
Later, the fourth author showed
that the non-degeneracy assumption is not necessary for the
construction, and therefore any `pre-$n$-plectic' manifold has
such a Lie $n$-algebra \cite{HDirac}.

In this work, we slightly generalize these previous constructions
and associate to any manifold equipped with a closed $(n+1)$-form
its `Lie $n$-algebra of observables'. When the form is non-degenerate,
this Lie $n$-algebra is isomorphic to the one constructed in
\cite{RogersL}, and in particular, we recover the underlying Lie
algebra of the usual Poisson algebra when $n=1$. 
In analogy with Kostant's central extension for a
symplectic manifold, this Lie $n$-algebra of observables can be
characterized uniquely, up to homtopy, as a $L_{\infty}$-extension of 
the Lie algebra of Hamiltonian vector fields whose classifying cocycle
is determined by the closed form. (See Thm.\ 3.4.1 in \cite{FRS}.)

If a Lie group $G$ acts on a pre-$n$-plectic manifold $(M,\omega)$ and the
infinitesimal action induces a Lie algebra morphism between $\g$ and
the Hamiltonian vector fields, then we define a
homotopy moment map, or  just `moment map' for short, to be a lift
of this Lie algebra morphism to an $L_{\infty}$-morphism from $\g$ to
the Lie $n$-algebra {of observables. The precise
 definition is given in Def.\ \ref{main_def}.
For $n=1$, we recover the usual notion of a
(co)-moment map in pre-symplectic geometry.}
At first sight, this definition may seem
too abstract or technical to be useful. However, thanks to some of the
tools developed here, we can easily and systematically construct such maps and
therefore produce a large variety of interesting examples.
{In almost all of these, the moment map is not a ``strict''
$L_{\infty}$-morphism.} 

{The ultimate goal is to complete the analogy with the symplectic
  case by understanding the role homotopy moment maps play in
  quantization and representation theory. Indeed, the results in
  \cite{FRS} imply that a homotopy moment map will lift a $\g$-action on
  $(M,\omega)$ to an action on a `higher bundle gerbe' over $M$ whose
  curvature is $\omega$. (See also Remark \ref{inf_sym_remark}
  below). Along with this, we are also interested in pursuing the
  geometric relationship between these moment maps and conserved
  quantities, in the sense of Hamiltonian dynamics and also classical
  field theory. (See Sec.\ \ref{end_sec} for further details.) }

\subsection{Summary of results}
Our exposition throughout is aimed at a broad audience 
of geometers and topologists. We assume the reader has essentially
no expertise in homotopical algebra or familiarity with higher geometric structures.

We begin with a quick introduction to $L_{\infty}$-algebras in Sec.\
\ref{Linfty_sec} and leave the more technical aspects to the appendix.
We review the necessary background on $n$-plectic geometry,
Hamiltonian vector fields, and the Lie $n$-algebra of observables in
Sec.\ \ref{nplectic_sec}.  
{We introduce the homotopy moment map in
Sec.\ \ref{sec:momaps}}.

\subsubsection*{Equivariant de Rham cohomology}
In Sec.\ \ref{equi_deRham_sec}, we present our first main result:
If we have a compact Lie group acting on a manifold,
then from any $(n+1)$-cocycle in equivariant de Rham cohomology we can
naturally and explicitly produce a $G$-invariant pre-$n$-plectic
structure and a homotopy moment map (Thm.\ \ref{thm:BSmm} and Thm.\ \ref{thm:CartanMM}).
The formula for the moment map generalizes the important
relationship between moment maps in symplectic geometry and degree 2
cocycles in equivariant cohomology. We find this result particularly
interesting since it suggests a geometric interpretation of higher
degree equivariant cocycles.

\subsubsection*{Closed 3-forms}
The first truly new (i.e.\ non-symplectic) examples will arise on
manifolds equipped with a closed 3-form. Such manifolds also play an
important role in generalized geometry and the theory of gerbes.  
So in Sec.\ \ref{sec:c3f}, we focus on aspects specific to this case.

\subsubsection*{Basic examples}
We then present some basic examples in Sec.\ \ref{examples_sec}. These include:
\begin{itemize}
\item {Exact pre-plectic forms: This generalizes familiar results in
  symplectic geometry involving $G$-invariant symplectic potentials.
  Special cases include $G$-actions on exterior powers of cotangent bundles,
  and the action of $\mathrm{SO}(n)$ on $\R^n$ equipped with the usual volume form.
}
\item{Compact Lie groups: The Cartan 3-form on such a group is
    invariant under conjugation and can be uniquely extended to an equivariant
    closed 3-form. This gives a moment map for the adjoint
    action. We point out a relationship between this moment map and certain
    quasi-Hamiltonian $G$-spaces.
 }
\item{$\mathrm{SO}(n)$-action on the $n$-sphere: This generalizes the
  Hamiltonian circle action on $S^2$ whose moment map corresponds to the
   ``height function'' along the $z$-axis.
}
\end{itemize}

\subsubsection*{Obstructions and higher central extensions}
In order to produce more examples, we study in Sec.\ \ref{obstruct_sec} the obstructions
to lifting a $G$-action to a homotopy moment map.
The results we present here are natural
generalizations of the symplectic ones. The
existence of a moment map for a $G$-action on a (connected)
pre-$n$-plectic manifold implies that a degree $(n+1)$ class $[c]$ in Lie algebra
cohomology is trivial. Conversely, if $[c]=0$ and
$M$ satisfies certain topological assumptions, then we can always
construct a moment map lifting the action (Thm.\ \ref{thm:momapyes}).

If $[c] \neq 0 $, then in Sec.\ \ref{central_ext_sec}
we show how to construct a $L_{\infty}$-morphism not from $\g$, but
rather from a Lie $n$-algebra $\widehat{\g}$.
The Lie $n$-algebra $\widehat{\g}$ is built using a representative of
$[c]$ and plays the role of a (non-trivial) higher central extension
of $\g$. This gives a new way to geometrically construct Lie $n$-algebras.
For example, via this construction we recover the string Lie 2-algebra
$\str(\g)$, which plays an interesting role in elliptic cohomology,
and `string structures'.

\subsubsection*{Moduli spaces and loop spaces}
In Sec.\ \ref{sec:flatc}, we use the results of Sec.\
\ref{obstruct_sec} to produce a more sophisticated example of a
homotopy moment map on an infinite-dimensional manifold.  If $P$
is a principal $G$-bundle on a $(n+1)$-dimensional compact oriented
manifold, then a degree $(n+1)$ invariant polynomial on $\g$ gives a
pre-$n$-plectic structure on the space of connections of $P$. We show
that this $(n+1)$-form is invariant under the action of the gauge
group, and that this action admits a moment map. If the $(n+1)$-form
is actually $n$-plectic, then we can perform a Marsden-Weinstein
reduction procedure to obtain the moduli space of flat
connections, endowed with a pre-$n$-plectic form. This generalizes the well-known Atiyah-Bott construction
in symplectic geometry for $G$-bundles over Riemann surfaces.

We continue to focus on infinite-dimensional examples in
Sec.\ \ref{loop_sec}. There we show that a homotopy moment map for a $G$-action 
on a pre-2-plectic manifold $(M,\omega)$ can be transgressed to
an ordinary moment map on the pre-symplectic loop space
$(LM,L\omega)$, where $L\omega$ is transgression of $\omega$, and the
action of $G$ on $LM$ is ``point-wise''. This gives an example of how
the higher geometry on $M$ interacts with the classical geometry on
$LM$. 

\subsubsection*{Comparisons with other work}
Numerous generalizations of moment maps already exist in the literature.
In Sec.\ \ref{others_sec}, we describe some relationships between 
homotopy moment maps and related work done by others.
In particular, we consider: multi-momentum maps studied by a variety
of authors in multisymplectic field theory
\cite{IbortCarCrampMultimomap, GIMM}, 
the multi-moment maps of Madsen and Swann 
\cite{MadsenSwannClosed, MadsenSwannMultimomap}, 
Bursztyn, Cavalcanti, and Gualtieri's work on group actions and
Courant algebroids \cite{BCG}, and Uribe's
work on group actions and dg-manifolds \cite{UribeDGman}.
Finally, we conclude in Sec.\ \ref{end_sec} with some open questions.

\begin{ack} 
C.L.R.\ thanks Christian Blohmann for helpful conversations. M.Z.\
thanks Alberto Cattaneo, Ezra Getzler and Camille Laurent-Gengoux for
inspiring discussions related to this work. C.L.R.\ and M.Z.\
gratefully thank Bernardo Uribe for several helpful conversations, and for
sharing with us his ideas on the relationship between
equivariant cohomology and moment maps.
Y.F.\ thanks Christian Duval from whom he learned about Atiyah-Bott's interpretation of moment maps in terms of equivariant cohomology.

We thank the referee for useful remarks that improved the paper.  

C.L.R.\ acknowledges support from the German Research Foundation
(Deutsche Forschungsgemeinschaft (DFG)) through the Institutional
Strategy of the University of G\"{o}ttingen, as well as RMATH
(Luxembourg) and ICMAT (Madrid) for travel support and hospitality.
M.C.\ acknowledges support from DFG
through the Research Training Group GK1493 "Mathematical Structures in
Modern Quantum Physics".
Y.F.\ was supported by FNR through FNR/09/AM2c/04, by the MarieCurie IOF fellowship 274032, and by the
RMATH (warmest thanks to Martin Schlichenmaier). M.Z.\ was partially
supported by projects PTDC/MAT/098770/2008 and PTDC/MAT/099880/2008
(Portugal) and by projects MICINN RYC-2009-04065, MTM2009-08166-E,
MTM2011-22612, ICMAT Severo Ochoa project SEV-2011-0087 (Spain).

\end{ack}
 
\section{Preliminaries}
Here we list the notation and conventions used throughout the
paper. We also give a brief review of the Cartan calculus for multi-vector fields.
\subsection{Graded linear algebra} Let $V$ be a $\Z$ graded vector space.  
For any $k \in \ZZ$, $V[k]$ is the graded vector space
\[
V[k]^{i} = V^{i+k}.
\]

Let $x_{1},\hdots,x_{n}$ be
elements of $V$ and $\sigma \in \Sn_n$ a permutation. The
\textbf{Koszul sign} $\epsilon(\sigma)=\epsilon(\sigma ;
x_{1},\hdots,x_{n})$ is defined by the equality
\[
x_{1}  \cdots  x_{n} = \epsilon(\sigma ;
x_{1},\hdots,x_{n}) x_{\sigma(1)} \cdots  x_{\sigma(n)},
\]
which holds in the free graded commutative algebra generated by
$V$, with product denoted by concatenation of elements. Given $\sigma \in \Sn_n$, let $(-1)^{\sigma}$
denote the usual sign of a permutation. Note that $\epsilon(\sigma)$ does
not include the sign $(-1)^{\sigma}$. 

We say $\sigma \in \Sn_{p+q}$ is a {\bf $\mathbf{(p,q)}$-unshuffle}
iff $\sigma(i) < \sigma(i+1)$ whenever $i \neq p$.  The set of
$(p,q)$-unshuffles is denoted by $\Sh(p,q)$. 
For example, $\Sh(2,1)$ is the set of cycles $\{ (1), (23), (123) \}$.
{The notion of a $(i_1,\ldots,i_k)$-unshuffle and the set
$\Sh(i_1,\ldots,i_k)$ are defined similarly.}

If $V$ and $W$ are graded
vector spaces, a linear map $f \maps V^{\tensor n} \to W$ is
\textbf{skew-symmetric} iff
\[
f(v_{\sigma(1)},\hdots,v_{\sigma(n)}) = (-1)^{\sigma}\epsilon(\sigma)
f(v_{1},\hdots,v_{n}),
\]
for all $\sigma \in \Sn_{n}$. The degree of an element $x_{1} \tensor \cdots
\tensor x_{n} \in V^{\tensor \bullet}$ of the graded tensor algebra
generated by $V$ is defined to be $\deg{x_{1} \tensor \cdots
\tensor x_{n}}=\sum_{i=1}^{n} \deg{x_{i}}$. 

Finally, the following sign occurs frequently, so
we give its own notation. For an
integer $k$ define:
\begin{equation*}
\vs(k) =-(-1)^{\frac{k(k+1)}{2}}.
\end{equation*} 
% \begin{equation*}
% \vs(k) =
% \begin{cases}
% (-1)^{\frac{k}{2} +1} & \text{if $k$ even}\\
% (-1)^{\frac{k-1}{2}} & \text{if $k$ odd}\\
% \end{cases}
% \end{equation*}
So for $k=1,2,3,4,5,...$ we have $\vs(k)=1,1,-1,-1,1,...$.
Notice that $\vs(k-1)\vs(k)=(-1)^{k}$ for all $k$.

\subsection{Group actions}
Throughout this paper $G$ denotes a Lie group and $\g$ its Lie
algebra. A $G$ action on a manifold $M$ is always from the left, unless
stated otherwise. Our convention for the induced action on forms
is the one given in \cite[Sec.\ 2.1]{G-S}. Namely, $G$ acts on $\Omega^{\bullet}(M)$
from the left via inverse pullback
\[
g \cdot \omega \mapsto \phi^{\ast}_{g^{-1}} \omega,
\]
where $\phi_{g}$ is the diffeomorphism corresponding to $g$.
We denote the corresponding infinitesimal action of the Lie
algebra $\g$ by the map
\begin{equation} \label{lie_alg_action}
v_{-} \maps \g \to \X(M), \quad x \mapsto v_x,
\end{equation}
where
\[
v_{x} \vert_{p} = \frac{d}{dt} \exp(-tx) \cdot p \vert_{t=0} \quad
\forall p \in M.
\]
We call $v_{-}$ the \textbf{fundamental vector field} associated 
to the $G$-action. Note that it is \textit{minus} the infinitesimal generator
associated to the $G$-action, and hence it is a morphism of Lie
algebras.

If $G$ is finite-dimensional, then we 
denote by $\g^{\vee}$ the dual of the Lie algebra $\g$.
Recall that the Chevalley-Eilenberg differential $\delta_{\CE}$ on $\Lambda^*(\g^\vee)$ is
\begin{equation} \label{eq:CE_diff}
\begin{split}
 \delta_{\CE}&\colon\Lambda^n(\g^\vee)\to \Lambda^{n+1}(\g^\vee), \\
\delta_{\CE}(c)(x_1,\ldots,x_{n+1})&:= \sum_{1\leq i < j\leq n+1} (-1)^{i+j} c([x_i,x_j],x_1,\ldots,\hat{x}_i,\ldots,\hat{x}_j,\ldots,x_{n+1}).
\end{split}\end{equation}
If $M$ is a $G$-manifold, then
let $ \iota^{k}_{\g} \maps \Omega^{\ast}(M) \to \Lambda^{k}(\g^{\vee})
\tensor \Omega^{\ast -k}(M)$ be the insertion operations
\begin{equation}\label{eq:iota_g}
\begin{split}
 \iota_{\g}^k\colon\Omega^*(M) &\to \Lambda^k(\g^\vee) \tensor
 \Omega^{\ast -k}(M), \\ 
 \iota^k_{\g}\alpha(x_1,\ldots,x_k)&:=\iota_{v_{x_k}}\cdots\iota_{v_{x_1}}\alpha\in\Omega^{\ast-k}(M)
\end{split}\end{equation}
where $\alpha\in\Omega^\ast(M)$, and  $x_1,\ldots,x_k\in\g$.
We will also use $\iota_{\g}$ to denote $\iota_\g^1$.

\subsection{Cartan calculus}
The \textbf{Schouten bracket}
% $[\cdot,\cdot] \maps \LX(M) \times \LX(M) \to \LX(M)$ 
% is a degree $-1$ Lie bracket which satisfies the graded Leibniz rule with respect to the
% wedge product.
of two decomposable multivector fields
$u_{1} \wedge \cdots \wedge u_{m}, v_{1} \wedge \cdots \wedge v_{n}
\in \LX(M)$ is
\begin{multline} \label{Schouten}
\left [ u_{1} \wedge \cdots \wedge u_{m}, v_{1} \wedge \cdots \wedge
  v_{n} \right] 
= \\\sum_{i=1}^{m} \sum_{j=1}^{n} (-1)^{i+j} [u_{i},v_{j}]
\wedge u_{1} \wedge \cdots \wedge \hat{u}_{i} \wedge  \cdots \wedge
u_{m}\\
\quad \wedge v_{1} \wedge \cdots \wedge \hat{v}_{j} \wedge \cdots \wedge v_{n},
\end{multline}
where $[u_{i},v_{j}]$ is the usual Lie bracket of vector fields.

The \textbf{interior product} of a decomposable
multivector field $v_{1} \wedge \cdots \wedge v_{n}$ with $\alpha \in \Omega^{\bullet}(M)$ is
\begin{equation} \label{interior}
\iota(v_{1} \wedge \cdots \wedge v_{n}) \alpha = \iota_{v_{n}} \cdots
\iota_{v_{1}} \alpha,
\end{equation}
where $\iota_{v_{i}} \alpha$ is the usual interior product of vector
fields and differential forms. 
% The interior product of an arbitrary
% multivector field is obtained by extending by $\cinf(M)$-linearity. 

The \textbf{Lie derivative} $\L_{v}$ of a differential form along a multivector field $v \in
\LX(M)$ is the graded commutator of $d$ and $\iota(v)$:
\begin{equation} \label{Lie}
\L_{v} \alpha =  d \iota(v) \alpha - (-1)^{\deg{v}} \iota(v) d\alpha,
\end{equation}
where $\iota(v)$ is considered as a degree $-\deg{v}$ operator.

The last identity we will need is for the graded commutator of
the Lie derivative and the interior product. Given $u,v \in
\LX(M)$, it follows from  \cite[Proposition A3]{Forger} that
\begin{equation} \label{commutator}
\iota([u,v]) \alpha = (-1)^{(\deg{u}-1)\deg{v}} \L_{u} \iota(v)  \alpha - \iota(v)\L_{u} \alpha.
\end{equation}

\section{$L_{\infty}$-algebras} \label{Linfty_sec}
In this section we briefly review $L_{\infty}$-algebras and
explicitly describe $L_{\infty}$-morphisms for the special cases
considered in this paper.

\begin{defi}[\cite{Lada-Markl}] \label{Linfty} An
{\boldmath $L_{\infty}$}{\bf-algebra} is a graded vector space $L$
equipped with a collection
\[\left \{l_{k} \maps L^{\tensor k} \to L| 1
  \leq k < \infty \right\}\]
of
graded skew-symmetric linear maps with  $\deg{l_{k}}=2-k$ such that
the following identity holds for $1 \leq m < \infty :$
\begin{align} \label{gen_jacobi}
   \sum_{\substack{i+j = m+1, \\ \sigma \in \Sh(i,m-i)}}
  (-1)^{\sigma}\epsilon(\sigma)(-1)^{i(j-1)} l_{j}
   (l_{i}(x_{\sigma(1)}, \dots, x_{\sigma(i)}), x_{\sigma(i+1)},
   \ldots, x_{\sigma(m)})=0.
\end{align}
\end{defi}
In the appendix (Sec.\ \ref{coalg-Linfty}), we recall how any
$L_{\infty}$-algebra $(L,l_{k})$ corresponds to a certain kind of
graded coalgebra $C(L)$ equipped with a coderivation $Q$ which
satisfies the identity 
$$Q \circ Q =0.$$
This identity is the origin of Eq.\  \eqref{gen_jacobi}. But it is easy to see that for small values of $m$
that Eq.\  \eqref{gen_jacobi} is a ``generalized Jacobi identity'' for the multi-brackets
$\{l_{k}\}$. For $m=1$, it implies that the degree 1 linear map
$l_{1}$ satisfies
\[
l_{1} \circ l_{1}=0
\]
and hence every $L_{\infty}$-algebra $(L,l_{k})$ has an underlying cochain complex
$(L,l_{1})$.
\begin{defi} \label{LnA} An $L_{\infty}$-algebra $(L,\{l_{k} \})$
  is a {\bf Lie} {\boldmath $n$}{\bf -algebra} iff the underlying
  graded vector space $L$ is concentrated in degrees $0, -1, \ldots, 1-n$.
\end{defi}
Note that if $(L,\{l_{k} \})$ is a Lie $n$-algebra, then by degree
counting $l_{k} =0 $ for $k > n+1$. An ordinary Lie algebra is the
same as a Lie $1$-algebra.

\subsection{$L_{\infty}$-morphisms} \label{morph-in-text_sec}
The following definition may at first seem satisfactory:
\begin{defi}[\cite{Lada-Markl}] \label{strict_morph_def_1}
Let $(L,l_{k})$ and $(L',l'_{k})$ be $L_{\infty}$-algebras. A degree 0
linear map $f \maps L \to L'$ is a \textbf{strict
 \boldmath $L_{\infty}$-morphism} iff
\begin{equation} \label{strict_def_eq}
\begin{split}
l'_{k} \circ f^{\tensor k} = f \circ l_{k} \quad \forall k \geq 1.
\end{split}
\end{equation}
\end{defi}
However, this definition of $L_{\infty}$-morphism does not reflect
the higher structure naturally residing within the
theory. Indeed, the better definition \cite[Remark 5.3]{Lada-Markl} 
uses the aforementioned relationship between $L_{\infty}$-algebras and dg-coalgebras.
% This ultimately gives the collection of morphisms between two $L_{\infty}$-algebras the
% structure of a  simplicial set \cite{Hinich:2001}, which then 
% allows one to consider homotopies between morphisms, homotopies between
% homotopies, and so on. It is not necessary for the reader to understand
% these remarks precisely. However, 
We emphasize that the flexibility provided by
this higher structure is what allows us to produce the many explicit
examples of homotopy moment maps considered in this paper.

\begin{defi} \label{Linfty-morph_basic_def}
An {\boldmath $L_{\infty}$}-{\bf morphism} between
$L_{\infty}$-algebras $(L,l_{k})$ and $(L',l_{k}')$ is a morphism \newline
$F \maps \bigl(C(L),Q \bigr) \to \bigl(C(L'),Q' \bigr)$
between their corresponding differential graded (dg) coalgebras. That is, $F$ is a morphism
between the graded coalgebras $C(L)$ and $C(L')$ such that
\begin{equation} \label{preserve_codiff}
F \circ Q = Q' \circ F.
\end{equation} 
\end{defi}
It turns out that an $L_{\infty}$-morphism between $(L,l_{k})$ and
$(L',l'_{k})$ corresponds to an infinite collection of graded
skew-symmetric `structure maps'
\[
f_{k} \maps L^{\tensor k} \to L' \quad 1 \leq k < \infty,
\]
where $\deg{f_{k}} =1-k$, and such that a complicated compatibility relation with the
multi-brackets is satisfied. In particular, the degree zero map $f_{1}$ must be a
morphism between the underlying complexes $(L,l_{1})$ and $(L',l'_{1})$:
\[
f_{1}l_{1} = l'_{1} f_{1}.
\]
The compatibility relation, in the language of coalgebras, corresponds exactly to Eq.\  \eqref{preserve_codiff}.
Strict morphisms in the sense of Def.\ \ref{strict_morph_def_1}
correspond to the special case when $f_{k}=0$ for $k \geq 2$.
(See Prop.\ \ref{strict_morphism_prop} for more details.)
Outside of Sec.\ \ref{Linfty_morph_sec}, we shall mildly abuse
notation and denote a $L_{\infty}$-morphism via its structure maps as
\[
(f_{k}) \maps (L,l_{k}) \to (L',l'_{k}).
\]

$L_{\infty}$-morphisms are composable in the usual sense, and hence one can
speak of the category of $L_{\infty}$-algebras without explicitly describing
the higher structure mentioned above.
\begin{defi} \label{Linf_cat_def}
We denote by $\Linf$ the category whose objects are $L_{\infty}$-algebras
(Def.\ \ref{Linfty}) and whose morphisms are $L_{\infty}$-morphisms 
(Def.\ \ref{Linfty-morph_basic_def}). 
\end{defi}

The following is the correct notion of equivalence between
$L_{\infty}$-algebras which reflects the aforementioned homotopical structure between morphisms.
\begin{defi} \label{Linfty_qiso_def}
A morphism $(f_{k}) \maps (L,l_{k}) \to (L',l'_{k})$ of
$L_{\infty}$-algebras is a \textbf{$\mathbf{L_{\infty}}$-quasi-isomorphism}
iff the morphism of complexes
\[
f_{1} \maps (L,l_{1}) \to (L',l'_{1}) 
\]
induces an isomorphism on the cohomology:
\[
H^{\bullet} (f_{1}) \maps H^{\bullet}(L) \xto{\cong} H^{\bullet}(L').
\]
\end{defi}
\begin{remark} \label{quasi-remark}
$L_{\infty}$-quasi-isomorphisms induce an
equivalence relation on the category of $L_{\infty}$-algebras. 
Indeed, such a morphism between two $L_{\infty}$-algebras exists if and
only if the algebras are homotopy equivalent (e.g., as objects
of a simplicial category
%\cite{DolgErr} 
\cite[Sec.\ 3.2] {DHR:2015}).
Roughly speaking, the situation here is analogous to the Whitehead theorem for weak
homotopy equivalences between CW complexes.
\end{remark}

\subsection{Morphisms from Lie algebras to \boldmath $L_{\infty}$-algebras}\label{secP}
In this paper, we will be particularly interested in
$L_{\infty}$-morphisms from a Lie algebra to a
Lie $n$-algebra $(L',l'_{k})$ with the following property:
\begin{equation}\label{property}
\forall k \geq 2 \quad  l'_{k}(x_{1},\ldots,x_{k}) = 0 \quad
\text{whenever $\sum_{i=1}^{k} \deg{x_{i}}  <  0.$} \tag{P}
\end{equation}

The following characterization is proven in the appendix (Cor.\ \ref{Lie_alg_P_cor}).
\begin{prop}\label{Lie_alg_P_prop1}
If $(\g,[\cdot,\cdot])$ is a Lie algebra and $(L',l'_{k})$ is a Lie $n$-algebra
satisfying property \eqref{property}, then 
%a collection of 
the   graded
skew-symmetric maps
\[
f_{k} \maps \g^{\tensor k} \to L', \quad \deg{f_{k}} = 1-k, \quad 1
\leq k \leq n
\]
are the components of an $L_{\infty}$-morphism $\g \to L'$  if and
only if $\forall x_{i} \in \g$
\begin{multline} \label{cor_eq1}
\sum_{1 \leq i < j \leq k}
(-1)^{i+j+1}f_{k-1}([x_{i},x_{j}],x_{1},\ldots,\widehat{x_{i}},\ldots,\widehat{x_{j}},\ldots,x_{k})\\
=l'_{1} f_{k}(x_{1},\ldots,x_{k}) + l'_{k}(f_{1}(x_{1}),\ldots,f_{1}(x_{k})).
\end{multline}
for $2 \leq k \leq n$ and
\begin{multline} \label {cor_eq2}
\sum_{1 \leq i < j \leq n+1}
(-1)^{i+j+1}f_{n}([x_{i},x_{j}],x_{1},\ldots,\widehat{x_{i}},\ldots,\widehat{x_{j}},\ldots,x_{n+1})
=l'_{n+1}(f_{1}(x_{1}),\ldots,f_{1}(x_{n+1})).
\end{multline}
\end{prop}

\section{$L_{\infty}$-algebras from closed differential forms} \label{nplectic_sec}
Here  we recall various definitions and results about closed differential forms from
previous work \cite{RogersL} \cite{HDirac}, and introduce the
Lie $n$-algebra of observables associated to any manifold equipped with a
closed $(n+1)$-form. 

\begin{defi}
\label{n-plectic_def}
An $(n+1)$-form $\omega$ on a smooth manifold $M$ is 
{\boldmath $n$}{\bf-plectic}, or more specifically
an {\boldmath $n$}{\bf-plectic structure}, if it is both closed:
\[
    d\omega=0,
\]
and non-degenerate:
\[
\forall x \in M  ~  \forall v \in T_{x}M,\ \iota_{v} \omega =0 \Rightarrow v =0.
\]
If $\omega$ is an $n$-plectic form on $M$, then we call the pair $(M,\omega)$ 
an {\boldmath $n$}{\bf-plectic manifold}.
More generally, if $\omega$ is closed, but not necessarily
non-degenerate, then we call $(M,\omega)$ a {\bf pre-}{\boldmath $n$}{\bf-plectic manifold}
\end{defi}
Obviously, a (pre-) 1-plectic manifold is a (pre-) symplectic manifold.

\begin{defi} \label{hamiltonian}
Let $(M,\omega)$ be a pre-$n$-plectic manifold.  An $(n-1)$-form $\alpha$
is {\bf Hamiltonian} iff there exists a vector field $v_\alpha \in \X(M)$ such that
\[
d\alpha= -\ip{\alpha} \omega.
\]
We say $v_\alpha$ is a {\bf Hamiltonian vector field} corresponding to $\alpha$. 
The set of Hamiltonian $(n-1)$-forms and the set of Hamiltonian vector
fields on an pre-$n$-plectic manifold are both vector spaces and are denoted
as $\ham{n-1}$ and $\Xham(M)$, respectively. Note that if $\omega$ is
$n$-plectic, then associated to every Hamiltonian form is a unique
Hamiltonian vector field.
%\nrt{If $\omega$ is $n$-plectic there is a well defined map $\pi : \ham{n-1} \rightarrow \Xham(M)$.}
\end{defi}

\begin{defi}\label{loc_ham}
A vector field $v$ on a pre-$n$-plectic manifold $(M,\omega)$ is a {\bf
  local Hamiltonian vector field} iff
\[
\L_{v}\omega =0.
\]
The set of local Hamiltonian vector fields is a vector space and is
denoted as $\Xlham(M)$.
\end{defi}

\begin{defi}
\label{bracket_def}
Let $(M,\omega)$ be a pre-$n$-plectic manifold. Given $\alpha,\beta\in
\ham{n-1}$, the {\bf bracket} $\brac{\alpha}{\beta}$ is the
$(n-1)$-form given by
\[  \brac{\alpha}{\beta} = \iota_{v_{\beta}}\iota_{v_{\alpha}}\omega,\]
where $v_{\alpha}$ and $v_{\beta}$ are any Hamiltonian vector fields for $\alpha$ and $\beta$ respectively.
\end{defi}
%The bracket is well-defined even if $\omega$ is degenerate. 
 The bracket is well-defined, for if
both $v_{\alpha}$ and $v'_{\alpha}$ are Hamiltonian for $\alpha \in
\ham{n-1}$, then both $\iota_{v_{\beta}}\iota_{v_{\alpha}}\omega$ and
$\iota_{v_{\beta}}\iota_{v'_{\alpha}}\omega$ are equal to
$-\iota_{v_{\beta}} d\alpha$.

\begin{prop}\label{brac_prop}
If $(M,\omega)$ is a pre-$n$-plectic manifold and $v_{1},v_{2} \in
\Xlham(M)$ are local Hamiltonian vector fields, then $[v_{1},v_{2}]$
is a global Hamiltonian vector field with
\[
d\iota(v_{1} \wedge v_{2}) \omega = -\iota_{[v_{1},v_{2}]} \omega,
\]
and $\Xlham(M)$ and $\Xham(M)$ are Lie subalgebras of $\X(M)$.
\end{prop}
\begin{proof}
if $v_{1},v_{2}$ are
locally Hamiltonian, then by Eq.\  \eqref{commutator},
\[
\L_{v_{1}} \iota_{v_{2}} \omega = \iota_{[v_{1},v_{2}]} \omega.
\]
On the other hand, by Eq.\  \eqref{Lie},
\[
\L_{v_{1}} \iota_{v_{2}} \omega = \iota_{v_{1}} d\iota_{v_{2}} \omega
+ d \iota_{v_{1}} \iota_{v_{2}} \omega. 
\]
But $\iota_{v_{1}} d\iota_{v_{2}} \omega=0$, since $d
\iota_{v_{2}}=\L_{v_{2}} - \iota_{v_{2}}d$.  
\end{proof}
Prop.\ \ref{brac_prop} implies in particular that 
if $v_\alpha$ and $v_\beta$ are Hamiltonian vector fields for $\alpha $ and $\beta$ respectively, then $[v_\alpha,v_\beta]$ is a Hamiltonian vector field for $\{\alpha,\beta\}$.

The next theorem gives a natural $L_{\infty}$-structure on differential
forms, which extends the bracket $\{\cdot,\cdot\}$ on
$\ham{n-1}$. The theorem is essentially Thm.\ 5.2 in \cite{RogersL}, together with its
generalization Thm.\ 6.7 in \cite{HDirac}.
\begin{thm} \label{ham-infty}
Given a  pre-$n$-plectic manifold $(M,\omega)$, there is a Lie $n$-algebra
$L_{\infty}(M,\omega)=(L,\{l_{k} \})$ with underlying graded vector space 
\[
L^{i} =
\begin{cases}
\ham{n-1} & i=0,\\
\Omega^{n-1+i}(M) & 1-n \leq i < 0,
\end{cases}
\]
and maps  $\left \{l_{k} \maps L^{\tensor k} \to L| 1
  \leq k < \infty \right\}$ defined as
\[ 
l_{1}(\alpha)=d\alpha,
\]
if $\deg{\alpha} < 0$ and
\[
l_{k}(\alphadk{k}) =
\begin{cases}
\vs(k) \iota(\vk{k}) \omega  & \text{if  $\deg{\alphak{k}}=0$},\\
0 & \text{if $\deg{\alphak{k}} < 0$},
\end{cases}
\]
for $k>1$, where $v_{\alpha_{i}}$ is any Hamiltonian vector field
associated to $\alpha_{i} \in \ham{n-1}$.
\end{thm}
\subsection{The Lie {\boldmath $n$}-algebra of observables}
Note that for the $n=1$ case, the underlying complex of
$L_{\infty}(M,\omega)$ is just the vector
space of Hamiltonian functions $\cinf(M)_{\mathrm{Ham}} \subseteq
\cinf(M)$. The only non-trivial bracket is $l_{2}=\{\cdot,\cdot\}$,
which is a Lie bracket. Hence, we recover the
the underlying Lie algebra of the usual Poisson algebra associated to a
pre-symplectic manifold. 
%We call this the \textbf{Poisson Lie algebra}.
In the symplectic case,
$\cinf(M)_{\mathrm{Ham}} = \cinf(M)$, and there is a well-defined
surjective Lie algebra morphism
\[
\pi \maps \cinf(M) \epi \Xham(M)
\]
sending a function to its unique Hamiltonian vector field. If $M$ is
connected, then we see $\pi$ fits in the short exact sequence
\begin{equation} \label{KS_extension}
0 \to \R \to \cinf(M) \xto{\pi} \Xham(M) \to 0.
\end{equation}
This is the Kostant-Souriau central extension \cite{Kostant:1970,Souriau:1967}. It characterizes the
underlying Lie algebra of $\cinf(M)$, up to isomorphism, as the unique central extension determined by the
symplectic form (evaluated at a point $p \in M$). 

For the pre-symplectic case, Hamiltonian functions can have more than
one corresponding Hamiltonian vector field, and so a map $\cinf(M)_{\mathrm{Ham}} \to
\Xham(M)$ may not exist. Therefore one instead considers the Lie
algebra
\begin{align*}
\widetilde{\cinf(M)_{\mathrm{Ham}}}&= \{ (v,f) \in \Xham(M) \oplus \cinf(M)_{\mathrm{Ham}} ~ \vert ~ df =
-\iota_{v}\omega \} \\
 [(v_1,f_1),(v_2,f_2)]_{L} &=\bigl([v_{1},v_{2}], \brac{f_{1}}{f_{2}} \bigr).
\end{align*} 
The projection $(v,f) \mapsto v$ then gives a central
extension 
\begin{equation}\label{eq: sespre}
0 \to \R \to \widetilde{\cinf(M)_{\mathrm{Ham}}} \xto{\pi} \Xham(M) \to 0
\end{equation}
which generalizes \eqref{KS_extension} to any connected pre-symplectic
manifold \cite[Prop 2.3]{Brylinski_book}. If $(M,\omega)$ is
symplectic, then it is easy to see that
$\widetilde{\cinf(M)_{\mathrm{Ham}}}$ is isomorphic to
$\cinf(M)_{\mathrm{Ham}}=\cinf(M)$ as Lie algebras.

The higher analog of the central extension \eqref{eq: sespre} for a pre-$n$-plectic
manifold is obtained by slightly modifying the construction of $L_{\infty}(M,\omega)$.
\begin{thm} \label{poisson_thm}
Given a  pre-$n$-plectic manifold $(M,\omega)$, there is a Lie $n$-algebra
$\poi(M,\omega)$ with underlying graded vector space 
\begin{equation*}
\begin{split}
L^{0} & = \widetilde{\ham{n-1}}=\{ (v,  \alpha) \in \Xham(M) \oplus \ham{n-1}  ~ \vert ~ d
\alpha = -\iota_{v}\omega \} \\
L^{i} & = \Omega^{n-1+i}(M) \quad 1-n \leq i < 0,
\end{split}
\end{equation*}
and structure maps:
\begin{equation*}
\begin{split}
\tilde{l}_{1}(\alpha)&=
\begin{cases}
(0, d\alpha) & \text{if $\deg{\alpha}=-1$},\\
d \alpha & \text{if $\deg{\alpha} < -1$,}
\end{cases}\\
\tilde{l}_{2}(x_{1},x_{2}) &= 
\begin{cases}
\bigl( [v_{1},v_{2}], \iota(v_{1} \wedge v_{2}) \omega \bigr)
=\bigl( [v_{1},v_{2}], \{\alpha_1,\alpha_2\} \bigr) &
\text{if $\deg{x_{1} \tensor x_{2}}  = 0$,}\\
 0  & \text{otherwise},
\end{cases}
\end{split}
\end{equation*}
and, for $k > 2$: 
\[
\tilde{l}_{k}(x_{1},\ldots,x_{k}) = 
\begin{cases}
\vs(k)   \iota(\vk{k}) \omega & {\text{if $\deg{ x_{1}\tensor \cdots \tensor x_{k}} <0$}}, \\
0 & \text{otherwise}.
\end{cases}
\]
\end{thm}
\begin{proof} 
%\nrt{Notice that $l_2$ is well-defined by Prop. \ref{brac_prop}.}
Eq.\  \eqref{gen_jacobi} is satisfied since the bracket of vector fields
satisfies the Jacobi identity while the higher structure maps are identical to those of $L_{\infty}(M,\omega)$.
\end{proof}
We call $\poi(M,\omega)$ the \textbf{Lie {\boldmath $n$}-algebra of observables}
associated to $(M,\omega)$. 
\begin{prop}\label{map}
Let $(M,\omega)$ be a pre-$n$-plectic manifold. 
\begin{enumerate}
\item{
The cochain map
\begin{equation*}
\pi \maps \poi(M,\omega) \epi \Xham(M) \label{pi_map}
\end{equation*}
defined to be the projection $(v,\alpha) \mapsto v$ in degree 0, and
trivial in all lower degrees lifts to a strict morphism of $L_{\infty}$-algebras.
}
%\nrt{\item{
%If $\omega$ is non degenerate, the cochain map
%\[
%\pi \maps L_{\infty}(M,\omega) \epi \Xham(M)
%\]
%defined in Def. \ref{hamiltonian} lifts to a strict morphism of $L_{\infty}$-algebras.
%}}

\item{
If $(M,\omega)$ is $n$-plectic, then the cochain map
\[
\xymatrix{
\cinf(M) \ar[r]^-{d} \ar[d]^{\id} & \Omega^{1}(M) \ar[d]^{\id} \ar[r]^-{d} & \cdots \ar[r]^-{d}
& \Omega^{n-2}(M) \ar[d]^{\id} \ar[r]^-{d} & \ham{n-1} \ar[d]^{\phi} \\
\cinf(M) \ar[r]^-{d} & \Omega^{1}(M) \ar[r]^-{d} & \cdots \ar[r]^-{d} &
\Omega^{n-2}(M) \ar[r]^-{0 \oplus d} & \widetilde{\ham{n-1}} 
}
\]
with $\phi(\alpha)=(v_{\alpha},\alpha)$, where $v_{\alpha}$ is the
unique Hamiltonian vector field associated to $\alpha$, lifts to a
strict $L_{\infty}$-quasi-isomorphism $L_{\infty}(M,\omega) \xto{\sim}
\poi(M,\omega)$.
}
\end{enumerate}
\end{prop}
\begin{proof}
To prove (1), note we have
\[
\pi \tilde{l}_{2} \bigl( (v_{1},\alpha_{1}), (v_{2},\alpha_{2}) \bigr) = [ \pi
(v_{1},\alpha_{1}), \pi  (v_{1},\alpha_{1})].
\]
Hence, Eq.\  \eqref{strict_def_eq} is satisfied. 
%\nrt{Similarly for (2)}. 
For (2), note that
we have the equalities
\begin{equation*}
\begin{split}
\tilde{l}_{2} \phi^{\tensor 2} &= \phi l_{2} \\
\tilde{l}_{k} \phi^{\tensor k} &= l_{k} \quad \forall k > 2.
\end{split}
\end{equation*}
%where $\tilde{l}_{k}$ denotes the structure maps for $\poi(M,\omega)$ 
Hence, the cochain map lifts to a strict $L_{\infty}$-morphism. The
map $\phi$ is an isomorphism, hence the corresponding
$L_{\infty}$-morphism is a quasi-isomorphism.
\end{proof}

\begin{remark} \label{inf_sym_remark}
There is a nice conceptual
interpretation of the Lie $n$-algebra $\poi(M,\omega)$ within the
context of differential cohomology whenever $\omega$ represents an
integral cohomology class i.e.\
\[
[\omega] \in \im \bigl( H^{n+1}(M,\Z) \to H^{n+1}(M,\R) \bigr).
\]
For any manifold $M$ there is a short exact sequence
\[
0 \to H^{n}(M,\U(1)) \to H^{n}_{\mathrm{Del}}(M) \xto{\curv}
\Omega^{n+1}_{\mathrm{cl/int}}(M) \to 0,
\]
where $H^{\bullet}(M,\U(1))$ is ordinary $\U(1)$-valued cohomology, 
$\Omega^{n+1}_{\mathrm{cl/int}}(M)$ is the group of closed and
integral $(n+1)$-forms, and $H^{\bullet}_{\mathrm{Del}}(M)$ is `smooth
Deligne cohomology' \cite[Sec.\ 1.5]{Brylinski_book}. The group 
$H^{n}_{\mathrm{Del}}(M)$ classifies certain higher geometric objects which
one could call `principal $\U(1)$ $n$-bundles'
equipped with an `$n$-connection'. The curvature of such an
$n$-connection is given by the surjection in the above sequence, and
therefore is an integral pre-$n$-plectic form on $M$.

For example, $H^{1}_{\mathrm{Del}}(M)$ is in bijection with isomorphism classes of
principal $\U(1)$ bundles with connection over $M$. The surjection in the above sequence sends such a
bundle to its curvature 2-form. The group $H^{2}_{\mathrm{Del}}(M)$
classifies geometric objects called $\U(1)$-gerbes equipped a
`2-connection' (also called a `connective structure' and its
`curving'), whose curvature is a closed integral
3-form on $M$.

% Kostant \cite{Kostant:1970} showed that the Lie algebra of
% infinitesimal symmetries of a principal $\U(1)$-bundle which preserve
% a fixed connection is isomorphic to the Lie algebra that we call
% $\poi(M,\omega)$, when $(M,\omega)$ is pre-symplectic. 
% This fact plays an essential role in geometric
% quantization. 

It is known that the infinitesimal symmetries of gerbes form a Lie
2-algebra (for example, see \cite{Collier:2011}), and this pattern
continues for $n >2$. In \cite{FRS}, it is shown for any $n >1$
that the Lie $n$-algebra of connection-preserving infinitesimal
symmetries of a principal $\U(1)$ $n$-bundle whose curvature is
$\omega$ is quasi-isomorphic to the Lie $n$-algebra $\poi(M,\omega)$.
  \end{remark}

\section{Homotopy moment maps}\label{sec:momaps}
In symplectic geometry, a moment map $M \to \g^{\vee}$ can be equivalently expressed as
a co-moment map i.e.\ a Lie algebra morphism $\g \to \cinf(M)$.
%Let us recall that a moment map $M\to \g^*$ can be equivalently expressed as a co-moment map $\mu : \mathfra%k{g}\to(C^{\infty}(M),\{,\})$. 
In this section, we describe the natural $L_{\infty}$ analog of this
co-moment map. We call this a homotopy moment map.

%we give the definition of a 
%It is the natural $L_{\infty}$ analog of the co-moment map considered
%in symplectic geometry. 

%Let us recall that a moment map $\mu$ is an equivariant  $\mathfrak{g}^*$-valued equivariant smooth function on a symplectic manifold $(M,\omega)$ equipped with an action, i.e. $\mu\in (\mathfrak{g}^*\otimes \mathcal{C}^{\infty} (M)})^G,$  satisfying some compatibility condition with $\omega$. This notion can be dually expressed as a co-moment map, i.e. as a Lie algebra map $\mu : \mathfrak{g}\to(\mathcal{C}^{\infty}(M),\{,\})$ from the Lie algebra acting, to the Poisson algebra associated to the symplectic form. 

\begin{defprop} \label{main_def}
Let $G$ be a Lie group with Lie algebra $\g$. Let $(M,\omega)$ be a pre-$n$-plectic manifold equipped with a
$G$-action which preserves $\omega$, and such that the infinitesimal
$\g$-action $x \mapsto v_{x}$ is via Hamiltonian vector fields.
A {\bf homotopy moment map} (or {\bf moment map} for short)
is a lift 
\begin{equation} \label{the_lift}
\xymatrix{
&&  \poi(M,\omega) \ar[d]^{\pi} \\
\g \ar @{-->}[urr] \ar[rr]^{v_{-}} && \Xham(M).
}
\end{equation}
of the Lie algebra morphism $v_{-}$ through the
$L_{\infty}$-morphism $\pi$ (defined in Prop.\ \ref{map}) in the category of $L_{\infty}$-algebras. 

Such a lift corresponds to an $L_{\infty}$-morphism 
\[
(f_{k}) \maps  \g \to L_{\infty}(M,\omega)
\]
such that \[
-\iota_{v_x} \omega=d(f_{1}(x))\quad\quad \text{   for all $x\in \g$}.
\]
\end{defprop}
Before we give a proof of the above correspondence between lifts and
morphisms into $L_{\infty}(M,\omega)$, we explain
in more detail what a  homotopy moment map actually is:
\begin{itemize}
\item{The condition $-\iota_{v_x} \omega=d(f_{1}(x))$ implies
    that $v_{x}$  
is a Hamiltonian vector field for $f_{1}(x) \in \ham{n-1}$.
}
\item{
By Prop.\ \ref{Lie_alg_P_prop1}, an $\li$-morphism $\g \to L_{\infty}(M,\omega)$ consists of
  a collection of $n$ skew-symmetric maps
\[
f_{k} \maps \g^{\tensor k} \to L, \quad 1 \leq k \leq n, 
\]
where $L$ is the underlying vector space of $L_{\infty}(M,\omega)$
and $\deg{f_{k}} =1-k$, which satisfy
\begin{multline} \label{main_eq_1}
\sum_{1 \leq i < j \leq k}
(-1)^{i+j+1}f_{k-1}([x_{i},x_{j}],x_{1},\ldots,\widehat{x_{i}},\ldots,\widehat{x_{j}},\ldots,x_{k})\\
=df_{k}(x_{1},\ldots,x_{k}) + \vs(k)\iota(v_{1}\wedge \cdots \wedge v_{k})\omega
\end{multline}
for $2 \leq k  \leq n$ and
\begin{multline} \label{main_eq_2}
\sum_{1 \leq i < j \leq n+1}
(-1)^{i+j+1}f_{n}([x_{i},x_{j}],x_{1},\ldots,\widehat{x_{i}},\ldots,\widehat{x_{j}},\ldots,x_{n+1})
=\vs(n+1)\iota(v_{1}\wedge \cdots \wedge v_{{n+1}})\omega,
\end{multline}
where $v_{i}$ is the vector field associated to $x_{i}$ via the $\g$-action.
(These equalities are obtained from Eqs.\ \eqref{cor_eq1} and \eqref{cor_eq2}
via substitution using the definition of $f_{1}$ and the maps $l_{k}$
defined in Thm.\ \ref{ham-infty}. Notice that 
 Thm.\ \ref{ham-infty} implies that $L_{\infty}(M,\omega)$ has
property \eqref{property}.)
}
\item{Finally, note that Prop.\ \ref{brac_prop} imply that $v_{[x,y]}= [v_{x},v_{y}]$ is a
Hamiltonian vector field for 
\[
\brac{f_{1}(x)}{f_{1}(y)}=l_{2}(f_{1}(x),f_{1}(y)).
\]
Of course, $f_{1} \maps \g \to \ham{n-1}$ need not preserve the
bracket on $\g$ i.e.\ in general
$
f_{1}([x,y]) \neq \brac{f_{1}(x)}{f_{1}(y)}.
$
This is a good property, in view of the facts that the Lie bracket of $\g$ satisfies the Jacobi identity but $\brac{\cdot}{\cdot}$ does not.
}
\end{itemize}

\begin{proof}[Proof of Def./Prop.\ \ref{main_def}]
% The proof of the correspondence essentially follows from how we constructed
% the Lie $n$-algebra $\poi(M,\omega)$  (Thm.\ \ref{poisson_thm})
% from the Lie algebra $\Xham(M)$ (concentrated in degree 0) and the Lie $n$-algebra
% $L_{\infty}(M,\omega)$ (Thm.\ \ref{ham-infty}). Indeed, 
Suppose we have an $L_{\infty}$-morphism
\[
(\tilde{f}_{k}) \maps  \g \to \poi(M,\omega)
\]
corresponding to a lift \eqref{the_lift}. Note that $\poi(M,\omega)$
satisfies Property \eqref{property}. Therefore the morphisms 
$(\tilde{f}_{k})$ are trivial for $k \geq n+1$ and satisfy the compatibility
equations given in Prop.\ \ref{Lie_alg_P_prop1}.
By the definition of the projection $\pi$,  the degree 0 map can be written as
\[
\tilde{f}_{1}(x)=\bigl(v_{x},f_{1}(x) \bigr) \in \widetilde{\ham{n-1}} 
\]
where $f_{1} \maps \g \to \ham{n-1}$ is a linear map satisfying 
$-\iota_{v_x} \omega=d(f_{1}(x))$. Moreover, $\pi$ is a strict
$L_{\infty}$-morphism and therefore
\[
\tilde{f}_{1}([x,y]) = \bigl(v_{[x,y]},f_{1}([x,y]) \bigr) = \bigl([v_{x},v_{y}],f_{1}([x,y]) \bigr).
\] 
Combining these observations with the fact that structure maps $\tilde{l}_{k}$ of
$\poi(M,\omega)$ agree with those of $L_{\infty}(M,\omega)$ for $k
\geq 3$, we obtain an $L_{\infty}$ morphism
\[
(f_{k}) \maps \g \to L_{\infty}(M,\omega)
\] 
with $f_{k} = \tilde{f}_{k}$ for $k >1$. The fact that every such
morphism gives a lift is now obvious.
\end{proof}

\begin{remark}
Note we have a generalization of the fact that, in symplectic
geometry, the image of a moment map is a Lie subalgebra of the Poisson
algebra of functions on the symplectic manifold. 

More precisely,
given a moment map with components
$f_{k} \maps \g^{\tensor k} \to L$ (for  $1 \leq k \leq n$),
%$(f_k)\colon \g \to L'$ as in Prop. \ref{Lie_alg_P_prop1}, 
its image is not an $\li$-subalgebra of $L_{\infty}(M,\omega)$ in general (unless $n=1$). However, the subcomplex generated by its image,
which we denote by $I$ and which is given by
\[
I^{k} =
\begin{cases}
\im(f_n) & k=-n+1\\
\im(f_{-k+1})+d(\im(f_{-k+2}))  & -n+2\le k \le 0
\end{cases}
\]
is an $\li$-subalgebra of $L_{\infty}(M,\omega)$. Indeed $I$ is closed w.r.t. the differential $d$ by construction. It is seen to be closed w.r.t. the higher brackets using the definition of the latter (Thm.\ \ref{ham-infty}) together with Eq.\  \eqref{main_eq_1} and \eqref{main_eq_2}, and because the Hamiltonian vector field of an exact element of $\ham{n-1}$  is zero.

% from Prop. \ref{Lie_alg_P_prop1}. 
\end{remark}

\section{Equivariant cohomology} \label{equi_deRham_sec} In this
section, we establish the relationship between equivariant cohomology
and moment maps. We interpret the defining equations for a moment map
within the context of the Bott-Shulman-Stasheff de Rham model for
equivariant cohomology. This model is related to the more
computationally tractable Cartan model via a natural chain map (a
generalization of the
``Cartan'' map), which is a quasi-isomorphism if the group acting
is compact. In particular, given any $(n+1)$-cocycle in the Cartan model
extending a $G$-invariant pre-$n$-plectic form, we obtain an explicit
formula for a corresponding moment map. This correspondence is natural
in the sense that the moment map for the pullback of a cocycle along a
$G$-equivariant map is the pullback of the moment map. As a special
case, this formula recovers the well known relationship between moment maps in
symplectic geometry and 2-cocycles in the Cartan model.

\subsection{The Bott-Shulman-Stasheff model}
Let $G$ be a Lie group and $M$ a $G$-manifold.
Let $G \ltimes M_{\bu}$ denote the simplicial manifold 
\[
G \ltimes M_{n}= G^{n} \times M
\]
with the face maps $d_i\colon G^n\times M \to G^{n-1} \times M$ given by 
\begin{equation*}\label{eq:faceGbuM}(g_1,\ldots,g_n,p)\mapsto \begin{cases} (g_2,\ldots,g_n,p) & i=0,\\ (g_1,\ldots,g_ig_{i+1},\ldots,g_n,p) & 0<i<n,\\ (g_1,\ldots, g_{n-1},g_n p) & i=n.\end{cases} \end{equation*}
The Bott-Shulman-Stasheff complex is the total complex of the double
complex of differential forms on $G \ltimes M_\bu$: 
\begin{equation}\begin{split}\label{eq:bsComplex}
 \Omega^{j,k}(G \ltimes M_\bu)&:=\Omega^k(G^j\times M),\\
 \bOmega(G \ltimes M_\bu)&:=\Tot(\Omega^{*,*}(G \ltimes M_\bu)), \\
 D&:= \del + (-1)^j d, 
\end{split}
\end{equation}
where $\del$ is the simplicial differential and $d$ is the de Rham differential. 
% Note that the map 
% \begin{equation*}\begin{split}\label{eq:isoBG}G^{n+1}\times M&\to G^n\times M, \\ (g_0,\ldots,g_n,x)&\mapsto (g_0 g_1^{-1},\ldots g_{n-1}g_n^{-1},g_n x)\end{split}\end{equation*}
% induces an isomorphism of simplicial manifolds \[E_\bu G\times_G M\cong G^\bu \times M,\] where $E_\bu\times_G M$ is the quotient of $E_{\bu}G\times M$ by the diagonal $G$-action
% \begin{equation}\begin{split}\label{eq:actionEGM}
% G\times E_nG\times M&\to M, \\
% (h, g_0,\ldots, g_n,x)&\mapsto (g_0h^{-1},\ldots, g_nh^{-1}, hx).
% \end{split}\end{equation}
The de Rham theorem of Bott-Shulman-Stasheff \cite{BSS:1976} implies
that the cohomology of $\bigl(\bOmega(G\ltimes M_{\bu}),D \bigr)$ is the
equivariant cohomology of $M$ with real coefficients.

\subsection{Moment maps from equivariant cohomology}
We assume that $G$ is finite-dimensional, and
denote by $\g^{\vee}$ the dual of the Lie algebra $\g$.
% Recall that the Chevalley-Eilenberg differential $\delta_{\CE}$ on $\Lambda^*(\g^\vee)$ is
% \begin{equation} \label{eq:CE_diff}
% \begin{split}
%  \delta_{\CE}&\colon\Lambda^n(\g^\vee)\to \Lambda^{n+1}(\g^\vee), \\
% \delta_{\CE}(c)(x_1,\ldots,x_{n+1})&:= \sum_{1\leq i < j\leq n+1} (-1)^{i+j} c([x_i,x_j],x_1,\ldots,\hat{x}_i,\ldots,\hat{x}_j,\ldots,x_{n+1}).
% \end{split}\end{equation}
% If $M$ is a $G$-manifold, then
% let $ \iota^{k}_{\g} \maps \Omega^{\ast}(M) \to \Lambda^{k}(\g^{\vee})
% \tensor \Omega^{\ast -k}(M)$ be the insertion operations
% \begin{equation*}\begin{split}
%  \iota_{\g}^k\colon\Omega^*(M) &\to \Lambda^k(\g^\vee) \tensor
%  \Omega^{\ast -k}(M), \\ 
%  \iota^k_{\g}\alpha(x_1,\ldots,x_k)&=\iota_{v_{x_k}}\cdots\iota_{v_{x_1}}\alpha\in\Omega^{\ast-k}(M)
% \end{split}\end{equation*}
% where $\alpha\in\Omega^\ast(M)$, and  $x_1,\ldots,x_k\in\g$.
% We will also use $\iota_{\g}$ to denote $\iota_\g^1$.
If $(M,\omega)$ is a pre-$n$-plectic manifold equipped with a $G$-action preserving $\omega$, then 
it follows directly from Def./Prop.\ \ref{main_def} that a moment map corresponds
to a sum $f = \sum_{k=1}^{n} f_k$ with
\[
f_{k} \in \Lambda^{k}(\g^{\vee}) \tensor \Omega^{n-k}(M)
\]
such that
\[
(\delta_{\CE} \tensor \id) f + (\id \tensor d)f = - \sum_{k=1}^{n+1}\vs(k)\iota_{\g}^{k} \omega.
\]
Above, $ \delta_{\CE} \maps \Lambda^n(\g^\vee)\to
\Lambda^{n+1}(\g^\vee)$ is the Chevalley-Eilenberg differential
defined in Eq.\ \eqref{eq:CE_diff}, and
$ \iota^{k}_{\g} \maps \Omega^{\ast}(M) \to \Lambda^{k}(\g^{\vee})
\tensor \Omega^{\ast -k}(M)$ is the insertion operation defined in
Eq.\ \eqref{eq:iota_g}.
At this point, it is convenient to define the following complex:
\begin{equation} \label{eq:dbl_complex}
\begin{split}
 C^{k,m}_{\g}(M) &:=\Lambda^k(\g^\vee)\otimes\Omega^m(M)\\
 C^{\ast}_{\g}(M)&:=\Tot(C^{*,*}_{\g}(M) )\\
\bd &:= \delta_{\CE} + (-1)^{k}d.\\
\end{split}
\end{equation}
Given a vector in the total complex $\alpha \in  C^{i}_{\g}(M)$, we 
denote by $\alpha_{k}$ the summand of $\alpha$ belonging to 
$\Lambda^{k} (\g^{\vee}) \tensor \Omega^{i-k}(M)$.

The complex \eqref{eq:dbl_complex} leads us to a proposition whose proof is straightforward.
\begin{prop}\label{prop:dbl_complex}
If $(M,\omega)$ is a pre-$n$-plectic manifold equipped with a $G$-action 
preserving $\omega$, then $f = \sum_{k=1}^{n} f_k \in C^{n}_{\g}(M)$
is a moment map if and only if
\begin{equation} \label{eq:prop:dbl_complex}
\bd f^\vs = \sum_{k=1}^{n+1} (-1)^{k+1}\iota^{k}_{\g}\omega 
\end{equation}
where $f^\vs = \sum_{i=1}^{n} f^\vs_k$ with $f^{\vs}_k = \vs(k)f_k$. 
\end{prop}

\begin{remark}
The homotopy theory of $L_{\infty}$-algebras provides
a nice conceptual explanation for the appearance of the complex
$C^{\ast}_{\g}(M)$. The multilinear maps $ v_1 \wedge v_2 \wedge
\cdots \wedge v_k \mapsto \vs(k) \iota_{v_k}\iota_{v_{k-1}} \cdots
\iota_{v_{1}} \omega$ define an $L_{\infty}$-morphism $F^{\omega}$ from the Lie
algebra of Hamiltonian vector fields $\Xham(M)$ to an abelian
$L_{\infty}$-algebra $B^nA$ corresponding to the cochain complex $\Omega^{0}(M) \xto{d}
\Omega^{1} \xto{d} \cdots \xto{d} \Omega^{n-1}(M) \xto{d}
d\Omega^{n-1}(M)$, where $d \Omega^{n-1}(M)$ is in degree 0
\cite[Prop.\ 3.8]{FRS}. (The notation $B^nA$ is used to remind the
reader of the analogous situation for classifying spaces of bundles.)
The homotopy fiber of $F^{\omega}$ is the Lie
$n$-algebra of observables $\poi(M,\omega)$, i.e.\ there is a homotopy
pullback diagram 
\[
\begin{tikzpicture}[descr/.style={fill=white,inner sep=2.5pt},baseline=(current  bounding  box.center)]
\matrix (m) [matrix of math nodes, row sep=3em,column sep=4em,
  ampersand replacement=\&]
  {  
\poi(M,\omega) \& 0 \\
\Xham(M) \& B^nA \\
}; 
\path[->,font=\scriptsize] 
 (m-1-1) edge node[auto] {$ $} (m-1-2)
 (m-1-1) edge node[auto,swap] {$\pi $} (m-2-1)
 (m-1-2) edge node[auto] {$$} (m-2-2)
 (m-2-1) edge node[auto] {$F^{\omega}$} (m-2-2)
;
%begin pullback symbol%
\begin{scope}[shift=($(m-1-1)!.30!(m-2-2)$)]
\draw +(-.25,0) -- +(0,0)  -- +(0,.25);
\node(h) at (.15,-.15) {\scriptsize{$h$}};
\end{scope}
%end pullback symbol%
\end{tikzpicture}
\]
As in Def.\ \ref{main_def}, we have the Lie algebra morphism $v_{-}
\maps \g \to \Xham(M)$ encoding the infinitesimal action. The
structure maps of the composition $F^{\omega} \circ v_{-}$ are
precisely the multilinear maps $\iota^{k}_{\g}\omega$. 
There is, in general, a ``mapping space'' for $L_{\infty}$-algebras, i.e., a simplicial
set $\Map_{\bullet}(L,L')$ whose 0-simplicies are
$L_{\infty}$-morphisms between $L$ and $L'$ \cite[Sec.\
3.2]{DHR:2015}. 
In the present case,  one
can show a moment map, or lift, $\g \to \poi(M,\omega)$ exists if and only if
$F^{\omega} \circ v_{-}$ is null--homotopic, i.e., $[F^{\omega} \circ
v_{-}]=[0] \in \pi_0 \Map_{\bullet}(\g,B^nA)$. This is in complete
analogy with how isomorphism classes of principal bundles correspond
to homotopy classes of maps into classifying spaces. In the present case, since $B^nA$ is an
abelian $L_{\infty}$-algebra (a chain complex), the simplicial set
$\Map_{\bullet}(\g, B^nA)$ is a simplicial vector space. Via the
Dold--Kan correspondence, the normalized chain complex of
$\Map_{\bullet}(\g, B^nA)$ is homotopy equivalent to the
complex $C^{\ast}_{\g}(M)$. Thus, a homotopy moment map exists if and only if
$F^{\omega} \circ v_{-}$ is null-homotopic if and only if it is a trivial $n$-cocycle
in $C^{\ast}_{\g}(M)$. This last statment is precisely the content of Prop.\
\ref{prop:dbl_complex}.

Applications of the complex $C^{\ast}_{\g}(M)$ as a computational tool 
are presented in \cite{FLRZ}.
\end{remark}

The relationship between moment maps and equivariant cohomology stems
from the fact that the complex $C^{\ast}_{\g}(M)$ can be identified
with the subcomplex 
\[
\Omega^{\ast}(G \times M)^{G} \subseteq \Omega^{1,\ast}(G \ltimes M_\bu)
\]
of the double complex \eqref{eq:bsComplex}. This is the subcomplex of forms on $G\times M$ invariant with
respect to the $G$-action $g' \cdot (g,p) = (g'g,p)$. We view $\Omega^{\ast}(G \times M)^{G}$
as the total complex of the double complex
\begin{equation}\begin{split}\label{eq:prodDoubleCx}
 \Omega^{k,m}(G\times M)^G:=&\Gamma(G\times M, \Lambda^k T^*G\otimes \Lambda^m T^*M)^G\subseteq \Omega^{k+m}(G\times M)^G, \\
 \Omega^*(G\times M)^G =& \Tot\bigl(\Omega^{*,*}(G\times M)^G\bigr), \\
 d =& d^G + (-1)^k d^M
\end{split}\end{equation}
where $d^G$ and $d^M$ denote the de Rham differentials in the $G$ and $M$ directions, respectively. 
\begin{lemma} \label{lem:dbl_complex1} 
Restriction to $M=\{e\}\times  M\overset{i}{\hookrightarrow} G\times
M$ induces an isomorphism of  double complexes
\[
r\colon\bigl(\Omega^{*,*}(G\times M)^G,d^G, d^M\bigr)\xto{\cong}
\bigl(C^{*,*}_\g(M),\delta_{\CE}, d \bigr), 
\]
and hence, an isomorphism of total complexes
\[  r\colon \bigl(\Omega^*(G\times M)^G,d\bigr)\xto{\cong}
\bigl(C_{\g}^*(M), \bd \bigr). \]
\end{lemma}

\begin{proof}
We use the following observation: if $E\to B$ is a $G$-equivariant
vector bundle and $i \maps M \hookrightarrow B$ an embedding such that
the action $G\times M\to B$ is a diffeomorphism, then the restriction of sections
$\Gamma(B,E)^G\to \Gamma(M, i^*E)$ is an isomorphism. 
Hence, the restriction of sections of $\Lambda^mT^*(G\times M)$ to $M =
  \{e\}\times M$ induces an isomorphism
\[
\Omega^m(G\times M)^G = \Gamma(G\times M, \Lambda^m(T^*(G\times M)))^G
\xto{\cong}\Gamma(M,i^*\Lambda^m T^*(G\times M)) =
\Gamma(M,\Lambda^m(\g^\vee\oplus T^*M)). 
\]
Composing this with the natural isomorphism
\begin{equation}\label{eq:exteriorProdOfSum}
\begin{split}
    \Psi\colon\Lambda^m(\g^\vee\oplus T^{*}_{p}M)&\xto{\cong}\bigoplus_{k+\ell=m} \Lambda^k(\g^\vee)\tensor\Lambda^\ell T^{*}_{p}M\\
    \Psi(\alpha)\bigl((x_1,\ldots,x_k)\otimes
    (w_1,\ldots,w_\ell)\bigr) &=
    \alpha\bigl((x_1,0),\ldots,(x_k,0),(0,w_1),\ldots,(0,w_\ell)\bigr)
\end{split}
\end{equation}
gives the isomorphism $r\colon\bigl(\Omega^{*,*}(G\times M)^G \bigr)\xto{\cong}
\bigl(C^{*,*}_\g(M) \bigr)$. It remains to show $r$ respects the differentials.
Indeed, since the Chevalley-Eilenberg
differential is the restriction of the de Rham differential to
left-invariant differential forms on $G$ (e.g. \cite[Chapter IV,
Prop. 3]{ccc2}), we have $r d^G = \delta_{\CE} r$. And finally, $r d^M = d r$ follows immediately from the definition of the differentials $d^M$ and $d$.
\end{proof}
Now that we identified $C_\g^*(M)$ as sitting inside
$\Omega^{1,*}(G \ltimes M_{\bu})$, we can reinterpret the term on the right hand
side of the moment map condition in Prop. \ref{prop:dbl_complex} in
terms of the Bott-Shulman-Stasheff complex.

\begin{prop}\label{prop:rDelOmega}
 Let $\omega\in\Omega^{n+1}(M)^G$. Then 
\[ 
r(\del \omega) =  \sum_{k=1}^{n+1} (-1)^{k+1}\iota_\g^k\omega.
\]
\end{prop}
\begin{proof}
  The face map $d_1\colon G\times M\to M$ is the
  $G$-action. Therefore, it is $G$-equivariant and hence
  $d_1^*\omega\in\Omega^{n+1}(G\times M)^G$. Since $d_0$ is the
  projection $G \times M \to M$, we
  also have $d_0^*\omega \in\Omega^{n+1}(G\times M)^G$,
  and hence
\[
\del\omega=d_0^*\omega - d_1^*\omega \in  \Omega^{n+1}(G\times M)^G. 
\]
The differential of the map $d_1$ at the point $(e,p)$ is given by 
\begin{equation*}
d_{1_{\ast}} \vert_{(e,p)}(x,u) = u - v_x, \quad \text{for $x \in
  T_eG$, $u \in T_pM$}.
\end{equation*}
If $x_1,\ldots,x_{n+1}\in\g$ and $u_1,\ldots,u_{n+1}\in T_pM$, then
\begin{equation*}
\begin{split}\label{eq:rdomega}
 r (d_0^*\omega - d_1^*\omega)(u_1,\ldots,u_{n+1}) =& 0, \\
r(d_0^*\omega - d_1^*\omega) (x_1,\ldots,x_k,u_1\ldots,u_{n-k+1})
=&  (-1)^{k+1} \bigl(\iota_{\g}^k\omega(x_1,\ldots,x_k) \bigr)(u_1,\ldots, u_{n-k+1}).
\end{split}
\end{equation*}
Thus, $r (\del\omega) = \sum_{k=1}^{n+1}(-1)^{k+1}\iota_\g^k\omega$.
\end{proof}

\begin{cor}\label{cor:htpyMMCond}
 An element $f=\sum_{k=1}^{n} f_k\in C_\g^{n}(M)$ with $f_k\in
 C_\g^{k,n-k}(M)$ is a homotopy moment map for the pre-$n$-plectic
 form $\omega \in \Omega^{n+1}(M)^G$ if and only if $f^\vs$ satisfies
 $\bd f^\vs = r(\partial \omega)$.
\end{cor}

If $G$ is compact and $\omega \in \Omega^{n+1}(M)$, then let
$\omega^{G} \in\Omega^{n+1}(M)^G$ denote the $G$-invariant form
obtained by averaging. Similarly, if $\beta \in\Omega^{n}(G \times
M)$, then let $\beta^G$ denote the $G$-invariant form on $G \times M$ obtained by
averaging with respect to the action $g' \cdot (g,p) = (g'g,p)$.

We now state our first main result which relates equivariant cohomology to
moment maps. In particular, if $G$ is compact, then any
cocycle in the Bott-Shulman-Stasheff complex gives a moment map.

\begin{thm}\label{thm:BSmm}
 Let $M$ be a $G$-manifold and $\alpha = \sum_{i=0}^{n+1} \alpha_{i}
 \in \Omega^{n+1}(G \ltimes M_\bu)$ a degree $(n+1)$-cocycle in the
Bott-Shulman-Stasheff complex with $\alpha_{i}
 \in \Omega^{i,n-i+1}(G \ltimes M_{\bu}) = \Omega^{n-i+1}(G^i\times M)$.
\begin{enumerate}
\item{ 
    If $\alpha_{0}\in\Omega^{n+1}(M)^G$ and $\alpha_{1}
    \in\Omega^{n}(G\times M)^G$, then $f^\vs = \sum_{k=1}^{n}r (\alpha_{1})_k$
    defines a homotopy moment map $f$ for the $G$-invariant
    pre-$n$-plectic form $\omega=\alpha_0$.
}
\vspace{.25cm}
 \item{
If $G$ is compact, then $f^\vs = \sum_{k=1}^{n}
r\bigl(\alpha_1^G\bigr)_k$ defines a homotopy moment map $f$ for the
$G$-invariant pre-$n$-plectic form $\omega=\alpha_0^G\in \Omega^{n+1}(M)^G$. 
}
\end{enumerate}
\end{thm}
\begin{proof}
The cocycle condition $D \alpha = 0$ implies that 
\begin{equation*}
d\alpha_0 = 0, \quad \del \alpha_0 =d\alpha_1.
\end{equation*}
Therefore, $\omega=\alpha_0$ is indeed a pre-$n$-plectic form. Since
$r(\partial \omega)_0 =0$, we have
$\bd r (\alpha_1) = \bd \sum_{k=1}^{n} r (\alpha_1)_{k}$. 
The first claim then follows from Cor.\ \ref{cor:htpyMMCond}.

For the second claim, note that the following equalities hold:
\[
\bd r\bigl(\alpha_1^G\bigr) = r\bigl(d \alpha_1^G\bigr)  =
r\bigl((d\alpha_1)^G\bigr) = r\bigl((\del \alpha_0)^G\bigr) =
r\bigl(\del\alpha_0^G\bigr), 
\]
and, like before, we have $\bd r\bigl(\alpha_1^G\bigr) = \sum_{k=1}^{n}
\bd {r}\bigl(\alpha_1^{G}\bigr)_k$.
\end{proof}

\subsection{Moment maps from the Cartan model}
When $G$ is compact, the 
Cartan model for the equivariant de Rham cohomology of $M$ is
better suited for computations than the Bott-Shulman-Stasheff model. Indeed,
in symplectic geometry the relationship between equivariant cohomology
and moment maps is typically described using this model. Analogously,
the main result in this section is that any cocycle in the Cartan model gives a
homotopy moment map.

Recall that the Cartan model
\cite[Sec.\ 4.2]{G-S} is the complex
\[
C^\ast_{G}(M) = \bigl (S^{\ast}(\g^{\vee})\tensor\Omega^{\ast}(M)
 \bigr)^{G}\\
\]
in which $\g^{\vee}$ is implicitly placed in degree 2, and
with differential
\[
d_{G}(\alpha)(x) = d(\alpha(x)) - \iota_{v_{x}} \bigl
( \alpha(x) \bigr) \quad \forall x \in \g,
\]
where $\alpha \in C_G(M)$ is regarded as a polynomial map $\alpha \colon \g \to \Omega^{\bullet}(M)$.
For any $k$, denote by 
 $$C^k_{G}(M)=\bigoplus_{2j\leq k}\;\; 
 {\bigl(S^{j}(\g^{\vee})\tensor\Omega^{k-2j}(M)\bigr)^G
}$$
the degree $k$ component of the Cartan model $C_G(M)$.
Given an element ${\alpha} \in C^k_{G}(M)$,
we denote by ${\alpha}_j$ its component in ${\left(S^{j}(\g^{\vee})\tensor\Omega^{k-2j}(M)\right)^G }$, for all $j=1,\dots,\lfloor \frac{k}{2} \rfloor$.
Hence, $${\alpha}={\alpha}_0+\dots+ {\alpha}_{\lfloor \frac{k}{2} \rfloor}.$$

\begin{defi}\label{extension}
An \textbf{extension} of an invariant closed differential form $\omega \in
\Omega^{k}(M)$ is a cocycle $\alpha \in C^k_{G}(M)$ such that 
$${\alpha}_0=\omega.$$

A {\bf {\boldmath$j$}-step extension} is an extension of the form $$\alpha={\alpha_0}+\dots+{\alpha}_j.$$
\end{defi}

\begin{thm}\label{thm:CartanMM}
  Given a degree $n+1$ cocycle 
$\omega +  \sum_{i=1}^{\lfloor\frac{n+1}{2}\rfloor} P_i\in C_G^{n+1}(M)$  
in the Cartan complex, with $\omega\in\Omega^{n+1}(M)^G$ and $P_i\in
  (S^i(\g^\vee)\otimes\Omega^{n-2i+1}(M))^G$, there is a natural
  homotopy moment map $\{f_k\}$ for the $G$-action on the pre-$n$-plectic
  manifold $(M,\omega)$. More precisely, for $k=1,\ldots,n$ we have
 \begin{equation} \label{eq:CartanMM}
 f_k = \sum_{i=1}^{\lfloor\frac{k+1}{2}\rfloor} \tfrac{(-1)^i\vs(k)i!(k-i)!}{2^{i-1}(k-2i+1)!}  
\Alt_k\bigl(\iota_\g^{k-2i+1}
P_i(\cdot,\underbrace{[\cdot,\cdot],\ldots,[\cdot,\cdot]}_{i-1})\bigr),
 \end{equation}
where $\Alt_k \maps (\g^{\vee}) ^{\tensor k} \to
\Lambda^{k} (\g^{\vee})$ is the (ungraded) skew-symmetrization 
$$q_1 \tensor q_2 \tensor \cdots \tensor q_k \mapsto \frac{1}{k!}
\sum_{\sigma \in \Sn_k} (-1)^{\sigma} q_{\sigma(1)} \tensor
q_{\sigma(2)} \tensor \cdots \tensor q_{\sigma(k)}.$$
In particular, the homotopy moment map $f$ is $G$-equivariant, i.e., $f_k \in \bigl(\Lambda^k(\g^\vee)\otimes\Omega^{n-k}(M)\bigr)^G$.
\end{thm}

\begin{proof}
The complete proof requires some technical prerequisites, and so it is given in the appendix
(see \ref{sec:CartanMM_proof}). Here we provide a sketch of the key points.
In Appendix \ref{sec:proof_CartanMM}, 
we first construct a natural map
\[
\Phi \maps C^{\ast}_{G}(M)  \to \Omega^{\ast}(G \ltimes M_{\bu})
\]
from the Cartan complex to the Bott-Shulman-Stasheff complex. This map
is often mentioned in the literature, but we were unable to find a
reference which is explicit enough for our needs.
We then consider the $(n+1)$-cocycle 
$\alpha = \sum_{k=0}^{n+1} \alpha_k =\Phi \bigl(\omega +
\sum_{i=1}^{\lfloor\frac{n+1}{2}\rfloor} P_i \bigr)$ in the
Bott-Shulman-Stasheff complex and show that the components $\alpha_0$ and $\alpha_1$
are $G$-invariant, even if $G$ is not compact.
The formula \eqref{eq:CartanMM} for the moment map $f$ then follows
from applying the first part of Thm.\
\eqref{thm:BSmm} to $\alpha$.
As the formula shows,  $f$ is constructed by applying the insertion operator
$\iota_{\g}^k$ to invariant elements in
$(S^i(\g^\vee)\otimes\Omega^{n-2i+1}(M))^G$. Hence, $f$ is equivariant.
\end{proof}

For the applications and examples ahead, it is worthwhile to write out
the first few terms of the moment map $f$.  If $\omega +
\sum_{i=1}^{\lfloor\frac{n+1}{2}\rfloor} P_i\in C_G^{n+1}(M)$ is a degree
$n+1$ cocycle as in Thm.\ \ref{thm:CartanMM}, then
\begin{equation} \label{eq:CartanMM_few}
\begin{split} 
f_1 &=  -P_1, \\
f_2 &= -\Alt_2 \iota_\g P_1, \\
f_3 &= \Alt_3\iota_\g^2 P_1 - \Alt_3P_2(\cdot,[\cdot,\cdot]), \\
f_4 &= \Alt_4\iota_\g^3 P_1 - 2\Alt_4\iota_\g P_2(\cdot,[\cdot,\cdot]),\\
f_5 &= -\Alt_5\iota_\g^4 P_1 + 3\Alt_4\iota_\g^2 P_2(\cdot,[\cdot,\cdot])- 3\Alt_5P_3(\cdot,[\cdot,\cdot],[\cdot,\cdot]), \\
&\vdots
\end{split}
\end{equation}
Also, we remark that if $M=\pt$, then $C_G^*(M) = C_G^*(\pt) =
\bigl(S^*(\g^\vee)\bigr)^G$, and Thm.\ \ref{thm:CartanMM} gives a map
$\bigl(S^i(\g^\vee)\bigr)^G \to \bigl(\Lambda^{2i-1}(\g^\vee)\bigr)^G$
which sends $P_i$ to

\[
f^\vs_{2i-1} = (-1)^i \tfrac{i!(i-1)!}{2^{i-1}}
\Alt_{2i-1} \bigl(P_i(\cdot,\underbrace{[\cdot,\cdot],\ldots,[\cdot,\cdot]}_{i-1})\bigr).
\]
This assignment differs by an additional factor of $-i!$ from the ``Cartan map'' defined in \cite[Section 2]{Cartan} (cf. \cite[Chapter VI Prop.\ IV]{ccc3}).

Applying Thm.\ \ref{thm:CartanMM} to
the special case when the degree $(n+1)$ Cartan cocycle is just a $1$-step
extension is particularly interesting. Indeed, it recovers the
classical correspondence in symplectic geometry between moment
maps and equivariant cohomology.

\begin{cor} \label{main1}
If $(M,\omega)$ is a pre-$n$-plectic manifold equipped with a
$G$-action, and $\omega - \mu$ is a $1$-step extension of $\omega$, then the maps
\[
\begin{array}{c}
f_{k} \maps \g^{\tensor k} \to \Omega^{n-k}(M),\\
f_{k}(x_{1},\ldots x_{k})= \vs(k)\iota(v_{x_{1}} \wedge\cdots \wedge
v_{x_{k-1}}) \mu(x_{k}),
\end{array}
\]
for   $ 1 \leq k \leq n$, are the components of a {$G$-equivariant} moment map $\g \to L_{\infty}(M,\omega)$. 
\end{cor}

\begin{proof}
Applying Thm.\ \ref{thm:CartanMM} to the cocycle $\omega + P_1$, where
$P_1=-\mu$, we obtain
\[ 
 f_k =  -\vs(k) \Alt_k \bigl(\iota_\g^{k-1} P_1 \bigr) = \vs(k) \Alt_k
 \bigl(\iota_\g^{k-1} \mu \bigr). 
\]
Let $x_1,\ldots,x_k \in \g$. The skew-symmetrization of $\iota_\g^{k-1}
\mu$ implies that $f_k(x_1,\ldots,x_k)$ is the sum of terms of the form 
\begin{equation} \label{eq:main_cor1}
\vs(k) (-1)^{\sigma} \frac{1}{k!}\iota(v_{x_{\sigma(1)}} \wedge\cdots \wedge v_{x_{\sigma(k-1)}}) \mu(x_{\sigma(k)})
\end{equation}
where $\sigma \in \Sn_{k}$. 

Since $d_{G}(\omega - \mu)=0$, we have
the equalities
\begin{equation}\label{eq:dgclosed}
d\omega=0, \quad 
d\mu(x) = -\iota_{v_{x}} \omega, \quad \iota_{v_{x}} \mu(x) =0
\end{equation}
for all $x \in \g$. The latter equality implies that
\[
\iota_{v_{x}} \mu(y) = -  \iota_{v_{y}}\mu(x) \quad  \forall x,y \in \g,
\]
which further implies
\[
\iota_{v_{k-1}} \iota_{v_{k-2}} \cdots \iota_{v_{1}} \mu(x_{k}) 
= - \iota_{v_{k-1}} \iota_{v_{k-2}} \cdots \iota_{v_{k}} \mu(x_{1}),
\]
where we write $v_{i}$ for the vector field $v_{x_{i}}$.
Hence, each term \eqref{eq:main_cor1} is skew-symmetric, and therefore
$f_{k}(x_{1},\ldots x_{k})= \vs(k)\iota(v_{x_{1}} \wedge\cdots \wedge
v_{x_{k-1}}) \mu(x_{k})$.

\end{proof}

\begin{remark} \label{no_compact_rem}
Note that in the proof of Thm.\ \ref{thm:CartanMM} and Cor.\ \ref{main1},
we do not require that $G$ be compact. We did not
need to use the fact that the cohomology of $C_{G}(M)$ is isomorphic to
the $G$-equivariant cohomology of $M$.
In particular, only the algebraic properties of the complexes
$C_{G}(M)$ and $\Omega^{\ast}(G \ltimes M_{\bu})$ are used to show that $1$-step
extensions give equivariant moment maps. Hence, Cor.\ \ref{main1}
also implies that we can produce equivariant moment maps from $1$-step
extensions in $C_{G}(M)$ where $G$ is any \textit{non-compact} Lie group
as well. This provides a useful algebraic
tool to build examples with.
\end{remark}

\section{Closed 3-forms}\label{sec:c3f}
In this section, we analyze the simplest case in which
non-strict moment maps appear, namely the case of closed
3-forms. Under the assumption that the Lie group $G$ is compact
semisimple, we prove the existence of moment maps provided the
$G$-action has a fixed point, and the uniqueness of moment maps up to
a certain equivalence. Furthermore, in the general setting, we
show that not all equivariant moment maps arise from $1$-step
extensions.

An important example is the Lie group itself
equipped with its Cartan 3-form. We consider this as a special case
later on in Sec.\ \ref{subs:conj}.

\subsection{Review of symplectic case}  
Let us first briefly discuss the case of 2-forms. In this case, moment maps as in Def./Prop.\ \ref{main_def} are necessarily strict by degree reasons.
 Let $G$ be a Lie group acting on a symplectic manifold $(M,\omega)$.
The classical notion of equivariant moment map is 
the following \cite{Ana}: a map $J \colon M \to \g^{\vee}$
such that $v_x$ is  the Hamiltonian vector field of $J^*(x)$ for all $x\in \g$ and so that $J$ is equivariant w.r.t. the $G$-action on $M$ and the coadjoint action of $G$ on $\g^{\vee}$.
In terms of  the co-moment map, i.e. the pullback of functions $f=J^*\colon \g \to C^{\infty}(M)$,
this means  
\begin{itemize}
\item [] $v_x$ is  the Hamiltonian vector field of  $f(x)$,  for all $x\in \g$,
\item [] $f\colon (\g,[\cdot,\cdot]) \to (C^{\infty}(M),\{\cdot,\cdot\})$ is a Lie algebra morphism,
\end{itemize} 
where 
%$v_x$ is the infinitesimal generator of the action given by $x$, $v_{f(x)}$ the hamiltonian vector field of the function $f(x)$, and 
$\{\cdot,\cdot\}$ the Poisson bracket on $M$.
Hence, in the symplectic case, our Def.\ \ref{main_def} agrees with classical notion of moment map.

An equivalent characterization of moment map is that
$\omega-f$ is a {closed} degree $2$ element of  the Cartan model for equivariant cohomology. In other words: in the symplectic case, all homotopy moment maps arise from equivariant extensions of $\omega$.

\subsection{Notation} \label{3-form_notation_sec}
In the remainder of this section, let $\omega$ be a closed 3-form on
$M$ which is invariant under  the action of a 
Lie group $G$. 
In this  case, a {moment map} (Def.\ \ref{main_def}) consists of two components 
$$ f_1 \colon \g \to \ham{1} ,\;\;\;\;\;\; f_2 \colon \g\otimes \g \to C^{\infty}(M),$$
where the second is skew-symmetric,
  such that  
  \begin{enumerate}
\item[A.]\label{cond_A}{$v_x$ is  the Hamiltonian vector field of  $f_1(x)$,  for all $x\in \g$,} 
\item[B.] \label{cond_B}{ the following two equations are satisfied: 
 \begin{eqnarray}
\label{eq:2momap1}
f_1([x,y]) -\underbrace{\{f_1(x), f_1(y)\}}_{\omega(v_x,v_y,\;\cdot\;)}&=&df_2(x,y)\\
\label{eq:2momap2}
 -\underbrace{l_3(f_1(x),f_1(y),f_1(z))}_{\omega(v_x,v_y,v_z)} &=&
f_2(x,[y,z])
 -f_2(y,[x,z])+f_2(z,[x,y]). 
\end{eqnarray}
} 
\end{enumerate}

%{Notice: $\mu|_{\cP_{\g}}$ recovers Swann's notion of multimoment map}.
 
\subsection{Existence of moment maps}  
If $G$ is also connected and semisimple, then it is well-known for the symplectic case
that a symplectic $G$-action admits a moment map \cite[Chapter
X]{Ana}.  However, for the present case, we now need an additional condition on the action.

\begin{prop}\label{cor:zeros}
If $G$ is compact, connected, and  semisimple and for all $x\in \g$ there is  point $p\in M$ such that $v_x(p)=0$, then there exists an equivariant homotopy moment map. 
\end{prop}

Notice that the assumption on the action in Prop.\
\ref{cor:zeros} is satisfied when $M$ is oriented, compact, and
has non-zero Euler characteristic; since in
  this case the Poincar\'e-Hopf
theorem implies that every vector field on $M$ has a zero.
 
 The proof of Prop.\ \ref{cor:zeros} is based on the following lemmas,
 which are straightforward analogs of well-known statements in
 symplectic geometry. We phrase
 these more generally for $\omega$ any invariant closed
 form $\omega \in \Omega^{n+1}(M)$ with $n \geq2$.

\begin{lemma} \label{semi}
If $\g$ satisfies $[\g,\g]=\g$, then there exists a linear map $\mu \colon \g
\to \ham{n-1}$ such that $v_x$ is a Hamiltonian vector field for $\mu(x)$, that is, $d\mu(x) = -\iota_{v_{x}} \omega$ for all $x\in \g$.
% \begin{equation}\label{ham}
%d\mu(x) = -\iota_{v_{x}} \omega.
%\end{equation}
%2) the result holds if one supposes instead that $H^{n}(M,\RR)=0$.
\end{lemma}

\begin{proof} Let $x\in \g$.
Since $\g=[\g,\g]$, we can write $x=\sum_i [x_i,x'_i]$. 
The locally Hamiltonian vector field $v_x$ can hence be written as
$ v_{\sum_i[x_i,x'_i]}=\sum_i[v_{x_i},v_{x'_i}]$, and by Prop.\ \ref{brac_prop} it is  the  Hamiltonian vector field of 
$$\mu(x)=\sum_i \iota(v_{x_i} \wedge v_{x'_i}) \omega.$$
Now let $x$ range through a basis of $\g$, and extend $\mu$ to obtain a  linear map $\mu \colon \g
\to \ham{n-1}$. 
\end{proof}

\begin{lemma} \label{compact}
If  $G$ is   compact and the linear map $\mu \colon \g
\to \ham{n-1}$ satisfies $d\mu(x) = -\iota_{v_{x}} \omega$ for all $x\in \g$,  then
$$\mu'(x)=\int_G g^*(\mu (\Ad_gx))$$ 
is equivariant 
%(i.e.,  $\L_{v_{x}} \mu'(y ) = \mu'([x,y]))$ 
and satisfies $d\mu'(x) = -\iota_{v_{x}} \omega$.
\end{lemma}
\begin{proof}
For every $g\in G$ and $x\in \g$, one computes that $d(g^*(\mu (\Ad_gx)))
=-\iota_{v_{x}}\omega$ using the fact that $\omega$ is $G$-invariant. Integrating over $G$ we obtain a map $\mu'$ which is equivariant.
\end{proof}

\begin{lemma}\label{lem:abc}
If the linear map $\mu \colon \g
\to \ham{n-1}$ is equivariant and
 $d\mu(x) = -\iota_{v_{x}} \omega,$ then $\iota_{v_{x}} \mu(x)$ is automatically a closed $n-2$ form for all $x\in \g$.
\end{lemma}

\begin{proof}
The equivariance of $\mu$ is equivalent to 
$\L_{v_{x}} \mu(y ) = \mu([x,y])$ for all $x,y\in\g$. Now
 $d\iota_{v_{x}} \mu(x)=-\iota_{v_{x}} d\mu(x)+\L_{v_{x}} \mu(x)=0$. 
 \end{proof}

\begin{proof}[Proof of Prop.\ \ref{cor:zeros}]
  Since $G$ is semisimple, Lemma \ref{semi} produces $\mu \colon
  \g \to \ham{n-1}$ satisfying $d\mu(x) = -\iota_{v_{x}} \omega$. The
  compactness of $G$ allows to apply Lemma \ref{compact}, and
  therefore we may assume that $\mu$ is equivariant. For every $x\in
  \g$, the function $\iota_{v_{x}} \mu(x)$ is a constant function by
  Lemma \ref{lem:abc}.  We conclude that it must be identically zero
  since it vanishes at $p$. Hence $\omega -\mu$ is a $1$-step extension,
  so  Thm.\ \ref{main1} produces an equivariant
    moment map with components $f_1=\mu$ and $f_2(x,y)=\iota_{v_x}\mu(y)$.
\end{proof}

\subsection{Uniqueness of moment maps} \label{unique_subsec} Next, we comment briefly on uniqueness issues. Recall that a moment
map in symplectic geometry is unique if $H^1(\g,\RR)=0$
\cite[Section26]{Ana}. Similarly, in the pre-2-plectic case,
cohomological constraints on $\g$ will ensure uniqueness, but only up to a certain
equivalence given by an action of $C^{\infty}(M)$.

\begin{lemma}\label{lem: equi}
\mbox{}
\begin{itemize}
\item[a)]{If $\omega-\mu$ is a $1$-step extension of $\omega$ and
    $\psi\colon \g \to C^{\infty}(M)$ is any $G$-equivariant linear map, then
    $\omega-(\mu+d\psi)$ is also a $1$-step extension.}
\item[b)]{If $(f_1,f_2) \maps \g \to L_{\infty}(M,\omega)$ is a moment
    map and $\psi \maps \g \to C^{\infty}(M)$ is any linear map, then
    we obtain a new moment map with components
    \begin{align*}
      &\tilde{f}_{1} =f_1+d\psi  \\
      &\tilde{f}_{2}(x,y)  = f_2(x,y)+\psi([x,y]). 
  \end{align*}
    }
\end{itemize}
\end{lemma}

% \begin{lemma}\label{lem: equi}
% a) If $\omega-\mu$ is an equivariant extension, then for all $G$-equivariant $\psi\colon \g \to C^{\infty}(M)$ we obtain a new equivariant extension $\omega-(\mu+d\psi)$.

% b) Given a moment map with components $f_1,f_2$,  for all  $\psi\colon \g \to C^{\infty}(M)$ we obtain a new moment map whose two components are   \begin{align*}
% f_1+d\psi \colon \g &\to \ham{1}\\
% \Lambda^2 \g&\to C^{\infty}(M),\;\; x\wedge y\mapsto
% f_2(x,y)+\psi([x,y]). 
%\end{align*}
%\end{lemma} 
\begin{proof}
A straightforward calculation left for the reader.
\end{proof}
\begin{remark}\label{rem:eqmomap}
In \cite{FLRZ}, a notion of equivalence for moment maps is considered
whose equivalence classes are strictly larger than those arising in Lemma \ref{lem: equi} b.
\end{remark}

\begin{prop} 
If $G$ is compact and semisimple or, more generally,  $H^1(\g,\RR)=H^2(\g,\RR)=0$, then: 
\begin{itemize}
\item[a)]{Any two equivariant extensions are related as in Lemma
    \ref{lem: equi} a,}
\item[b)]{Any two moment maps are related as in Lemma \ref{lem: equi} b.}
\end{itemize}
\end{prop}  
\begin{proof}
a) We have to show that given two $1$-step extensions
$\omega-\mu$, $\omega-\mu'$ 

there is a  
$G$-equivariant $\tilde{\psi}\colon \g \to C^{\infty}(M)$ such that $\mu'=\mu+d\tilde{\psi}$. Recall that by  Thm.\ \ref{main1}, $f_1=\mu$ and 
$f_2(x,y)=\iota_{v_x}\mu(y)$ are the components of a moment map. Similarly, we obtain a moment map $(f_1',f_2')$ using $\mu'$. Hence from  Eq.\  \eqref{eq:2momap1}, we see that $(\mu'-\mu)([x,y])$ is an exact 1-form for all $x,y\in \g$. From $[\g,\g]=\g$, which is equivalent to $H^1(\g,\RR)=0$, we deduce that
  there exists a map $\psi\colon \g \to C^{\infty}(M)$ such that $\mu'=\mu+d\psi$. 
  
The map $\psi$ will not be equivariant in general. We now modify it
suitably to obtain an equivariant map. {Since $\mu$ and $\mu'$
 are $1$-step extensions of $\omega$, 
 % Eqs.\ \eqref{eq:Ginv} and 
their equivariance and  Eq.\ \eqref{eq:dgclosed} imply that we have for all $x,y\in \g$:
\begin{equation}\label{eq:dpsi}
\L_{v_{x}} d\psi(y) 
%\overset{(\ref{eq:Ginv})}{=} 
=
d\psi([x,y]),\;\;\;\;\;\;\;
\iota_{v_{x}} d\psi(x)=\L_{v_{x}} \psi(x) 
=0.
%\overset{(\ref{eq:dgclosed})}{=}0.
\end{equation}}
Define $$c(x,y)=\L_{v_{x}} \psi(y)-\psi([x,y]).$$
The map $c$  is clearly the obstruction to $\psi$ being equivariant,
 and Eq.\  \eqref{eq:dpsi} implies that it is $\R$-valued and skew-symmetric.

We claim that $c$ is a Lie algebra cocycle. We have 
$$(\delta_{\CE}c)(x,y,z)=-\left(c([x,y],z)+c.p.\right)=\left(
\L_{v_{z}} \psi([x,y])-\psi([z,[x,y]])\right) +c.p.$$
 This is zero by the Jacobi identity of $\g$ and because  
 $$\L_{v_{z}} \psi([x,y]) +c.p.=\iota_{v_z}(\mu'-\mu)([x,y])+c.p.=
 \left(f_2(z,[x,y])-f'_2(z,[x,y])\right)+c.p.=0,$$
 where the last equality uses Eq.\  \eqref{eq:2momap2}.
 
From  $H^2(\g,\RR)=0$ we know that there exists $b\colon \g \to \RR$ such that $c=d_{\g}b$, that is,
$c(x,y)=-b([x,y])$ for all $x,y\in \g$. The map
$$\tilde{\psi}=\psi-b \colon \g \to C^{\infty}(M)$$ 
 clearly satisfies
 $\mu'=\mu+d\tilde{\psi}$, % (since it \mccomment{it=$\tilde{\psi}$?!} differs from $\psi$
 % by a constant term)
and it is equivariant since $
\L_{v_{x}} \tilde{\psi}(y)=\L_{v_{x}}{\psi}(y)=\tilde{\psi}([x,y])$.

b) We have to show that given two moment maps with components $(f_1,f_2)$
and $(f'_1,f'_2)$ respectively,  there is $\psi\colon \g \to C^{\infty}(M)$ such that $f'_1-f_1=d\psi$ and $(f'_2-f_2)(x,y)=\psi([x,y])$.

Notice that $f'_2-f_2 \colon \Lambda^2\g \to 
C^{\infty}(M)$ is a Lie algebra cocycle w.r.t. the \emph{trivial} representation of  $\g$ on $C^{\infty}(M)$, by Eq.\  \eqref{eq:2momap2}.
From $H^2(\g,C^{\infty}(M))=H^2(\g,\RR)\otimes C^{\infty}(M)=0$ we obtain a map $\psi \colon \g \to C^{\infty}(M)$ such that $(f'_2-f_2)(x,y)=\psi([x,y])$ for all $x,y \in \g$. To conclude we just need to assure that the identity
$f'_1-f_1=d\psi$ holds:  it does when both sides are applied to elements of the form $[x,y]\in \g$,   by Eq.\  \eqref{eq:2momap1} and the above. As $\g=[\g,\g]$, it holds for all elements of $\g$.
\end{proof}

\subsection{Equivariant moment maps not arising from equivariant cocycles}  \label{no_cocycle_sub_sec}
%!!!
Let $G$ act on the pre-$2$-plectic manifold $(M,\omega)$.  Let
$\omega-\mu$ be an equivariant $1$-step extension, and let
  $(f_1,f_2) \maps \g \to \li(M,\omega)$ be the corresponding
  equivariant moment map (Thm.\ \ref{main1}).Here we explain
  how to modify
  this moment map to obtain a new \emph{equivariant} moment map that does \emph{not}
  arise from any equivariant cocycle.

If a linear map
  $\tilde{f_1}\colon \g\to \Omega^1(M)$
\begin{itemize}
\item \text{takes values in closed 1-forms}
\item \text{is equivariant:}   $d(\iota_{v_x}\tilde{f_1}(y))=\cL_{v_x}\tilde{f_1}(y)=\tilde{f_1}([x,y])$ for all $x,y\in \g$,
\end{itemize}
then $f_1+\tilde{f_1}$  satisfies condition (A)  at the beginning of Sec.\ \ref{3-form_notation_sec} and is equivariant.
If  a skew-symmetric map $\tilde{f_2}\colon  \g\otimes\g \to C^{\infty}(M)$ 
\begin{itemize}
\item is equivariant
\item  satisfies $\tilde{f_1}([x,y])=d(\tilde{f_2}(x,y))$  for all $x,y\in \g$
\item  satisfies $\tilde{f_2}(x,[y,z])+c.p.=0$  for all $x,y,z\in \g,$
\end{itemize}
 then $f_1+\tilde{f_1}$ and
$f_2+\tilde{f_2}$ are the components of a new equivariant moment
map. (The last two conditions above guarantee Eq.\
  \eqref{eq:2momap1} and \eqref{eq:2momap2} are satified).
Furthermore, if we require that the constant function $\iota_{v_x}\tilde{f_1}(x)$ is
non-zero for some $x\in \g$, then the last condition in Eq.\
\eqref{eq:dgclosed} cannot be satisfied, and hence the new moment map can \emph{not}
arise from any equivariant cocycle.
When $\g$ is an abelian Lie algebra and $M$ is connected, the equivariance of $\tilde{f_1}$ boils down to the condition that $\iota_{v_x}\tilde{f_1}(y)$ is a constant function for all $x,y \in \g$, and the three conditions on $\tilde{f_2}$ simply
% to equivariance and  
imply that $\tilde{f_2}$ takes values in the constant functions.
Below we present a concrete instance of this construction: 
\begin{ep}\label{ep:noteq}
Let $G$ be the abelian group $S^1\times S^1$, and $(M,\omega)=(S^1\times S^1\times \RR, d\theta_1\wedge d\theta_2 \wedge dz)$. We take the infinitesimal action of $\g$ on $M$ to be 
%generated by the vector fields $\partial_{x}$ and $\partial_y$.
  $(1,0)\in \g \mapsto \partial_{\theta_1}$, $(0,1)\mapsto \partial_{\theta_2}$.
It is easily checked that $\omega-\mu$ is an $1$-step extension (see Eq.\  \eqref{eq:dgclosed}), where $$\mu \colon \g \to \ham{1},\;\; (1,0)\mapsto zd\theta_2, (0,1)\mapsto -zd\theta_1.$$
We can take $\tilde{f_1}$ to be $$\tilde{f_1}\colon \g\to \Omega^1_{closed}(M),\;\; (1,0) \mapsto d\theta_1, (0,1) \mapsto d\theta_2,$$ and  any arbitrary skew-symmetric $\tilde{f_2}\colon  \g\otimes \g \to \RR$. Then, as seen above, $f_1+\tilde{f_1}$ and $f_2+\tilde{f_2}$ are the components of an equivariant moment map, which can not arise from any equivariant cocycle
since $\iota_{v_{x}}\tilde{f_1}(x)=1\neq 0$ for $x=(1,0)$.
\end{ep}
The discussion in this subsection proves:
\begin{prop}\label{prop:notall}
Not all equivariant moment maps for actions on $2$-plectic manifolds
arise from cocycles in the Cartan complex via the formula given in
Thm.\ \eqref{thm:CartanMM}.
%1-step extensions.
\end{prop}

\begin{remark}
In \cite{FLRZ} Example \ref{ep:noteq} is used to obtain an equivariant moment map which is \emph{not equivalent} (in the sense of \cite{FLRZ}, see Remark \ref{rem:eqmomap}) to any moment map arising from a $1$-step extension of $\omega$. This provides a statement stronger than Prop.\ \ref{prop:notall} above.
\end{remark}

\section{Examples} \label{examples_sec} In this section we present
more examples of moment maps, many of which are generalizations of
interesting examples from symplectic geometry. In light of Remark
\ref{no_compact_rem}, all of them  can be understood as arising from
extensions as in Thm.\ \ref{main1}, even if $G$ is not compact.

In Sec.\ \ref{sec:flatc} we will give one more (infinite dimensional)
example.

\subsection{Exact pre-{\boldmath $n$}-plectic forms}\label{sub:exact}
Let $M$ be a manifold with a $G$-action, let
  $\alpha\in \Omega^n(M)^G$ be a $G$-invariant $n$-form and consider $\omega=d\alpha$.

\begin{lemma}\label{extexact}
  If $\mu\in (\Omega^{n-1}(M)\otimes \g^{\vee}[-2])^G$ is
  defined by $\mu(x)=\iota_{v_x}\alpha$, then $\omega-\mu$ is a $1$-step extension of $\omega$.
\end{lemma}

\begin{proof}
It suffices to check that $\mu$   
%satisfies  Eq.\  \eqref{eq:Ginv} 
is $G$-equivariant and satisfies \eqref{eq:dgclosed}. 
%Relation \eqref{eq:Ginv}
The equivariance is encoded by the condition $\cL_{v_x} {\mu(y)}=\mu([x,y])$ for all $x,y\in \g$, and 
follows from the Cartan relation $\iota_{[v,w]}=[\cL_v,\iota_w]=\cL_v\circ \iota_w -\iota_w\circ \cL_v$, together with the fact that $\cL_{v_x}\alpha=0$ for all $x$ in $\g$.
Concerning \eqref{eq:dgclosed}, clearly $d\omega=dd\alpha=0$. Moreover, the Cartan relation $d\circ i_{v_x}+i_{v_x}\circ d=\cL_{v_x}$, together with the invariance of $\alpha$ gives $d\mu(x)=-i_{v_x}\omega$. Finally the antisymmetry of $\alpha$ gives $i_{v_x}\mu(x)=0$.
\end{proof}

Hence by Thm \ref{main1} we obtain a homotopy moment map $\g \to L_{\infty}(M,\omega)$, given by  
 \[
\begin{array}{c}
f_{k} \maps \g^{\tensor k} \to \Omega^{n-k}(M),  ~ 1 \leq k \leq n\\
f_{k}(x_{1},\ldots x_{k})= (-1)^{k-1}\vs(k)\iota(v_{x_{1}} \wedge\cdots \wedge
v_{x_{k}}) \alpha
\end{array}
\]  

We present two concrete examples. The first one generalizes actions on cotangent bundles by cotangent lifts in symplectic geometry.
 \begin{ep}[Cotangent lifts]\label{ex:ctlifts}
 If $G$ acts on a manifold $N$ and $n$ is an integer, take
 $M=\Lambda^nT^*N$ and the $G$-action induced by the cotangent
 lift. Take $\alpha \in\Omega^n(M)$ to be the canonical form defined
 by 
\[
\alpha(w_1,\dots,w_n) \vert_{\xi}=\xi( \pi_*w_1, \ldots, \pi_*w_n) \quad
  \forall \xi \in M  
\]

where $\pi\colon M\to N$ is the projection and $w_1,\dots,w_n\in T_{\xi}M$. The $n$-form $\alpha$ is invariant, and it is known that $d\alpha$ is an $n$-plectic form. (In the case $n=1$ it is, up to sign, the canonical symplectic from on $T^*N$). From Lemma \ref{extexact} we know that an equivariant extension $\mu$ of $\omega$ is given by
$\mu(x)=\iota_{v_x}\alpha$, i.e. 
\[
\mu(x)(w_2,\dots,w_n) \vert_{\xi}  =\xi \bigl( (v_x)_{\pi(\xi)}, \pi_*w_2,
  \ldots, \pi_*w_n \bigr),
\]
where $v_x$ denotes the fundamental vector field for the action on $\Lambda^{n}T^{\ast}N$ (which restricts to the fundamental vector field for the action on $N$).

In other words, $(\mu(x))_{\xi}=\pi^*(\iota_{v_x}\xi)$.

\end{ep}

\begin{ep}[Linear actions on vector spaces]\label{linear}
  It is well-known that any symplectic representation $G \to
    \GL(V)$ on a symplectic vector space $(V,\omega)$ is
  Hamiltonian. More precisely, if we denote the action of $\xi
    \in \g$ on $V$ by $\xi \cdot p$,
and consider the unique moment map $J \maps V \to \g^{\vee}$
  vanishing at the origin, then its \textbf{components} $J^{\xi} \maps V \to \R$ $\forall \xi \in \g$
are the quadratic functions
\[
J^{\xi}(p) = J(p)(\xi) = -\frac{1}{2} \omega(p, \xi \cdot p).
\]
 Below we generalize such actions to higher degree forms, and will
see that the the analogs of the components for the moment map are no longer quadratic.

Let $V$ be any vector space and $\omega\in \wedge^{n+1}V^*$, giving rise to a constant pre-$n$-plectic form on $V$ (which is obviously closed and hence exact).
Consider a linear action of a Lie group $G \to \GL(V)$  preserving $\omega$.
We claim that a $G$-invariant primitive of $\omega$ is
\begin{equation*}
\alpha= \frac{\iota_E\omega}{n+1}
\end{equation*}
where $E$ is the Euler vector field on $V$ (in coordinates, $E=\sum_j x_j\pd{x_j}$). Indeed it is straightforward to check that $d\iota_E\omega=
\cL_{E}\omega=(n+1)\omega$. Also we have the equality 
$\cL_{v_\xi}\iota_E\omega=\iota_E\cL_{v_\xi}\omega+\iota_{[v_\xi,E]}\omega=0$  for all $\xi\in \g$, since $\omega$ is $G$-invariant and $E$ commutes with all linear vector fields.

Hence, we have a moment map induced by the $1$-step extension of $\omega$ given in Lemma \ref{extexact}. 
Let us study this map in more detail. We denote by
$\phi\colon \g\to \gl(V)$ the Lie algebra morphism
associated to the linear action.
The infinitesimal generator of the action given by any $\xi\in \g$
 is the linear vector field  $v_\xi|_p=-\phi(\xi) p$ (matrix multiplication).  For all $p\in V$ we have $E|_p=p$, so 
we obtain the following expression for the $k$-th component of moment map:
\begin{align*}
f_{k}(\xi_{1},\ldots \xi_{k})|_p&= (-1)^{k-1}\vs(k)\iota(v_{\xi_{1}} \wedge\cdots \wedge
v_{\xi_{k}}) \alpha|_p\\
&=-\vs(k)\frac{1}{n+1} \iota(p \wedge\phi(\xi_1)  p\wedge \cdots \wedge
\phi(\xi_k)  p) \omega.
\end{align*}
Notice that the coefficients of the $(n-k)$-form $f_{k}(\xi_{1},\ldots \xi_{k})$   are polynomials of degree $k+1$.
\end{ep}

The next example is a special case of Example \ref{linear} and  generalizes the following simple case of Hamiltonian action on a symplectic manifold: the action of the circle on $\RR^2$ by rotations, with moment map
$(x_1,x_2)\mapsto -\frac{1}{2}(x_1^2+x_2^2)$.

\begin{ep}[$\SO(n)$-action on $\RR^n$]\label{sorn}
We consider  the canonical action of $G=\SO(n)$ on  $\RR^n$, the latter endowed with the constant volume form
$$\omega=dx_1\dots dx_{n}.$$ 
 It is $G$-invariant, and by Example \ref{linear} an invariant
  primitive is $$\alpha=\frac{1}{n}\sum_{k=1}^{n}(-1)^{k+1}x_kdx_1\dots\widehat{dx_k}\dots dx_{n}.$$
 A  basis of the Lie algebra $\mathfrak{so}(n)$ is $\{e_{ij}: 1\le i<j\le n\}$, where $e_{ij}$ denotes the matrix with $-1$ in the $(i,j)$-th position, $1$ 
in the $(j,i)$-th position and zeros elsewhere.
The corresponding   generators of the action 
are given by 
$$v_{ij}=x_j\pd{x_i}-x_i\pd{x_j}.$$
By Lemma \ref{extexact} we know that 
 an equivariant $1$-step extension $\mu$ of $\omega$ is given by
\begin{eqnarray*}
\mu(e_{ij}) & = & \iota_{v_{ij}}\alpha\\
                     & = & \frac{1}{n}\Big(\sum_{k=1}^{i-1}-\sum_{k=i+1}^{n}\Big)(-1)^{k+1+i}  x_{j}x_k
dx_1\dots\widehat{dx_i}\dots \widehat{dx_k}\dots dx_{n} \\
& - & \frac{1}{n}\Big(\sum_{k=1}^{j-1}-\sum_{k=j+1}^{n}\Big)(-1)^{k+1+j}
x_{i}x_k 
dx_1\dots\widehat{dx_j}\dots \widehat{dx_k}\dots dx_{n},
\end{eqnarray*}
and that we  can build a moment map out of $\mu$.
Notice that,  in the symplectic case ($n=2$), one recovers  
$\mu(e_{12})=-\frac{1}{2}(x_1^2+x_2^2).$
\end{ep}

\subsection{Adjoint action and conjugacy classes}\label{subs:conj}
Let $G$ be a  
compact Lie group whose Lie algebra $\g$ is equipped with
an $\Ad$-invariant inner-product $\innerprod{\cdot}{\cdot}$. Then $G$
equipped with the bi-invariant Cartan 3-form
\[
\omega = \frac{1}{12} \innerprod{\theta_{L}}{[\theta_{L},\theta_{L}]}=\frac{1}{12} \innerprod{\theta_{R}}{[\theta_{R},\theta_{R}]}
\] 
is a pre-2-plectic manifold. Here $\theta_{L}$ and $\theta_{R}$ are, respectively,
the left and right invariant Maurer-Cartan 1-forms on $G$ satisfying
\[
d\theta_{L} + \frac{1}{2}[\theta_{L},\theta_{L}] =0 \quad d\theta_{R} - \frac{1}{2}[\theta_{R},\theta_{R}] =0.
\] 
If $G$ is, for example, also semi-simple, then $\omega$ is
non-degenerate and hence $(G,\omega)$ is 2-plectic.

Clearly, the action of $G$ on itself via conjugation preserves $\omega$, and
this action gives rise to a homotopy moment map. The Hamiltonian vector field
associated to $x \in \g$ is
\[
v_{x} = v^{L}_{x} - v^{R}_{x},
\]
where $v^{L}$ and $v^{R}$ are, respectively, the left and right
invariant vector fields on $G$ associated to $x$. A straightforward
calculation using the above Maurer-Cartan equations and the identity
$\Ad_{g} \theta_{L} = \theta_{R}$ gives:
\[
\frac{1}{2} d \innerprod{\theta_{L} + \theta_{R}}{x} = -\iota(v_{x}) \omega.
\]
The fact that this action lifts to a moment map follows from Thm.\ \ref{main1} and the
well-known fact that the $\g^{\vee}$-valued 1-form 
\[
\mu(x)=\frac{1}{2} \innerprod{\theta_{L} + \theta_{R}}{x} \quad
\forall x \in \g
\]
gives an equivariant extension $\omega -\mu$ of the Cartan 3-form. 

Let us write out the structure map $f_{2} \maps \g \tensor \g \to
\cinf(G)$ explicitly for this case. By definition, at a point $g \in G$,
we have
\begin{align*}
f_{2}(x,y)(g) = \iota(v_{x})\mu(y) \vert_{g} &= \frac{1}{2}
\innerprod{\theta_{L}(v_{x}) + \theta_{R}(v_{x})}{y} \vert_{g}\\
&=\frac{1}{2}\innerprod{ \bigl(\Ad_{g} -\Ad_{g^{-1}} \bigr)x}{y}.
\end{align*}
This piece of the moment map is related to an interesting invariant
2-form defined on the conjugacy classes of $G$. 
The Hamiltonian vector fields $v_{x}$ are minus the fundamental vector fields associated to
the conjugation action, and therefore span the tangent spaces of the
conjugacy classes. It follows from Proposition 3.1 in \cite{AMM} that
if $\iota_{C} \maps C
\hookrightarrow G$ is the inclusion of a conjugacy class, then
\[
dB = -\iota_{C}^{\ast}\omega
\]
where $B \in \Omega^{2}(C)^{G}$ is
\[
B_{g}(v_{x},v_{y}) = f_{2}(x,y)(g) \quad \forall g \in C \quad \forall
x,y \in \g. 
\]
Conjugacy classes are important examples of ``quasi-Hamiltonian
$G$-spaces'' \cite{AMM}, just as coadjoint orbits are examples of Hamiltonian
$G$-spaces in symplectic geometry. As the above example suggests, it may be interesting to
investigate further the relationship between  quasi-Hamiltonian $G$-spaces and
homotopy moment maps.

%\ccomment{below moved from section6}
\subsection{Examples from {\boldmath $2$}-step extensions} \label{subsec:2-step}
%\mccomment{we need to at least rename this subsection, rewrite first sentence.}
The previous examples of moment maps arose from $1$-step
extensions. Here we give examples in which $2$-step extensions arise naturally.
The second example, the $\SO(n)$-actions on the $n$-sphere,
generalizes the well-known Hamiltonian action of $S^1$ on $S^2$.	  
  
\subsubsection{Products} Let $G_i$ act on the manifold $M_i$ and $\alpha_i$ be an equivariant cocycle
in the Cartan model for this action, i.e. $d_{G_i}\alpha_i=0$ ($i=1,2$).
% \mcomment{We should add here what we worked out about cocyles with 3
% pieces...}
Consider the product $\alpha_1\alpha_2$ (obtained simply wedge-multiplying the
differential form components).
% (The product is the restriction of the product on the tensor product of the
% differential forms and the Weil algebra, and simply wedge-multiplies the
% differential form components.)
Then $\alpha_1\alpha_2$ is an equivariant cocycle in the Cartan model for the
product action of $G_1\times G_2$ on $M_1\times M_2$, since $$d_{G_1\times
  G_2}(\alpha_1\alpha_2)=(d_{G_1\times G_2}\alpha_1)\alpha_2\pm
\alpha_1(d_{G_1\times G_2}\alpha_2)=(d_{G_1}\alpha_1)\alpha_2\pm \alpha_1(d_{
  G_2}\alpha_2)=0.$$

We spell this out when the equivariant cocycles are of the kind considered in
Cor.\ \ref{main1}:

\begin{prop}
  Let $G_i$ act on $(M_i,\omega_i)$ with $\omega_i\in \Omega^{n_i+1}(M_i)$ a
  closed form, for $i=1,2$. Let $\omega_i-\mu_i$ be $1$-step extensions, and
  regard $\mu_i$ as maps $\g_i\to \Omega^{n_i-1}(M_i)$.  Then the product
  action of $G_1\times G_2$ on the pre-$(n_1+n_2+1)$-plectic manifold
  $(M_1\times M_2, \omega_1 \omega_2)$ admits an equivariant extension given
  by
\begin{equation*}
\omega_1 \omega_2-\eta+p
\end{equation*}
where
\begin{align*}
  \eta \colon \g_1\oplus \g_2 \to \Omega^{n_1+n_2}(M_1\times M_2),&\;\;\;\; x_1+x_2\mapsto \mu_1^{{x_1}}  \omega_2+\omega_1 \mu_2^{{x_2}}\\
  p \colon S^2(\g_1\oplus \g_2) \to \Omega^{n_1+n_2-2}(M_1\times
  M_2),&\;\;\;\; (x_1+x_2,y_1+y_2)\mapsto \frac{1}{2}\left(\mu_1^{{x_1}}
    \mu_2^{{y_2}}+\mu_1^{{y_1}} \mu_2^{{x_2}}\right).
\end{align*}
(Here we denote by $x_i,y_i,...$ elements of $\g_i$, by $v_{x_i}$ the corresponding vector fields on $M_i$, and for the sake of readability
we omit wedge products and write  $\mu^x$ for $\mu(x)$.)
  % (decorating with an index only the Lie algebra element, not the vector field itself).
 \end{prop}

In particular, when $\omega_1$ and $\omega_2$ are symplectic,
% (i.e.$n_1=n_2=1$), 
the  action of $G_1\times G_2$ on the $3$-plectic manifold $(M_1\times M_2, \omega_1 \omega_2)$ admits a moment map. 

\subsubsection{$\SO(n)$-action on the $n$-sphere}
A classical example of a Hamiltonian action in symplectic geometry is the action of $S^1=\SO(2)$  on $S^2$ by rotations about the $z$- axis. The ``height function'' $-z \colon S^2\to \RR$  provides a moment map.

We now extend this to $S^n$ for $n\le 5$, by means of a computation that makes clear how to generalize this to higher values of $n$ as well. More precisely, we assume the following set-up:
the unit sphere $M=S^n\subset \RR^{n+1}$ is endowed with the
$(n-1)$-plectic volume
%\footnote{In particular $\omega$ is an $(n-1)$-plectic form.} 
form $\omega$ obtained restricting $$\sum_{k=1}^{n+1}(-1)^{k+1}{x_k}dx_1\dots\widehat{dx_k}\dots dx_{n+1}=\alpha \wedge dx_{n+1}+\frac{(-1)^n}{n}x_{n+1}\cdot d\alpha,$$
where
$$\alpha=\sum_{k=1}^{n}(-1)^{k+1}{x_k}dx_1\dots\widehat{dx_k}\dots dx_{n}.$$
View $G=\SO(n)$ as a subgroup of $\SO(n+1)$ (embedded as matrices with a ``$1$'' in the lower right corner), and consider the obvious action on $S^n\subset \RR^{n+1}$ by matrix multiplication.

It is sufficient to find a $2$-step extension of $\omega$ in the Cartan complex, since then applying Thm. \ref{thm:CartanMM} one obtains an equivariant moment map. To this aim, we first spell out what it means to have such a cocycle in the Cartan complex. Consider an element $\omega+P+Q$ of the Cartan complex for the action  of some Lie group $G$, where
the degrees as polynomials on $\g$ are $0,1,2$ respectively. Choose a basis of $\g$, giving rise to infinitesimal generators of the action $v_I$ and  a basis $\{\xi_I\}$ of $\g^{\vee}$. Write $P=\sum_{I,J} \xi_I\otimes P^I$, and $Q=\sum_I \xi_I\xi_J\otimes Q^{IJ}$, where  $Q^{IJ}=Q^{JI}$. Then $\omega+P+Q$ is a Cartan cocycle if{f} $d\omega=0$ and
\begin{align}
\label{eq:I} -\iota_{v_I}\omega+dP^I&=0 \text{ for all }I,\\
\label{eq:IJ}  -\Sym(\iota_{v_J}P^I)+dQ^{IJ}&=0 \text{ for all }I\le J,\\
\label{eq:IJK}  \Sym(\iota_{v_K}Q^{IJ})&=0\text{ for all }I\le J\le K.
\end{align} 
 Here $\Sym(\iota_{v_J}P^I)=\frac{1}{2}(\iota_{v_J}P^I+\iota_{v_I}P^J)
  $ denotes the   symmetrization, and similarly for $\Sym(\iota_{v_K}Q^{IJ})$.

% As we wish to extend this to arbitrary values of $n$ in the future, we carry out part of our computations for arbitrary $n$.

We now preceed to find a $2$-step extension of $\omega$ in the Cartan complex.
\vspace{.5cm}

\paragraph{\textit{Step 1: find $P$ solving equation \eqref{eq:I}}}

\noindent
This $\SO(n)$-action on $S^n$ preserves $\omega$, for in Ex.\ \ref{sorn} we saw that it preserves $\alpha$.
With $v_{ij}$ as in Ex.\ \ref{sorn}, we have $$\iota_{v_{ij}}\omega=d\left(\frac{(-1)^{n+1}}{n}\iota_{v_{ij}}\alpha \cdot x_{n+1}\right)+\frac{n+1}{n}\iota_{v_{ij}}\alpha \wedge dx_{n+1},$$
using the fact that the action preserves $\alpha$.
{To write out the right-most term we consider 
$\iota_{v_{ij}}{\alpha}$.  Using the fact that the function $\sum_{k=1}^{n+1}x_k^2$ equals one on $S^n$, and therefore the pullback to $S^n$ of its differential $2\sum_{k=1}^{n+1}x_kdx_k$ vanishes, a  lengthy computation shows that 
$$\iota_{v_{ij}}{\alpha}=(-1)^{i+j}\left(1-x_{n+1}^2\right)dx_1\dots\widehat{dx_i}\dots \widehat{dx_j}\dots dx_{n}+ (\text{terms containing $dx_{n+1}$}).$$
Hence
$$\iota_{v_{ij}}{\alpha} \wedge dx_{n+1}=(-1)^{i+j+n}d\left((x_{n+1}-\frac{1}{3}x_{n+1}^3)
dx_1\dots\widehat{dx_i}\dots \widehat{dx_j}\dots dx_{n}\right).$$
} 
A primitive for  $\iota_{v_{ij}}\omega$ is therefore
\begin{equation}\label{eq:muij}
P(e_{ij})=\frac{(-1)^{n+1}}{n}x_{n+1}\cdot\iota_{v_{ij}}\alpha +(-1)^{i+j+n}\frac{n+1}{n}
\left(x_{n+1}-\frac{1}{3}x_{n+1}^3\right)
dx_1\dots\widehat{dx_i}\dots \widehat{dx_j}\dots dx_{n}.
\end{equation}
Notice that  $P\in \mathfrak{so}(n)^{\vee}\otimes \Omega^{n-1}(M)$ is
$\SO(n)$-invariant: the condition $$\cL_{v_{ij}}P(e_{i'j'})=P([e_{ij},e_{i'j'}]),$$
for all $i<j$ and $i'<j'$, follows from a computation that uses
the identity $[\cL_v,\iota_w]=\iota_{[v,w]}$ and a careful care of signs. 
%Defining $P:=-\mu$ we see that eq. \eqref{eq:I} is satisfied.
\vspace{.5cm}

\paragraph{\textit{Step 2: find $Q$ solving equation \eqref{eq:IJ}}}

\noindent
We now look for $Q$ so that eq. \eqref{eq:IJ} is satisfied.
Fix $i<j$ and $l<m$, where all four indices lie in $\{1,\dots,n\}$.
Assume first that all of $i,j,l,m$ are distinct. Then, writing $d_i$ as a short-form for  $dx_i$, we have
$$\iota_{v_{lm}}(d_1\dots\widehat{d_i}\dots \widehat{d_j}\dots d_{n})=(-1)^{l+m+1}B^{i,j}_{l,m}(x_ld_l+x_md_m)(d_1\dots\widehat{d_i}\dots \widehat{d_j}\dots \widehat{d_l}\dots\widehat{d_m}\dots d_{n})$$
where we define
$$B^{i,j}_{l,m}:=(-1)^{\mathrm{card}\{x\in \{i,j\}: l<x<m\}}.$$ 
From eq. \eqref{eq:muij}, using the obvious fact that $\iota_{v_{lm}}\iota_{v_{ij}}\alpha=-\iota_{v_{ij}}\iota_{v_{lm}}\alpha$, we obtain
\begin{align*}
&\frac{1}{2}(\iota_{v_{lm}}P^{ij}+\iota_{v_{ij}}P^{lm})=\\
&\frac{n+1}{2n}
\left(x_{n+1}-\frac{1}{3}x_{n+1}^3\right)
(-1)^{n+1}
(-1)^{i+j+l+m}B^{i,j}_{l,m}(x_id_i+x_jd_j+x_ld_l+x_md_m)(d_1\dots\widehat{d_i}\dots \widehat{d_j}\dots \widehat{d_l}\dots\widehat{d_m}\dots d_{n}).
\end{align*}
Using that $\sum_{k=1}^{n+1}x_kdx_k$ vanishes on $S^n$, 
we can replace the sum $x_id_i+x_jd_j+x_ld_l+x_md_m$ above by $-x_{n+1}d_{n+1}$, and 
we see that the above expression is an exact form, with primitive
$$Q^{(ij)(lm)}:=\frac{n+1}{2n}
\left(\frac{1}{3}x_{n+1}^3-\frac{1}{15}x_{n+1}^5\right)
(-1)^{n}
(-1)^{i+j+l+m}B^{i,j}_{l,m}(d_1\dots\widehat{d_i}\dots \widehat{d_j}\dots \widehat{d_l}\dots\widehat{d_m}\dots d_{n}).$$
When not all of $i,j,l,m$ are distinct, one can prove that $\iota_{v_{lm}}P^{ij}+\iota_{v_{ij}}P^{lm}=0$, hence 
even in  that case a primitive is given by the above formula for $Q^{(ij)(lm)}$, which equals zero since we are removing twice the same one-form.
Notice that $Q^{(ij)(lm)}=Q^{(lm)(ij)}$,
since   $B^{i,j}_{l,m}=B^{l,m}_{i,j}$. Further $Q\in S^2(\mathfrak{so}(n))^{\vee}\otimes\Omega^{n-4}(S^n)$ is $\SO(n)$-invariant.
%, since it satisfies the identity $\cL_{v_{pq}}Q^{(ij)(lm)}=Q^{[e_{ij},e_{pq}],e_{lm}]+Q^{(ij)(lm)}$
\vspace{.5cm}

\paragraph{\textit{Step 3: check that equation \eqref{eq:IJK} is satisfied} }

\noindent
We have to check that
$$\iota_{v_{pq}}Q^{(ij)(lm)}+\iota_{v_{lm}}Q^{(pq)(ij)}+\iota_{v_{ij}}Q^{(lm)(pq)}=0.$$
Since $n\le 5$, two of the six indices appearing above will agree, so we may assume that $i=p$.
In that case the middle term vanishes, while the first and last one cancel each other out, as one can check keeping track carefully of the signs.

\begin{remark}
1) Writing out explicitly Eq.\  \eqref{eq:muij} we obtain  
\begin{align*}
P(e_{ij})&=(-1)^{i+j+n}\left(x_{n+1}-\frac{n-2}{3n}x_{n+1}^3\right) dx_1\dots\widehat{dx_i}\dots \widehat{dx_j}\dots dx_{n}\\
&+\frac{(-1)^{i+j}}{n}\left(\sum_{k=1}^{i-1}-\sum_{k=i+1}^{j-1}+\sum_{k=j+1}^{n}\right)(-1)^{k-1} x_k x_{n+1}^2 \cdot dx_1\dots\widehat{dx_i}\dots \widehat{dx_j} \dots \widehat{dx_k}\dots dx_{n+1}.
\end{align*}
Notice that the cubic terms disappear only in the case $n=2$.

2) The above proof suggests that, for arbitrary values of $n$, an extension of the volume form $\omega$ on $S^n$ to a cocycle in the Cartan complex is given by
$\omega + P_1+\dots+  P_{\lfloor\frac{n}{2}\rfloor}$ where $P_1$ is given in eq. \eqref{eq:muij} and, for $k\ge 2$:
$$P_k((i_1,j_1)\otimes\dots\otimes (i_k,j_k))=\pm \frac{1}{k!}\frac{n+1}{n}
\left(\frac{x_{n+1}^{2k-1}}{(2k-1)!!}-\frac{x_{n+1}^{2k+1}}{(2k+1)!!}\right)
 (d_1\dots\widehat{d_{i_1}}\dots \widehat{d_{j_1}}\dots \widehat{d_{i_k}}\dots\widehat{d_{j_k}}\dots d_{n}),$$
where $i_1<j_1,\dots,i_k<j_k$ are integers between $1$ and $n$, and
$N!!:=N\cdot(N-2)\cdot  \dots 5\cdot 3$. 
%\nrt{End of part added by Marco in Feb, 2015}

%\mccomment{beginning of remark added by Martin Mar 12, 2015}
3) If $n$ is even, it is easy to see that the  volume form $\omega$ on
the $n$-sphere has a $\frac{n}{2}$-step extension. Indeed, the cocycle condition for a $\frac{n}{2}$-step extension $\omega + \sum_{i=1}^{\frac{n}{2}}P_i$ with $P_i\in\left(S^i(\mathfrak{so}(n)^{\vee})\otimes \Omega^{n-2i}(M)\right)^{\SO(n)}$ reads
\begin{align*}
 dP_1&= \iota_{\mathfrak{so}(n)} \omega, \\
dP_2 &=\Sym\left(\iota_{\mathfrak{so}(n)} P_1\right) , \\
\cdots &\\
 dP_{\frac{n}{2}}&=\Sym\left(\iota_{\mathfrak{so}(n)} P_{\frac{n-2}{2}}\right), \\
0 &=\Sym\left(\iota_{\mathfrak{so}(n)} P_{\frac{n}{2}}\right) 
\end{align*}
Using $H^{2i-1}(S^n) = 0$ and the fact that we can average to obtain $\SO(n)$-invariant forms, we can solve these one after the other for $P_1, P_2,...,P_{\frac{n}{2}}$. Note that since $n$ is even, $P_{\frac{n}{2}}$ has differential form degree zero and hence the last equation is automatically satisfied.
%\mccomment{end of remark added by Martin}
\end{remark}

\section{Obstructions and central extensions} \label{obstruct_sec}
Here we describe an obstruction to the existence of moment maps
characterized by a class in Lie algebra cohomology
(Thm.\ \ref{main2}). Conversely, we show that if both the class
and certain de Rham cohomology groups vanish, then a moment
map always exists (Thm.\ \ref{thm:momapyes}).  If the obstruction does
not vanish, we obtain a $L_{\infty}$-morphism into $\poi(M,\omega)$
not from the Lie algebra, but from a Lie $n$-algebra which can be
described  a `higher central extension' (Prop.\ \ref{ext_yes_prop}). 

Throughout this section, we assume we have
a Lie group $G$ 
acting on a pre-$n$-plectic manifold $(M,\omega)$
such that $\omega$ is preserved via infinitesimal diffeomorphisms i.e.\  the Lie algebra
$\g$ acts via local Hamiltonian vector fields:
\[
\L_{v_{x}} \omega =0,
\]
giving us the usual Lie algebra morphism
$$\g \to \Xlham(M),\; x \mapsto v_x.$$

\subsection{Lie algebra cohomology}

For any Lie algebra $\k$ (possibly infinite-dimensional), we 
have the cochain complex $\CE(\k)=\Hom(\Lambda^{\bullet}\k,\R)$ 
equipped with the usual Chevalley-Eilenberg differential $\delta_{\CE}$ as in Eq.\ \eqref{eq:CE_diff}
\[
\delta_{\CE} (c)(x_{1},\ldots,x_{n+1}) = \sum_{1 \leq i < j \leq n+1}
(-1)^{i+j}c([x_{i},x_{j}],x_{1},\cdots,
\hat{x}_{i},\cdots,\hat{x}_{j},\ldots,x_{n+1}).
\]

% \begin{remark}\label{rem:starting} 
% \mccomment{cf. Prop. \ref{prop:dbl_complex}, drop?}
% If $(f_{k}) \maps \g \to L_{\infty}(M,\omega)$ is a moment map, then
% the structure maps $f_{k}$ must satisfy Eqs.\  \eqref{main_eq_1} and
% \eqref{main_eq_2}. Observe that the left-hand side of both of these equations can be written as
% \[
% - \bigl(\delta_{\CE}f_{k-1} \bigr)(x_{1},\ldots,x_{k}).
% \]
% The relationship between Chevalley-Eilenberg cohomology and homotopy moment maps is the starting point of the work in \cite{FLRZ}.

% \end{remark}

Before considering $G$-actions, we make an observation about
the Lie algebra of local Hamiltonian vector fields. 
The following proposition says that $\omega$ determines a class in
$H^{n+1}_{\CE}(\Xlham (M))$.

\begin{prop}\label{cocycle_prop}
If $(M,\omega)$ is a pre-$n$-plectic manifold, 
then $\forall p \in M$ the linear map  
\begin{align*}
c_{p} &\maps \Lambda^{n+1} \Xlham(M) \to \R\\
&v_{1} \wedge \cdots \wedge v_{n+1} \mapsto (-1)^{n} \vs(n+1) \iota( v_{1} \wedge \cdots \wedge v_{n+1})
\omega \vert_{p},
\end{align*}
is a degree $(n+1)$-cocycle in $\CE(\Xlham(M))$. Moreover, if $M$ is connected, 
then the cohomology class $[c_{p}]$ is independent of  $p \in
M$.
\end{prop}
To prove the above proposition, we need the following technical
lemma. It is  a special case of \cite[Lemma 3.1]{MadsenSwannClosed}, to which we refer the reader for a proof. It also generalizes \cite[Lem.\ 3.7]{RogersL}\cite[Lem.\ 6.8]{HDirac}.
\begin{lemma}\label{tech_lemma}
If $(M,\omega)$ is a pre-$n$-plectic manifold and $v_{1},\hdots,v_{m} \in
\Xlham(M)$ with $m \geq 2$ then
\begin{multline} \label{big_identity}
d \iota(v_{1} \wedge\cdots \wedge v_{m}) \omega = \\(-1)^{m}\sum_{1 \leq i < j \leq
  m} (-1)^{i+j} \iota([v_{i},v_{j}] \wedge v_{1} \wedge \cdots
  \wedge \hat{v}_{i} \wedge \cdots \wedge \hat{v}_{j} \wedge \cdots \wedge v_{m})
\omega. 
\end{multline}
\end{lemma}

Now we have what we need to prove Prop.\ \ref{cocycle_prop}.
\begin{proof}[Proof of Prop.\ \ref{cocycle_prop}]
Clearly $c_{p} \in \Hom(\Lambda^{n+1}\Xlham(M),\R)$. We compute:
\begin{multline}\label{cocycle_prop_eq1}
\delta_{\CE}(c_{p})(x_{1},\ldots,x_{n+2}) =\\ \vs(n+1) \sum_{1 \leq i < j \leq n+2}
(-1)^{n+i+j} \iota([v_{i},v_{j}] \wedge v_{1} \wedge\cdots \wedge
\hat{v}_{i} \wedge\cdots \wedge \hat{v}_{j} \wedge\ldots v_{{n+2}})
\omega \vert_{p}.
\end{multline}
We use Lemma \ref{tech_lemma} for $m=n+2$. The right-hand side of 
of Eq.\  \eqref{cocycle_prop_eq1} above is equal to plus or minus the right-hand side of Eq.\ 
\ref{big_identity} evaluated at the point $p$.
However, the left-hand side of Eq.\  \eqref{big_identity} vanishes because $\omega \in
\Omega^{n+1}(M)$. Hence, $\delta_{\CE}(c_{p})=0$.

Now, assume $M$ is connected, and let $p' \in M$. There exists a path
$\gamma \maps [0,1] \to M$ such that $\gamma(0)=p$ and $\gamma(1)=p'$.
We define a map $b \maps \Lambda^{n} \Xlham(M) \to \R$ by
\[
b(v_{1},\ldots,v_{n}) = - \vs(n+1)\int_{\gamma} \iota( v_{1} \wedge \cdots \wedge v_{n})\omega.
\]
It follows from Lemma \ref{tech_lemma} that
\begin{multline*} 
d \iota(v_{1} \wedge\cdots \wedge v_{n+1}) \omega = \\(-1)^{n+1}\sum_{1 \leq i < j \leq
  n+1} (-1)^{i+j} \iota([v_{i},v_{j}] \wedge v_{1} \wedge \cdots
  \wedge \hat{v}_{i} \wedge \cdots \wedge \hat{v}_{j} \wedge \cdots \wedge v_{n+1})
\omega. 
\end{multline*}
Integrating both sides of the above equation over $\gamma$ gives
\[
\iota(v_{1} \wedge\cdots \wedge v_{n+1}) \omega\vert_{p'} -\iota(v_{1}
\wedge\cdots \wedge v_{n+1})\omega\vert_{p}
= (-1)^{n} \vs(n+1) \delta_{\CE}(b)(v_{1},\ldots,v_{n+1}),
\]
and, hence, $c_{p'} - c_{p} = \delta_{\CE}b$.
\end{proof}

If $G$ is acting on $(M,\omega)$, then Prop.\ \ref{cocycle_prop} gives
an important corollary.
\begin{cor}\label{cocycle_cor}
If $(M,\omega)$ is a pre-$n$-plectic manifold equipped with a $G$-action
such that $\g$ preserves $\omega$
then $\forall p \in M$ the linear map
\begin{align*}
c^{\g}_{p} \maps \Lambda^{n+1} \g &\to \R\\
x_{1} \wedge \cdots \wedge x_{n+1} &\mapsto (-1)^{n} \vs(n+1) \iota( v_{1} \wedge \cdots \wedge v_{n+1})
\omega \vert_{p},
\end{align*}
where $v_{i}$ is the vector field associated to $x_{i} \in \g$, is a
degree $(n+1)$-cocycle in $\CE(\g)$. Moreover, if $M$ is connected, 
then the cohomology class $[c^{\g}_{p}]$ is independent of  $p \in
M$.
\end{cor}
\begin{proof}
By assumption, $\g$ acts via local Hamiltonian vector fields, 
and $c^{\g}_{p}$ is the pullback of the cocycle defined in Prop.\
\ref{cocycle_prop} along the Lie algebra morphism $v_{-}$.
\end{proof}
\begin{remark}
Note that if the $G$-orbit of a point $p$ is of dimension smaller than $n+1$, then $[c^{\g}_p]=0$.
In particular, this holds if $p$ is a fixed point of the $G$-action.
% Note that if the $G$-action has a fixed point then $[c^{\g}_p]=0$.
\end{remark}

The next proposition shows that the class $[c^{\g}_{p}] \in H^{n+1}_{\CE}(\g)$ is an obstruction to having a homotopy
moment map. 
\begin{prop}\label{main2}
If $(M,\omega)$ is a connected pre-$n$-plectic manifold, and $M$ is equipped with a
$G$-action which induces a homotopy moment map
$ \g \to L_{\infty}(M,\omega)$, then \[[c^{\g}_{p}]=0 \] where $[c^{\g}_{p}] \in
H^{n+1}_{\CE}(\g)$ is the cohomology class defined in Cor.\ \ref{cocycle_cor}.
\end{prop}
\begin{proof}
By Def.\ \ref{main_def}  
the homotopy
moment map corresponds to structure maps
$f_{1},\ldots,f_{n}$ satisfying Eqs.\ \eqref{main_eq_1} and
\eqref{main_eq_2}. 
%By definition of the $L_{\infty}$-structure on
%$L_{\infty}(M,\omega)$, Eq.\  \eqref{cor_eq2} can be rewritten as
%\begin{multline}\label{main2_eq1}
%\vs(n+1)\iota( v_{1} \wedge \cdots \wedge v_{n+1})\omega=
%\sum_{1 \leq i < j \leq n+1}
%(-1)^{i+j+1}f_{n}([x_{i},x_{j}],x_{1},\ldots,\widehat{x_{i}},\ldots,\widehat{x_{j}},\ldots,x_{n+1}),
%\end{multline}
%where $v_{i}$ is the vector field associated to 
Since $\deg{f_{n}}=1-n$, the map $f_{n}$ takes values in
$\cinf(M)$. Let $p \in M$, and define
\[
b(x_{1},\ldots,x_{n}) = (-1)^{n+1}f_{n}(x_{1},\ldots,x_{n}) \vert_{p}.
\]
Clearly, $b \in \Hom(\Lambda^{n}\g,\R)$. Eq.\  \eqref{main_eq_2} then
implies
\[
(c^{\g}_{p})(x_{1},\ldots,x_{n+1})= \sum_{1 \leq i < j \leq n+1}
(-1)^{i+j}b([x_{i},x_{j}],x_{1},\ldots,\widehat{x_{i}},\ldots,\widehat{x_{j}},\ldots,x_{n+1}).
\]
Hence, $c^{\g}_{p}=\delta_{\CE}b$.
\end{proof}

\subsection{Lifting {\boldmath $\g$}-actions to moment maps}\label{subs:lift} 
Recall from Prop.\ \ref{map} that there is a surjective (and strict)
$L_{\infty}$-morphism
\[
\pi \maps \poi(M,\omega) \epi \Xham(M)
\]
which is simply the projection $(v,\alpha) \mapsto v$ in degree
0. Suppose we have a Lie group $G$ acting on $(M,\omega)$, such that
the infinitesimal action of $\g$ is via Hamiltonian vector
fields. Exhibiting a moment map for such an action means finding a lift
\begin{equation} \label{lift_diagram}
\xymatrix{
&& \poi(M,\omega) \ar[d]^{\pi} \\
\g \ar @{-->}[urr] \ar[rr]^{v_{-}} && \Xham(M)
}
\end{equation}
in the category of $L_{\infty}$-algebras. Since $\g$ acts by
Hamiltonian vector fields there always exists a (non-unique) degree
zero \textit{linear map}
\begin{equation*}
\begin{split}
\g & \to \poi(M,\omega)\\
x  & \mapsto (v_{x},\phi(x)) \in \Xham(M) \oplus \ham{n-1}
\end{split}
\end{equation*}
such that $d\phi(x) = -\iota_{v_{x}}\omega$.
When does such a linear map lift to an $L_{\infty}$-morphism?
Thm.\ \ref{main2} implies that it is necessary that the cohomology
class $[c^{\g}_{p}]$ vanish. The next theorem shows that when
certain topological assumptions are satisfied, this is also sufficient.
\begin{thm}\label{thm:momapyes}
Let $(M,\omega)$ be a connected pre-$n$-plectic manifold 
equipped with a $G$-action such that $\g$ acts via Hamiltonian vector
fields. Let 
\[
\phi \maps \g \to \ham{n-1}
\]
be any linear map such that $d\phi(x)=-\iota_{v_{x}} \omega$ for all
$x \in \g$. If $H^{i}_{\mathrm{dR}}(M) =0$ for $1 \leq i \leq n-1$ and 
$[c^{\g}_{p}]=0$, where $[c^{\g}_{p}] \in H^{n+1}_{\CE}(\g)$ is 
the cohomology class defined in Cor.\ \ref{cocycle_cor}, then there
exists a homotopy moment map
\[
(f_{k}) \maps \g \to L_{\infty}(M,\omega)
\]
such that 
\[
f_{1}=\phi.
\]
\end{thm}

\begin{proof}
  Let $f_{1}=\phi$, implying that
  $d(f_1(x))=-\iota_{v_{x}}\omega$ for all $x\in \g$. Notice that this
  equation is what is obtained allowing $k=1$ in
  Eq.\  \eqref{main_eq_1} (taking $f_0=0$).
We now find recursively solutions for the equations appearing in 
\eqref{main_eq_1}.

\textbf{Claim 1:} \emph{For every $2\le k \le n+1$, if $f_{k-1}$ satisfies Eq.\  \eqref{main_eq_1} for $k-1$, then 
\begin{multline}\label{eq:d0}  
\sum_{1 \leq i < j \leq k}
(-1)^{i+j+1}f_{k-1}([x_{i},x_{j}],x_{1},\ldots,\widehat{x_{i}},\ldots,\widehat{x_{j}},\ldots,x_{k})-\vs(k)\iota(v_{1}\wedge \cdots \wedge v_{k})\omega
\end{multline}
is a closed $n+1-k$-form for all $x_1,\dots,x_k\in \g$.}

\noindent To prove the claim we proceed as follows.
We have
\begin{align*} 
%\label{big_identity}
&d \iota(v_{1} \wedge\cdots \wedge v_{k}) \omega \\=& (-1)^{k}\sum_{1 \leq i < j \leq
  k} (-1)^{i+j} \iota([v_{i},v_{j}] \wedge v_{1} \wedge \cdots
  \wedge \hat{v}_{i} \wedge \cdots \wedge \hat{v}_{j} \wedge \cdots \wedge v_{k})
\omega\\
=& (-1)^{k}\sum_{1 \leq i < j \leq k} (-1)^{i+j} \vs(k-1)
\Big(- (\delta_{\CE} f_{k-2})([x_{i},x_{j}],x_{1},\ldots,\widehat{x_{i}},\ldots,\widehat{x_{j}},\ldots,x_{k})\\
 & \hspace{5.9cm}-  
df_{k-1}([x_{i},x_{j}],x_{1},\ldots,\widehat{x_{i}},\ldots,\widehat{x_{j}},\ldots,x_{k}) 
\Big)\\
=&\underbrace{(-1)^k\vs(k-1)}_{=\vs(k)}\Big(  
(-\delta^2_{\CE} f_{k-2})(x_{1},\ldots,x_{k})\\
 & \hspace{2.8cm}+
d\sum_{1 \leq i < j \leq k}
(-1)^{i+j+1}f_{k-1}([x_{i},x_{j}],x_{1},\ldots,\widehat{x_{i}},\ldots,\widehat{x_{j}},\ldots,x_{k})\Big)
\end{align*}
using Lemma \ref{tech_lemma} in the first equality, and in the second the fact that $f_{k-1}$ satisfies Eq.\  \eqref{main_eq_1} for $k-1$ as well as
%  Remark \ref{rem:starting}. 
the definition of $\delta_{\CE}$.
 Since the Chevalley-Eilenberg differential $\delta_{\CE}$ 
squares to zero, the claim follows.

\textbf{Claim 2:} \emph{For all $2\le k \le n$, there exist 
 $f_k \colon \Lambda^k\g \to \Omega^{n-k}(M)$ satisfying  Eq.\  \eqref{main_eq_1} for $k$.}

\noindent We prove Claim 2 by induction on $k$. The  case $k=1$ holds, as seen earlier, with $f_1=\phi$. We fix $2\le k \le n$. By the induction assumption we are allowed to apply Claim 1 for $k$.
The assumption $H^{n+1-k}(M)=0$ implies that there exists
$f_k \colon \Lambda^k\g \to \Omega^{n-k}(M)$ such that $f_k(x_1,\dots,x_k)$  is a primitive for the $n+1-k$-form \eqref{eq:d0}, for all $x_1,\dots,x_k\in \g$. Equivalently, $f_k$ satisfies  Eq.\  \eqref{main_eq_1} for $k$, proving Claim 2. 

In general, $f_n$ will not satisfy Eq.\  \eqref{main_eq_2}. 
%(which we view as a ``consistency condition'' for $f_n$). 
It will   if{f} $h\in\Hom(\Lambda^{n+1}\g,C^{\infty}(M))$ vanishes, where $$
%\begin{multline}\label{eq:d0n+1}  
%  \sum_{1 \leq i < j \leq n+1}
 % (-1)^{i+j+1}f_{n}([x_{i},x_{j}],x_{1},\ldots,\widehat{x_{i}},\ldots,\widehat{x_{j}},\ldots,x_{n+1})
% -\vs(n+1)\iota(v_{1}\wedge \cdots \wedge v_{n+1})\omega.
h(x_1,\dots,x_{n+1})=(\delta_{\CE}f_{n})(x_{1},\ldots,x_{n+1})+\vs(n+1)\iota(v_{1}\wedge \cdots \wedge v_{n+1})\omega.
$$
%which is just \eqref{eq:d0} for $m=n+1$.
%(Here we used Remark \ref{rem:starting}.)
Now fix $p\in M$. We evaluate both summands of $h$ at $p$,
and obtain two elements of $\Hom(\Lambda^{n+1}\g,\R)$: the first one is $\delta_{\CE}$-exact
  by construction, the second is equal to $\pm c^{\g}_p$,
hence it is $\delta_{\CE}$-exact by assumption. This means that there exists $b\in
\Hom(\Lambda^{n}\g,\R)$ such that  $h|_p=\delta_{\CE}b$.  However by Claim 1 (for $k=n+1$) we know that $h(x_1,\dots,x_{n+1})$
is a closed zero form for all $x_1,\dots,x_{n+1}\in \g$, and since $M$
is connected this means that  $h$ lies in
$\Hom(\Lambda^{n+1}\g,\R)$. Hence $$h=\delta_{\CE}b\in \Hom(\Lambda^{n+1}\g,\R).$$ Replacing $f_n$ by
$f_n-b$ we therefore obtain a solution of Eq.\  \eqref{main_eq_2}, which
still satisfies Eq.\  \eqref{main_eq_1} for $n$ as $b$ takes values in the constants.  We conclude that
$f_1,\dots,f_{n-1},f_n-b$ are the components of a homotopy moment map.
\end{proof}
\begin{remark} \label{inf-dim_remark}
  % Note that the above proof does not require that the equality of
  % cohomology classes $[c^{\g}_{p}]=[c^{\g}_{p'}]$ hold for $p,p' \in
  % M$. To have the theorem, one only needs to find $p \in M$ such that
  % $[c^{\g}_{p}]=0$ and verify that the assumptions on de Rham
  % cohomology and connectedness hold. 
  Note that $G$ need not be
  finite-dimensional here; the theorem also applies to actions by
  locally exponential infinite-dimensional Lie groups. In Sec.\ \ref{sec:flatc},
  we consider a case in which $G$ is such a group acting on a
  pre-$n$-plectic locally convex topological vector space.
\end{remark}

\begin{remark}
The assumptions of Thm.\ \ref{thm:momapyes} can be
  weakened; in fact, only particular components of $H^{\bullet}_{\CE}(\g)\otimes H^{\bullet}_{\mathrm{dR}}(M)$ need to vanish \cite{FLRZ}.
\end{remark}

\subsection{Central {\boldmath $n$}-extensions} \label{central_ext_sec}
If $(M,\omega)$ is a connected symplectic manifold, then Kostant's
construction \cite{Kostant:1970} gives a morphism of central extensions
\[
\xymatrix{
\R \ar[r] \ar[d]  & \R \ar[d] \\
\widehat{\g} \ar[d] \ar[r]& \cinf(M) \ar[d]^{\pi} \\
\g \ar[r]^-{v_{-}}  & \Xham(M)
}
\]
where $\widehat{\g}$ is the central extension corresponding to the
2-cocycle $c^{\g}_{p}$. This central extension is non-trivial iff
there is no moment map which lifts the $\g$-action. 

Now we describe how these ideas generalize to
higher cases. First we recall a theorem of Baez and Crans \cite[Thm.\ 55]{hd6}:
Given a Lie algebra $(\g,[\cdot,\cdot])$ and a degree $(n+1)$-cocycle
$c \maps \Lambda^{n+1} \g \to \R$, there exists a Lie $n$-algebra
whose underlying complex is $\g$ in degree 0, $\R$ in degree $1-n$ and
0 in all other degrees. The structure maps are
trivial except in degree zero where we have:
\begin{equation*}
\begin{split}
l_{2}(x_1,x_2)& =[x_1,x_2] \\
l_{n+1}(x_{1},\ldots,x_{n+1}) & = c(x_{1},\ldots,x_{n+1})\\
l_{k}&=0 \quad \text{if } k \neq 2, k \neq n+1.
\end{split}
\end{equation*}
We call this Lie $n$-algebra a \textbf{central {\boldmath $n$}-extension} of $\g$
and denote it by $\widehat{\g}_{c}$. If $c$ and $c'$ are two such cocycles which differ by a coboundary,
then the corresponding Lie $n$-algebras are quasi-isomorphic (Cor.\ \ref{ext_qiso}).
If $n=1$, then we recover the usual notion of central extension by setting $l_{2}=[\cdot,\cdot] +c$.  
Let $\pi_{\g} \maps \widehat{\g}_c \epi \g$ denote the projection. It 
clearly lifts to a strict $L_{\infty}$-morphism.
\begin{prop}
The short exact sequence of complexes
\[
\R[n-1] \to \widehat{\g}_c \xto{\pi_{\g}} \g 
\]
lifts to a strict exact sequence \cite[Def.\ 9.3]{RogersPre} in the category of $L_{\infty}$-algebras.
\end{prop}
% The following proposition is analog to Thm. \ref{thm:momapyes}. It differs in that $[c^{\g}_{p}]$ is not assumed to vanish, and one obtains a ``homotopy moment map'' for the corresponding extension $\widehat{\g}$ rather than for $\g$.

The following proposition is the higher analog of Kostant's
construction in symplectic geometry for central extensions such as the Heisenberg Lie
algebra. 
\begin{prop} \label{ext_yes_prop}
Let $(M,\omega)$ be a connected pre-$n$-plectic manifold 
equipped with a $G$-action such that $\g$ acts via Hamiltonian vector
fields and let $p \in M$. Assume $H^{k}_{\mathrm{dR}}(M) =0$ for $1 \leq k \leq n-1$. 
If $\widehat{\g}$ is the central $n$-extension constructed from the $(n+1)$-cocycle
$c^{\g}_{p}$ defined in Cor.\ \ref{cocycle_cor}, then there exists an
$L_{\infty}$-morphism
\[
(f_{i}) \maps \widehat{\g} \to \poi(M,\omega)
\]
such that the following diagram (strictly) commutes
\begin{equation}\label{extend_diag}
\xymatrix{
\widehat{\g} \ar[d]_{\pi^{\g}} \ar[r]^-{(f_{i})} & \poi(M,\omega) \ar[d]^{\pi} \\
\g \ar[r]^{v_{-}} & \Xham(M)
}
\end{equation}
\end{prop}
\begin{proof}
We shall produce maps $f_{1},\ldots,f_{n}$ such that equalities
given in Prop.\ \ref{ext_morph_prop} are satisfied.
Since $\g$ acts by Hamiltonian vector fields, there exists 
a linear map $\phi \maps \g \to \widetilde{\ham{n-1}}$
with $\phi(x)=(v_{x},\alpha_{x})$ such that $d \alpha_{x} = -\iota_{v_{x}} \omega$.
The map
\begin{equation*}
\begin{split}
f_{1}(x) &= \phi(x) \quad \forall x \in \g \\
f_{1}(r) &= (-1)^{n}r \in \cinf(M) \quad \forall r \in \R
\end{split}
\end{equation*}
gives a degree 0 chain map $f_{1}$ from the underlying complex of
$\widehat{\g}$ to that of $\poi(M,\omega)$.

We then proceed as we did in the first part of the proof of Thm.\ \ref{thm:momapyes}.
Namely, since all closed $k$-forms have a primitive for $0 \leq k \leq
n-1$, we inductively obtain maps $f_i \maps \Lambda^{i}\g \to
\Omega^{n-i}(M)$ for $i=2,\ldots,n$ such that Eq.\ 
\eqref{ext_eq1} is satisfied. Let $b \maps \Lambda^{n} \g \to \R$ be
\[
b(x_{1},\ldots,x_{n}) = f_{n}(x_{1},\ldots,x_{n})\vert_{p},
\]
and let $\tilde{f}_{n}=f_{n}-b$. 

Since $db(x_{1},\ldots,x_{n}) =0$ for
all $x_{i}$, the map $\tilde{f}_{n}$ also satisfies \eqref{ext_eq1}.
It remains to show that Eq.\  \eqref{ext_eq2}
holds i.e.\ given $x_{1},\ldots,x_{n+1} \in \g$, the function
\begin{multline*}
C=\sum_{1 \leq i < j \leq n+1}
(-1)^{i+j+1}\tilde{f}_{n}([x_{i},x_{j}],x_{1},\ldots,\widehat{x_{i}},\ldots,\widehat{x_{j}},\ldots,x_{n+1})\\
+(-1)^{n} c^{\g}_{p}(x_{1},\ldots,x_{n+1}) - \vs(n+1)\iota(v_{1}\wedge\cdots \wedge v_{n+1})\omega
\end{multline*}
vanishes. Using Lemma \ref{tech_lemma} and 
Eq.\  \eqref{ext_eq1}
for the case $m=n$, we conclude that $C$ is closed and therefore:
\begin{equation*}
C =C(p) = (-1)^n c^{\g}_{p}(x_{1},\ldots,x_{n+1}) -\vs(n+1)\iota(v_{1}\wedge\cdots \wedge v_{n+1})\omega \vert_{p}=0.
\end{equation*}
Hence, the collection $f_{1},\ldots,f_{n-1},\tilde{f}_{n}$ gives the desired
morphism, and it follows from the definition of $f_{1}$ that the
diagram \eqref{extend_diag} commutes.
\end{proof}
A more homotopy-theoretic and conceptual interpretation  of the above
proposition appears in  Sec.\ 3.5 of \cite{FRS}. 

Below we give some examples of
what kinds of Lie $n$-algebras can be constructed in this way.

\begin{ep}[Heisenberg $n$-algebra] 
Let $V$ be a finite-dimensional real vector space. A linear
non-zero skew-symmetric form $\omega \in \Lambda^{n+1}
V^{\ast}$ of degree $n+1$  induces a translation-invariant closed
differential form on $V$. Therefore, $(V,\omega)$ is a pre-$n$-plectic
manifold and $V$ (seen as an abelian Lie algebra) acts on itself via translations. This gives a Lie
algebra morphism $v_{-} \maps V \to \Xham(V)$. 
Since $\omega$ is non-zero, the degree $(n+1)$ class $[c^{V}_{p}]$ is
non-trivial. Hence, there is no homotopy moment map lifting the
action of $V$. Let $\widehat{V}$ be the associated central
$n$-extension. Prop.\ \ref{ext_yes_prop} implies that $\widehat{V}$
sits in a commuting diagram of $L_{\infty}$-algebras of the form \eqref{extend_diag}. 

Compare with Ex.\ \ref{linear}, for which any linear action on $(V,\omega)$  admits a moment map.
\end{ep}

\begin{ep}[String Lie 2-algebra]
Let $G$ be a compact connected simple Lie group, and let 
$\omega=\frac{1}{12} \innerprod{\theta_{L}}{[\theta_{L},\theta_{L}]}$
be the Cartan 3-form.  As previously mentioned in
Sec.\ \ref{subs:conj}, $(G,\omega)$ is a 2-plectic manifold, and the
action of $G$ on itself via conjugation lifts to a homotopy moment
map. Clearly, $\omega$ is also preserved by the action of $G$ on
itself via left-translation, but the corresponding degree 3 class
$[c^{\g_{\mathrm{L}}}_{p}]$ is not trivial. (Indeed, $\langle \cdot,
[\cdot,\cdot] \rangle$ is a generator of $H^{3}_{\CE}(\g)$).
The corresponding central 2-extension is the
\textbf{string Lie 2-algebra} $\str(\g)$. When $G=\mathrm{Spin}(n)$, this Lie 2-algebra (or
rather its integration) plays a very interesting role in a certain elliptic
cohomology theory and in the theory of ``spin structures'' on loop spaces.
(See, for example,  Sec.\ 1 of \cite{S-P:2011} for a review.)

Since $G$ is compact and simple we have $H^{1}_{\mathrm{dR}}(G) \cong
H^{1}_{\CE}(\g)=0$. Hence, Prop.\ \ref{ext_yes_prop} implies that
there is a commuting diagram of $L_{\infty}$-algebras:
\[
\xymatrix{
\str(\g) \ar[d]_{\pi^{\g}} \ar[r] & \poi(G,\omega) \ar[d]^{\pi} \\
\g \ar[r]^{v^{\mathrm{left}}_{-}} & \Xham(G)
}
\]
This result gives a nice conceptual interpretation to the relationship
previously established in \cite{Baez-Rogers:2010}
between $\str(\g)$ and $L_{\infty}(G,\omega)$.
\end{ep}

\section{Moduli spaces of flat connections}\label{sec:flatc}
Here we consider homotopy moment maps on spaces of connections over higher-dimensional manifolds (see Thm.\ \ref{moduli_mmap_thm}). 
Currently, our motivation for this example is simply
to generalize the famous Atiyah-Bott construction \cite{A-B:1983} in symplectic
geometry.
Since our construction begins by considering an invariant polynomial in $S(\g^{\vee})^{G}$
of higher degree $\ge 2$, it is possible that some of these ideas could
find application in certain topological field theories which
  generalize Chern-Simons theory.

\subsection{Invariant polynomials}
Given an integer $n\ge 1$, we consider the following data:
\begin{itemize}
\item a real, finite dimensional Lie algebra $\g$ equipped with a
  invariant polynomial $q\in S^{n+1}(\g^{\vee})^G$,
\item a $(n+1)$-dimensional compact, oriented manifold $M$, and
\item a principal $G$-bundle $\pi\colon P\to M$,
where $G$ is any Lie group integrating $\g$. The group $G$ acts on the
right of $P$ via diffeomorphisms $R_{g}$.
\end{itemize}
We denote by 
\[
\hat{\xi}(p) = \frac{d}{dt} R_{\exp(t \xi)} (p)  \vert_{t=0}
\]
the infinitesimal generators of the action of
$G$ on $P$, for all $\xi \in \g$. 
We say an invariant polynomial  $q$ is
\textbf{non-degenerate} iff the map 
\[
\begin{array}{c}
\g \to S^n(\g^{\vee}) \\
x \mapsto \iota_{x}q
\end{array}
\]
is injective.
\begin{ep} \label{symtrace_example}
If $G$ is a matrix group, then the symmetrized (real) trace gives obvious
examples of invariant polynomials. In particular, for
$G=\SU(N)$,  we define:
\[
q_{k}(x_1,\ldots,x_{k}) = -\frac{1}{k!}\sum_{\sigma \in \Sn_{k}} \mathrm{Re}
\Tr (x_{\sigma(1)} x_{\sigma(2)}  \cdots x_{\sigma(k)}) \quad \forall
x_i \in \su(N).
\]

 It is well known that the polynomial $q_{2}$ gives a real inner product on
$\su(N)$, but more generally, one can show for $G=\SU(2)$ that every $q_{2n}$ is non degenerate, for $n>0$.
Consider $\{e_{i}\}$ the basis of $G=\SU(2)$ given by
\[
e_{1} = \frac{1}{2} \begin{pmatrix} 0&1 \\ -1 & 0 \end{pmatrix}
\quad 
e_{2} = \frac{1}{2} \begin{pmatrix} 0& \text{i} \\ \text{i} & 0 \end{pmatrix}
\quad
e_{3} = \frac{1}{2} \begin{pmatrix} \text{i} & 0 \\ 0 & -\text{i} \end{pmatrix}.
\]
The identity
\begin{equation} \label{su2_identity}
e_{i} e_{j} = -\frac{1}{4} \delta_{ij} I + \frac{1}{2} \sum_{k}
\varepsilon_{ijk} e_{k},\end{equation}
where $\varepsilon_{ijk}$ is totally skew and $\varepsilon_{123}=1$, implies that 
$e_{i} e_{j}^{2n} = (-\frac{1}{4})^{n} e_i$. Therefore \begin{equation*}
e_{i} e_{j}^{2n+1} = (-\frac{1}{4})^{n+1} \delta_{ij} I + \frac{1}{2} (-\frac{1}{4})^{n}\sum_{k}
\varepsilon_{ijk} e_{k}.\end{equation*}
In particular
\begin{equation}\label{truc}
q_{2(n+1)}(e_i,e_j, \dots, e_j)=-\mathrm{Re} \Tr(e_{i} e_{j}^{2n+1})=-2(-\frac{1}{4})^{n+1} \delta_{ij}.
\end{equation}
\end{ep}
 This enables to show that  $q_{2(n+1)}$ is non-degenerate. Indeed, suppose there exists $x
\in \su(2)$ such that
\[
q_{2(n+1)}(x,y_{1},\dots,y_{2n+1})=0 \quad \forall y_{i} \in \su(2).
\]
Write $x=\sum_{i}x^{i} e_{i}$, then (\ref{truc}) means that  $x_i=q_{2(n+1)}(x,e_{i},\dots,e_{i}) = 0$ for all
$i$  implies that $x=0$.
 So $q_{2(n+1)}$ is non-degenerate. Note that this is not true for all $q_{k}$.
In fact,  Eq.\  \eqref{su2_identity} shows that $q_{3}=0$. 
 
\subsection{The gauge group action}
A \textbf{connection} on $P$ is a $\g$-valued 1-form $A \in \Omega^1(P,\g)$
satisfying 
\[
A(\hat{\xi})=\xi, \quad  R_g^{\ast}A=\Ad_{g^{-1}} A 
\]
for all $\xi \in \g$ and $g \in G$.
The set $\cA$ of all connections on $P$ is
an affine space modeled on the vector space $(\Omega^1_{\hor}(P) \tensor \g)^G$,
where the first factor denotes the 1-forms on $P$ annihilated by
vectors tangent to the fibers. The left action of $G$ on
$\Omega^1_{\hor}(P) \tensor \g$ is
\[
g\cdot(\alpha \tensor \xi)= R_{g}^{\ast}\alpha \tensor \Ad_{g} \xi.
\]
The \textbf{gauge group} $\cG$ of $P$ is the group of smooth maps $f \maps
P \to G$ satisfying
\[
R_{g}^{\ast}f(p) = g^{-1}f(p)g \quad \forall g \in G. 
\]
Such a map $f$ can be identified with a $G$-equivariant map $\phi \maps
P \to P$ covering $\id_{M}$ where
\[
\phi(p) = R_{f(p)} (p). 
\]
The gauge group acts on the space of connections $\cA$ from the left: 
\begin{equation} \label{gauge_action}
f\cdot A= \Ad_f A + (f^{-1})^*\theta_L,
\end{equation}
where $\theta_L\in \Omega^1(G,\g)$ is the left invariant Maurer-Cartan
form on $G$ and $f^{-1}$ is the composition of $f$ with the inversion on $G$. If $\phi$ is the bundle automorphism associated to $f$,
then the action is simply $\phi\cdot A=(\phi^{-1})^{\ast}A$. 

To obtain the infinitesimal analog of the above,  we consider maps $X \maps P \to \g$
satisfying
\[
R_{g}^{\ast}X(p) = \Ad_{g^{-1}}X(p)\quad \forall g \in G. 
\] 
The space of all such maps forms the \textbf{Lie algebra
 of infinitesimal gauge transformations} $\Lie(\cG)$. This plays the
role of the Lie algebra associated to $\cG$. 
\begin{remark}
Indeed, $\cG$ is a locally
exponential Lie group modeled on the Lie algebra $\Lie(\cG)$
\cite [Thm.\ 1.11]{Wockel:2007}.
This means that for each $X \in \Lie(\cG)$ the initial value problem 
\[
\gamma(0)=e_{\cG}, \quad \gamma(t)^{-1} \cdot \gamma^{\prime}(t)=X
\]
has a solution $\gamma_{X} \in \cinf(\R, \cG)$, and there exists a unique smooth exponential map
\[
\exp \maps \Lie(\cG) \to \cG, \quad X \mapsto \gamma_{X}(1)
\]
and an open neighborhood $0 \in W \subset \Lie(\cG)$ such that $\exp
\vert_{W}$ is a diffeomorphism onto some open neighborhood of the identity $e_{\cG}$.
\end{remark}
Differentiating the
action \eqref{gauge_action} gives an action of $\Lie(\cG)$ on
$\cA$. Specifically, given $X \in \Lie(\cG)$ we define the
fundamental vector
field $V_{X} \maps \cA \to (\Omega^1_{\hor}(P) \tensor \g)^G$ by
\begin{equation} \label{eq:inf_gauge}
V_{X}(A) = \frac{d}{dt} \left(\exp(-tX) \cdot A\right) \vert_{t=0}.
\end{equation}
Note that the assignment $X \mapsto V_{X}$ is a Lie algebra morphism
from $\Lie(\cG)$ to $\X(\cA)$. Also, a simple calculation shows
\[
V_{X}(A) = d_{A}X.
\]
where $d_{A} \maps (\Omega^{\bullet}_{\hor}(P) \tensor \g)^G \to
(\Omega^{\bullet+1}_{\hor}(P) \tensor \g)^G$ is the 
\textbf{(exterior) covariant derivative} 
\[
d_{A}\alpha = d\alpha + [A,\alpha].
\]

\subsection{Closed forms from invariant polynomials}
An invariant polynomial $q \in S^{n+1}(\g^{\vee})^{G}$ gives a
constant, hence closed, $(n+1)$-form on the space of connections
$\cA$. To see this, first consider the following $(n+1)$-form on $P$:
\[
q(\alpha_{1},\ldots,\alpha_{n+1}) \in \Omega^{n+1}(P),
\]
where each $\alpha_{i}$ is in $(\Omega^1_{\hor}(P)\otimes \g)^G$. This form
clearly vanishes when contracted with any vertical vector on $P$. Moreover, the $\Ad$
invariance of $q$ combined with the $G$ invariance of $\alpha_{i}$
implies that  $q(\alpha_{1},\ldots,\alpha_{n+1})$
is invariant under the action of $G$ on $P$. Therefore,
$q(\alpha_{1},\ldots,\alpha_{n+1})$ is \textbf{basic} i.e.\ it corresponds to the
pullback of a unique $(n+1)$-form on $M$ along $\pi \maps P \to M$.
We ``abuse notation'' by also denoting this $(n+1)$-form on $M$
as $q(\alpha_{1},\ldots,\alpha_{n+1})$. By integration, we then obtain a closed 
$(n+1)$-form on $\cA$:
\begin{equation}\label{eq:oma}
\omega(\alpha_1,\dots,\alpha_{n+1}) \vert_{A}=\int_{M}
q(\alpha_1,\dots,\alpha_{n+1}) \quad \forall \alpha_i\in T_{A}\cA=(\Omega^1_{\hor}(P)\otimes \g)^G.
\end{equation}
The following proposition shows that, for some cases, $\omega$ is in fact $n$-plectic.
\begin{prop}\label{invpoly_non-degen_prop}
If the invariant polynomial $q \in  S^{n+1}(\g^{\vee})^{G}$ is
non-degenerate, then $\omega$ \eqref{eq:oma}  is an $n$-plectic structure on $\cA$.
\end{prop} 
\begin{proof}
Given $\beta \in (\Omega^1_{\hor}(P)\otimes \g)^G$ such that
\[
\int_{M} q(\beta,\alpha_2,\dots,\alpha_{n+1}) =0  \quad \forall \alpha_{i} \in (\Omega^1_{\hor}(P)\otimes \g)^G,
\]
we assume, in order to lead to a contradiction, that there exists $p
\in P$ such that $\beta \vert_{p} \neq 0$. Let $U
\subseteq M$ be a chart containing $y=\pi(p)$ admitting a
trivialization $\tau \maps \pi^{-1}(U) \xto{\sim} U \times G$ such that
$\tau(p)=(y,e)$. Let $x^{1},\hdots,x^{n+1}$ be coordinates
on $U$. Working locally over $\pi^{-1}(U)$, and implicitly using the trivialization, we write $\beta =
\sum^{n+1}_{i=1} \beta_{i} d \pi^{\ast}x^{i}$ where $\beta_{i} \maps U \times G
\to \g$ satisfies $\beta_{i}(x,g) = \Ad_{g^{-1}} \beta_{i}(x,e)$. By
our assumption, there exists an $i$ such that $\beta_{i}(y,e) \neq
0$. Without loss of generality, we may further assume $i=1$.

Since $q$ is non-degenerate, there exists
$\xi_{2},\ldots,\xi_{n+1} \in \g$ such that $q(\beta_{1}(y,e),\xi_{2},\ldots,\xi_{n+1}) > 0$.
Hence, there exists a smaller neighborhood $V \subseteq U$ containing
$y$ such that
\[
q(\beta_{1}(x,e),\xi_{2},\ldots,\xi_{n+1}) > 0 \quad \forall x \in V.
\]
Define $\g$-valued maps $f_{2},\hdots,f_{n+1}$ on $\pi^{-1}(V)$ by
\[
f_{i}(x,g) = \Ad_{g^{-1}} \xi_{i}. 
\]
Finally, let $\varphi \maps M \to [0,1]$ be a ``bump function'' whose support
is contained in $V$.

Using all of this, we can define global $\g$-valued 1-forms
$\alpha_{2},\ldots,\alpha_{n+1}$ on $P$ by
\[
\alpha_{i} = \pi^{\ast}(\varphi)  d \pi^{\ast}(x^{i}) \tensor f_{i}.
\]
By construction, each $\alpha_{i}$ is in
$(\Omega^1_{\hor}(P)\tensor\g)^G$, and therefore we have a contradiction:
\[
0=\int_{M} q(\beta,\alpha_2,\dots,\alpha_{n+1}) =
\int_{\supp \varphi} q(\beta_{1}(x,e),\xi_{2},\ldots,\xi_{n+1})
dx^{1} dx^{2} \cdots dx^{n+1} > 0
\]
This implies that $\beta$ is zero. Hence, $\omega$ is non-degenerate.
\end{proof}

\subsection{The moment map}
From here on, we assume the following:
\vspace{.5cm}
\begin{itemize}
\item{The principal $G$-bundle $P \to M$ admits a flat connection. }
\end{itemize}
\vspace{.5cm}
We equip the space of connections $\cA$ with the closed
$(n+1)$-form $\omega$ given in Eq.\  \eqref{eq:oma}.

We begin by considering the linear map $\mu \colon \Lie(\cG)\to
\Omega^{n-1}(\cA)$ defined as
\begin{equation}\label{eq:mux}
\mu(X) (\alpha_1,\dots,\alpha_{n-1}) \vert_{A} =\int_{M} q(F_{A},\alpha_1,\dots,\alpha_{n-1},X),
\end{equation}
for all $X\in \Lie(\cG), A \in \cA$ and $\alpha_i\in T_{A}\cA$. Here  $F_{A}$ is the curvature of the connection $A$: 
\[
F_{A}=dA+\frac{1}{2}[A,A]\in (\Omega_{\hor}^2(P) \tensor\g)^{G}.
\]
The curvature is horizontal and it satisfies 
$R_g^*F_{A}=\Ad_{g^{-1}}F_{A}$. So, the $(n+1)$-form $q(F_{A},\alpha_1,\dots,\alpha_{n-1},X)$ on $P$
descends to a unique $(n+1)$-form on $M$.

\begin{prop}\label{noteworthy}
The map $\mu \maps \Lie(\cG)\to \Omega^{n-1}(\cA)$ defined by Eq.\  \eqref{eq:mux}
is $\cG$-equivariant, i.e.\
\[
(f^{-1})^*\mu(X)=\mu(\Ad_{f} X) \quad \forall f\in \cG, X \in \Lie(\cG).
\]
\end{prop}
\begin{proof}
Let  $A \in \cA$ and $\alpha_1,\dots,\alpha_{n-1} \in T_{A}\cA$.
We have 
\begin{equation} \label{equimu}
(f^*\mu(X))(\alpha_1,\dots,\alpha_{n-1} )\vert_A=
\mu(X)(f_*\alpha_1,\dots,f_*\alpha_{n-1})\vert_{f\cdot A}=
\int_{M} q(F_{f\cdot A},f_*\alpha_1,\dots, f_*\alpha_{n-1},X).
\end{equation}
A straightforward computation shows that $F_{f\cdot A}=\Ad_fF_{A}$.
It is also not difficult to show that the differential $f_{\ast}$ of the map $A
\mapsto f \cdot A$ is  $f_*\alpha_i=\Ad_{f}\alpha_{i}$.
Hence, from the $\Ad_f$-invariance of $q$, we see that the right-hand
side of Eq.\  \eqref{equimu} is
\[
\mu(\Ad_{f^{-1}}X)(\alpha_1,\dots,\alpha_{n-1})\vert_{A}.
\]
\end{proof}

Next, we show the image of $\mu$ lies in Hamiltonian forms. The
associated Hamiltonian vector fields are those induced by the infinitesimal
gauge transformations \eqref{eq:inf_gauge}.

\begin{prop}\label{lem:dm}
If $X \in \Lie(\cG)$, then $\mu(X)$ is a Hamiltonian $(n-1)$-form with Hamiltonian vector field $V_{X}$, where
\[
V_{X}\vert_{A} = d_{A} X = dX + [A,X] \quad \forall A \in \cA.
\]
\end{prop}

To prove this, we will use the following lemma.  
\begin{lemma}\label{lem:techAB} 
If $q \in S^{k}(\g^{\vee})^{G}$ is a degree ${k}$ invariant polynomial,
$A\in \cA$ is a connection, and  $\beta_{1},\ldots,\beta_{k} \in
(\Omega_{\hor}^{\bullet} (P) \tensor \g)^{G}$ are forms with
$\deg{\beta_{1}} +  \deg{\beta_{2}} + \cdots + \deg{\beta_{k}} =n$,
then
\[
\sum_{i=1}^{k} (-1)^{|\beta_1|+\dots+|\beta_{i-1}|}
\int_{M} q( \beta_1,\dots,d_{A}\beta_{i},\dots, \beta_{k})=0,
\]
where the above sign for $i=1$ is defined to be $+1$.
\end{lemma}
\begin{proof}
By replacing $d_{A}\beta_{i}$ by $d\beta_{i} + [A,\beta_{i}]$ for
all $i$ in
\begin{equation} \label{lem:teachAB_eq}
\sum_{i=1}^{k} (-1)^{|\beta_1|+\dots+|\beta_{i-1}|} q( \beta_1,\dots,d_{A}\beta_{i},\dots, \beta_{k})
\end{equation}
we rewrite the above as the sum of two basic $(n+1)$-forms on $P$:
\[
d \bigl( q( \beta_1, \beta_2,\ldots, \beta_{k}) \bigr ) + \left (\sum_{i=1}^{k} (-1)^{|\beta_1|+\dots+|\beta_{i-1}|}
q( \beta_1,\dots,[A,\beta_{i}],\dots, \beta_{k}) \right),
\]
where the summation over $i$  is, in fact, zero by the infinitesimal
$G$-invariance of $q$. Hence, \eqref{lem:teachAB_eq} descends to an
exact $(n+1)$-form on $M$, and so its integral vanishes by Stokes' theorem.
% $\int_{M} d(q( \beta_1,\dots,
% \beta_{k}))+
% \sum_{i=1}^{k} (-1)^{|\beta_1|+\dots+|\beta_{i-1}|}
% \int_{M} q( \beta_1,\dots,[A,\beta_{i}],\dots, \beta_{k}).$
% The first term is zero, since $q( \beta_1,\dots, \beta_{k})$ is a
% basic form and using Stokes theorem on $M$, while the second term is
% zero
% because of the  (infinitesimal) $G$-invariance of $q$.}
% % \nrt{Replace $d_{A}\beta_{i}$ by $d\beta_{i} + [A,\beta_{i}]$},
% % and use Stokes theorem on $M$ to deal with the first summand and
% % the (infinitesimal) $G$-invariance of $q$ for the second
\end{proof}

\begin{proof}[Proof of Prop.\ \ref{lem:dm}]
We need to show $d\mu(X)=-\iota(V_X)\omega$ for all $X \in \Lie(\cG)$.
Let $A \in \cA$. Given tangent vectors $\alpha_1,\dots,\alpha_{n}
\in T_{A}\cA$, we denote by the same symbols their extension to
constant vector fields on $\cA$. Hence, $[\alpha_{i},\alpha_{j}]=0$
for all $i$ and $j$, and so the de Rham differential becomes
\begin{align*}
d\mu(X)(\alpha_1,\dots,\alpha_{n})&=\sum_i(-1)^{i+1}\cL_{\alpha_i}(\mu(X)(\alpha_1,\dots,\widehat{\alpha_i},\dots,\alpha_{n})).
\end{align*}
The identity
\[
\Bigl. \frac{d}{dt}F_{A +t \alpha_i}  \Bigr \rvert_{t=0}=d_{A}\alpha_i,
\]
combined with Lemma \ref{lem:techAB} imply that
\begin{align*}
  d\mu(X)(\alpha_1,\dots,\alpha_{n}) &=\sum_i(-1)^{i+1} \int_{M}
  q(d_{A}\alpha_i,\alpha_1,\dots,\widehat{\alpha_i},\dots,\alpha_{n},X)\\
&=-(-1)^n\int_{M} q( \alpha_1,\dots,\alpha_{n},d_{A}X))\\
&=- (\iota(V_X)\omega)( \alpha_1,\dots,\alpha_{n}).
\end{align*}
\end{proof}

The main result of this section is the following theorem.
Its proof features the obstruction theory developed in Sec.\ \ref{obstruct_sec}. 
\begin{thm} \label{moduli_mmap_thm}
There exists a homotopy moment map 
\[
(f_{i}) \maps \Lie(\cG) \to L_{\infty}(\cA,\omega)
\]
lifting the action of $\Lie(\cG)$ on $(\cA,\omega)$ such that
\[
f_{1}(X)(\alpha_1,\dots,\alpha_{n-1}) \vert_{A} = \mu(X)(\alpha_1,\dots,\alpha_{n-1}) \vert_{A} =\int_{M} q(F_{A},\alpha_1,\dots,\alpha_{n-1},X)
\]
for all $A \in \cA$ and $\alpha_{i} \in T_{A}\cA$.
\end{thm}
\begin{proof}
  By Remark \ref{inf-dim_remark}, the theorem is proved
  if we can show that the assumptions listed in Thm.\ \ref{thm:momapyes} are
  satisfied by taking the linear map $\phi$ (defined there) to be $\mu$.
  Let $A_{0} \in \cA$ be a flat connection.  We identify
  $\cA$ with the locally convex topological vector space
  $(\Omega^1_{\hor}(P) \tensor \g)^G$ so that $A_{0}$ corresponds to
  the origin. The Poincare Lemma holds \cite[Lem.\
  1.4.1]{Brylinski_book},  and hence the de Rham cohomology of $\cA$ is
  $\R$ in degree 0, and trivial in all higher degrees. Next, observe
  that the Lie algebra cocycle
\[
c^{\Lie(\cG)}_{A_{0}}(X_{1},\ldots,X_{n+1}) = \pm \omega(V_{X_{1}},\ldots,V_{X_{n+1}})
\vert_{A_{0}} = \int_{M} q(d_{A_{0}}X_{1},\ldots,d_{A_{0}}X_{n+1}).
\]
introduced in Cor.\ \ref{cocycle_cor} is trivial. Indeed, Lemma \ref{lem:techAB} implies that
\[
\int_{M} q(d_{A_{0}}X_{1},\ldots,d_{A_{0}}X_{n+1}) = \sum^{n+1}_{i=2}
\pm \int_{M} q( X_{1},\dots,d^{2}_{A_{0}}X_{i},\dots,
d_{A_{0}}X_{n+1}),
\]
and, since $A_{0}$ is flat, we have $d^{2}_{A_{0}}X_{i} = [F_{A_{0}},X_{i}] =0$.
%Therefore, we conclude that $c^{\Lie(\cG)}_{A_{0}}(X_{1},\ldots,X_{n+1})=0$ for all
%$X_{i} \in \Lie(\cG)$.
\end{proof}

\subsection{Reduction}\label{subsec:red}
We equip $(\cA,\omega)$ with a moment map given by Thm.\
\ref{moduli_mmap_thm}, and now describe a type of
Marsden-Weinstein reduction. As we shall see, when $\omega$ is
non-degenerate, the quotient of a zero-level set gives a pre-$n$-plectic structure on the moduli space of flat connections.

We denote by $\cAf \subset \cA$ the set of flat connections.
If $A$ is a flat connection, then the tangent space $T_{A}\cAf$
consists of all vectors $\alpha \in T_{A} \cA$  that satisfy $d_{A}
\alpha =d\alpha + [A,\alpha]=0$. We also consider the following ``zero
level set'' of $\mu$ :
\[
\cC =\{A \in \cA ~ \vert ~ \mu(X)\vert_{A}=0
~ \text{for all $X\in \Lie(\cG)$}\}. 
\]

The $\cG$-action restricts to $\cAf$, and by Prop.\
  \ref{noteworthy}, it also restricts to $\cC$.
It follows from the definition of $\mu$ that
\[
\cAf \subseteq \cC.
\]

At least in certain cases, the two subspaces are equal, as the following
proposition demonstrates.

\begin{prop}\label{lem:Aflat}
Let $q \in  S^{n+1}(\g^{\vee})^{G}$ be the invariant polynomial used
in constructing $\omega \in \Omega^{n+1}(\cA)$.
If $q$ is non-degenerate, then $\cAf = \cC.$
\end{prop}
\begin{proof}
Let $A \in \cC$ and assume, in order to lead to a contradiction, that
there exists $p \in P$ such that $F_{A} \vert_{p} \neq 0$.
We proceed as we did in the proof of Prop.\
\ref{invpoly_non-degen_prop}, replacing there the 1-form $\beta$ with
the 2-form $F_{A}$. The non-degeneracy allows us to construct 1-forms 
$\alpha_{1},\ldots\alpha_{n-1} \in (\Omega^1_{\hor}(P)\otimes \g)^G$
and element $X \in \Lie(\cG)$ such that
\[
\int_{M} q(F_{A},\alpha_1,\dots,\alpha_{n-1},X) >0.
\]
Hence, we have $\mu(X)(\alpha_1,\dots,\alpha_{n-1}) \vert_{A} >0$,
which contradicts $A \in \cC$.
\end{proof} 
 
Next, we consider pre-$n$-plectic structures induced on $\cAf$ and $\cC$.
\begin{prop}
If  $\iota \maps \cAf \hookrightarrow \cA$ is the inclusion, then 
\[
V_X \in \ker \iota^* \omega 
\]
for all $X \in \Lie(\cG)$.

\end{prop}
\begin{proof}
Let $X \in \Lie(\cG)$ and $A \in \cAf$. Fix $\alpha_2,\dots,\alpha_{n+1} \in T_{A}\cAf$. 
We have the equalities
\[
\omega(V_{X},\alpha_{2},\ldots,\alpha_{n+1}) \vert_{A} = 
\omega(d_{A}X,\alpha_2\dots,\alpha_{n+1})=\int_{M}q(d_{A}X,\alpha_2,\dots,\alpha_{n+1}).
\]
Hence, Lemma \ref{lem:techAB} implies that
\[
\omega(V_{X},\alpha_{2},\ldots,\alpha_{n+1}) \vert_{A} = 
\sum^{n+1}_{i=2} \pm \int_{M} q( X,\dots,d_{A} \alpha_{i},\dots,\alpha_{n+1}).
\]
The right-hand side of the above is zero, since $d_{A} \alpha_{i}=0$ for
all $\alpha_{i}$.
\end{proof}
Recall that Prop.\ \ref{lem:dm} implies that 
$\iota^{\ast}\omega$ is $\cG$-invariant. This fact combined with the above
two propositions gives the following:
\begin{cor}
If the quotient space $\cAf/ \cG$ is a smooth manifold, then it
carries a closed $(n+1)$-form induced by $\omega$. In addition, if
$\omega$ is non-degenerate, then $\cAf/ \cG=\cC/ \cG$.
\end{cor}
% We say that of $(S^{n}\g^{\vee})^G$ is non-degenerate if, viewing it as a  symmetric, equivariant map $p\colon   \g^{\otimes n} \to \RR$, the injectivity
% $\g \to (\g^{\vee})^{\otimes n-1}, \xi \mapsto \iota(\xi)p$ holds. 
% \begin{lemma}
% If $p$ is non-degenerate, then   $\omega$ is a non-degenerate $n$-form on $\cA$ 
% \end{lemma} 
% \begin{proof}
% One proceeds as in the proof of Lemma \ref{lem:Aflat}.
% \end{proof}
 
Unfortunately, if  $\omega$ is non-degenerate, then it
does \emph{not} follow that the $(n+1)$-form on $\cC / \cG$  is non-degenerate, as this example shows.

\begin{ep}

  We consider the simplest possible ``higher'' case.  Let $M$ be of
  dimension $3$, $G=\RR$, and $P$ the trivial bundle $\RR\times M \to
  M$.  Let $p \colon \RR
  \times \RR\times \RR\to \RR$  be just the multiplication of three
  numbers; it is clearly a non-degenerate invariant polynomial on the Lie algebra.

Since the bundle is trivial, we can write our data as
\begin{itemize}
\item $\cA=\Omega^1(M)$\text{, a vector space}
\item $\cG= C^{\infty}(M)$
\item $\Lie(\cG)=C^{\infty}(M)$
\item $V_X=dX$ for all $X\in \Lie(\cG)$
\item $\omega(\alpha_1,\alpha_2,\alpha_3)
  \vert_{A}=\int_{M}\alpha_1 \wedge \alpha_2 \wedge \alpha_3$
\end{itemize}
Clearly $\cA_{flat}=\Omega_{\mathrm{closed}}^1(M)$, again a vector space, and 
$\cAf/\cG=H^1(M)$.

The induced 3-form on $H^1(M)$ is the just the evaluation on the fundamental class of $M$:
$$H^1(M)\otimes H^1(M) \otimes H^1(M) \mapsto   \RR,\;\; a_1\otimes a_2\otimes a_3\mapsto \langle a_1\wedge a_2\wedge a_3, [M]\rangle.$$
Therefore it is non-degenerate iff
 for all $a_1\in H^1(M)$,
\begin{equation}\label{eq:123}
a_1a_2a_3=0 \text{ for all $a_2,a_3\in H^1(M)$}
\Rightarrow a_1=0. 
\end{equation}
In general, condition \eqref{eq:123} is not satisfied.
Recall that the pairing $H^1(M)\otimes
H^2(M)\to H^3(M)\cong \RR$ is non-degenerate 
by Poincar\'e
duality. Hence, the non-degeneracy condition \eqref{eq:123} is satisfied if{f} the
product map
$$H^1(M)\otimes H^1(M) \to H^2(M)$$ is surjective.

So, for example, if $M=S^2\times S^1$, then condition
\eqref{eq:123} is not satisfied, and indeed $H^1(S^2\times S^1)\cong
\RR$ must carry the zero 3-form. When $M=S^1\times S^1\times S^1$ is
the torus, condition \eqref{eq:123} is satisfied, and $H^1(M)\cong
\RR^3$ carries a constant volume form.
\end{ep}

\section{Loop spaces} \label{loop_sec} In this section, we show how
homotopy moment maps for $G$-actions on pre-$2$-plectic manifolds
$(M,\omega)$ can be transgressed to ordinary moment maps on the
associated pre-symplectic loop space $LM$ (see
Thm.\ \ref{loop_momap}). Motivation for studying such actions arises, for example,
in topological field theory. There one can consider a group of symmetries
$G$ acting on a ``target space'' $M$ equipped with a closed form
$\omega$. The form can be transgressed to a mapping space
$\mathrm{Map}(X,M)$ i.e.\ the ``space of fields''. The induced $G$-action 
is then defined ``point-wise''. The results of this section
could be interpreted as an elementary example of this process for the
case $X=S^{1}$.  Roughly speaking, this demonstrates how the higher
symplectic geometry on $M$ can interact with the ordinary geometry on
$\mathrm{Map}(X,M)$.

\subsection{Actions on loop spaces}

For any manifold $M$, the free loop space $LM$, i.e.\ the space of smooth loops
\[
LM = \cinf(S^{1},M),
\]
is an infinite-dimensional Fr\'echet manifold (\cite{Brylinski_book}, \cite{Stacey} and  \cite{Pre-Seg}). The tangent space at
$\gamma \in LM$ can be identified with global sections of the pullback of
$TM$ along $\gamma \maps S^{1} \to M$, i.e.\
\[
T_{\gamma}LM = \cinf(S^{1},\gamma^{\ast}TM)
\]
There is a degree $-1$
chain map 
\[
\ell \maps \Omega^{\bullet}(M) \to \Omega^{\bullet -1}(LM)
\]
called \textbf{transgression}.
Explicitly, it sends a
$k$-form $\alpha$ on $M$ to the $(k-1)$-form 
$\alpha^{\ell}$ on $LM$ given by the formula
\[
\alpha^{\ell} \vert_{\gamma}(v_1,\ldots,v_{k-1}) = \int^{2\pi}_{0} 
\alpha(v_{1},\ldots,v_{k-1},\dot{\gamma}) \vert_{\gamma(s)} ~ ds
\quad \forall \gamma \in LM, ~ \forall v_{1},\ldots,v_{k-1} \in T_{\gamma}LM.
\]
Since transgression commutes with the de Rham differential, any
pre-$n$-plectic structure $\omega$ on 
$M$ gives a pre-$(n-1)$-plectic structure $\omega^{\ell}$ on $LM$.

Suppose $M$ is a manifold equipped with a $G$-action $G \times M
\to M$. This induces a ``point-wise''  action $G \times LM \to LM$ given by:
\[
(g \cdot \gamma)(s) = g \cdot\gamma(s) \quad \forall g \in G, ~
\forall \gamma \in LM.
\]
Given an element $x \in \g$ and a loop $\gamma \in LM$, we
obtain a smooth path in $LM$
\[
\R \ni t \mapsto \exp(-x t)\cdot \gamma, 
\]
and an action of $\g$ on $LM$ via the fundamental vector field
\[
v^{\ell}_{x} \vert_{\gamma} = \left. \frac{d}{dt}
  \exp(-x t)\cdot \gamma \right \rvert_{t=0} \quad \forall
\gamma \in LM.
\]
We then observe, by differentiation,
that the fundamental vector field $v^{\ell}_{x}$ on $LM$ evaluated at
a point $\gamma \in LM$ is just $\gamma^{\ast}v_{x}$
(a section of $\gamma^{\ast}TM \to S^{1}$), where $v_{x}$ is the
fundamental vector field on $M$ associated to $x$.  

More generally, restricting vector fields on $M$ to loops in $M$ gives
us a map
\begin{equation*}
\begin{split}
\Gamma(TM) &\to \Gamma(TLM)\\
v &\mapsto v^{\ell}, 
\end{split}
\end{equation*}
defined by $v^{\ell} \vert_{\gamma}=\gamma^{\ast}v$. 
That $v^{\ell}$ is a smooth vector field with respect to the induced
smooth structure on $TLM$  follows from the fact that it is
the composition of two smooth maps. The first of these is a map $LM
\to LTM$ given by $ \gamma \mapsto v \circ \gamma$ \cite[Thm.\ 3.27]{Stacey}.
The second map is the natural diffeomorphism of vector bundles $LTM \cong TLM$
covering the identity on $LM$ \cite[Thm.\ 4.2]{Stacey}.
% comes from the fact that if $f:M\to N$ is smooth, then the induced
% map $Lf:LM\to LN$ which sends $\gamma \in LM$ to $f\circ \gamma\in
% LN$ is smooth (Thm.\ 3.27 \cite{Stacey}). The claim follows if we
% view a vector field $v$ on M as a map $f : M\to N$, with $N=TM$, and
% use the isomorphism $ LTM\simeq TLM$

\subsection{Actions on pre-symplectic loop spaces} 
Suppose $(M,\omega)$ is a pre-2-plectic manifold.
% equipped with a
%$G$ action such that $\g$ acts via Hamiltonian vector fields.
%By the discussion in the previous section, 
Then
$(LM,\omega^{\ell})$ is a
pre-symplectic manifold.
% equipped with an action by $G$. The 2-form
%$\omega^{\ell}$ is $G$-invariant. In fact, we have:
We have:
\begin{prop} \label{loops_ham}
If   $\alpha$ is a Hamiltonian 1-form with
Hamiltonian vector field  $v$, then 
%the fundamental 
the vector field $v^{\ell}$ is
Hamiltonian for the function $\alpha^{\ell} \maps LM \to \R$.
% defined as
%\[
%F(\gamma)=-\alpha^{\ell}(\gamma).
%\]
\end{prop}
\begin{proof} By assumption we have $d\alpha =-\iota (v) \omega$, and 
we wish to show 
\[
d\alpha^{\ell}=-\iota (v^{\ell}) \omega^{\ell}.
\]
Let $\gamma \in LM$. For all $u \in T_{\gamma}LM$, we have
\begin{align*}
(\iota_{v^{\ell}} \omega^{\ell})(u) &= 
\int^{2\pi}_{0} 
\omega( v^{\ell},u,\dot{\gamma}) \vert_{\gamma(s)} ~ ds\\
%&=\int^{2\pi}_{0} \omega(\dot{\gamma}, v ,u)
%\vert_{\gamma(s)} ~ ds\\
%&=-\int^{2\pi}_{0} \omega(v,\dot{\gamma} ,u) \vert_{\gamma(s)} ~ ds\\
&=-\int^{2\pi}_{0} d\alpha(u,\dot{\gamma}) \vert_{\gamma(s)} ~ ds\\
&=- (d\alpha)^{\ell}(u)
%&=\iota(u)d\int^{2\pi}_{0} \alpha(\dot{\gamma}) \vert_{\gamma(s)}
%~ ds \\
%&= \iota(u)dF \vert_\gamma.
\end{align*}
\end{proof}
Now suppose that  $(M,\omega)$ is equipped with  a
$G$-action and with a homotopy moment map $\g \to
L_{\infty}(M,\omega)$. Let $f_{1} \maps \g \to \ham{1}$, and $f_{2}
\maps \g \tensor \g \to \cinf(M)$ be the corresponding structure maps for this moment map.
For $x \in \g$, the vector field $v_{x}$ is Hamiltonian for the 1-form $\alpha_{x}=f_{1}(x)$. The 2-form
$\omega^{\ell}$ on $LM$ is $G$-invariant, and the $v_x^{\ell}$ are Hamiltonian.
The next theorem says that the homotopy moment map on $(M,\omega)$ transgresses
to an ordinary moment map for the action of $G$ on  $(LM,\omega^{\ell})$ 

\begin{thm} \label{loop_momap}
If $(M,\omega)$ is a pre-2-plectic manifold equipped with a
$G$-action and a homotopy moment map
$f \colon \g \to L_{\infty}(M,\omega)$, then  
 \begin{align*}
 \psi : \g & \to \cinf(LM)\\
     x & \mapsto (f_1(x))^{\ell}
  \end{align*} is a  moment map for 
the induced action of $G$ on the pre-symplectic loop space $(LM,\omega^{\ell})$ 
\end{thm}

\begin{proof}
%Let $\psi \maps \g \to \cinf(LM)$ be the linear map
%\[
%\psi(x)=F_{x},
%\]
%where $F_{x}$ is the function defined in the previous proposition. 
To
prove the theorem, it is sufficient to show that $\psi$
is a Lie
algebra morphism.

The bracket of the functions $\psi(x)$ and $\psi(y)$ evaluated at
$\gamma \in LM$ is:
\begin{align*}
\brac{\psi(x)}{\psi(y)} \vert_{\gamma} &= \int^{2\pi}_{0} \omega(v_{x} ,v_{y},\dot{\gamma})
\vert_{\gamma(s)} ~ ds\\
&=\int^{2\pi}_{0} \iota_{\dot{\gamma} }l_{2}(f_{1}(x),f_{1}(y)) \vert_{\gamma(s)} ~ ds\\
&=\big(l_{2}(f_{1}(x),f_{1}(y))\big)^{\ell}\vert_{\gamma},
\end{align*}
where $l_{2}$ is the bi-linear bracket for the Lie 2-algebra $L_{\infty}(M,\omega)$.
On the other hand, $\psi([x,y])$ is, by definition, equal to the
transgression of the 1-form $f_{1}([x,y])$. The definition of an
$L_{\infty}$-morphism implies that
\begin{equation}\label{eq:fl}
l_{2}(f_{1}(x),f_{1}(y))- f_{1}([x,y]) = df_{2}(x,y).
\end{equation}
Since $df_{2}(x,y)$ is an exact 1-form, Stokes theorem implies for all
$\gamma \in LM$:
\[
(d(f_{2}(x,y)))^{\ell}=\int^{2\pi}_{0} \iota_{\dot{\gamma}}d(f_{2}(x,y)) \vert_{\gamma(s)} ds =0
\]
Hence applying the transgression operator $\ell$ to Eq.\  \eqref{eq:fl} we obtain
\[
\brac{\psi(x)}{\psi(y)} - \psi([x,y])=0.
\]
\end{proof}

{ 
\begin{remark} The map $$\Omega^1_{\mathrm{Ham}}(M)\to C^{\infty}_{\mathrm{Ham}}(LM),\;\; \alpha \mapsto \alpha^{\ell}$$
% $\alpha \mapsto \alpha^{\ell}$ defined in Lemma \ref{loops_ham} 
is well-defined by  Prop.\ \ref{loops_ham} and preserves (binary) brackets, as shown in the  proof of Thm.\ \ref{loop_momap}. Therefore   
%If $\omega$ is a closed 3 form on $M$, then
there is a strict  $L_{\infty}$-morphism from $\li(M,\omega)$ to
$\li(LM,\omega^{\ell})$, whose only non-vanishing component is the map $\alpha \mapsto \alpha^{\ell}$.
Hence, given a homotopy moment map for $(M,\omega)$, composing with
the above $L_{\infty}$-morphism we get a homotopy moment map for
$(LM,\omega^{\ell})$. In \cite{FLRZ} we extend Thm.\ \ref{loop_momap}  further.

%to arbitrary pre-$n$-plectic manifolds.
% This follows since $L_{\infty}$-morphisms are closed under
% composition, and using Prop.\ \ref{loops_ham}. In conclusion, we
% recover Thm.\ \ref{loop_momap}.
\end{remark}
}

\section{Relation to other work} \label{others_sec}

 \subsection{Other notions of moment map}

Let $M$ be a manifold endowed with a closed $n+1$-form $\omega$, and $G$ a Lie group acting on $M$ preserving $\omega$.

{ A \textbf{multimomentum map}  in the sense of  Cari\~nena-Crampin-Ibort \cite[Sec.\ 4.2]{IbortCarCrampMultimomap} is a map $f_1 \colon \g \to \ham{n-1}$ satisfying \[
-\iota_{v_x} \omega=d(f_{1}(x))\quad\quad \text{   for all $x\in g$}.
\]
Such maps are called  \textbf{covariant momentum maps} 
 in \cite{GIMM}; they are used there to study symmetries in classical field theories.
Hence, if $(f_{k})$ is a homotopy moment map, then
its first component  is a multimomentum/covariant momentum map in the above sense.

Madsen-Swann
\cite[Sec.\ 3]{MadsenSwannClosed} consider the {\bf {\boldmath $n$}-th Lie kernel} $P_n$ (a $\g$-submodule of $\Lambda^n\g$), and define a {\bf multi-moment map}  as an equivariant map $\nu \colon M\to P_n^*$ such that $\iota(v_p)\omega=d(\nu^*p)$ for all $p \in P_n$. It is clear from Eq.\  \eqref{main_eq_1} (for $k=n$), that if $(f_i) \maps
\g \to L_{\infty}(M,\omega)$ is an equivariant homotopy moment map, then
the formula $\nu^*p=-\vs(n)f_n(p)$ for all $p\in P_n$
defines  a multi-moment map $\nu$.}

For the sake of clarity we spell out   the case  $n=2$, which was worked out first in  \cite[Sec.\ 2]{MadsenSwannMultimomap}.
In that case  $P_2=\ker([\cdot,\cdot]\colon \Lambda^2\g\to \g)$.
Notice that if $x,y\in \g$ are {commuting} elements, i.e. $x\wedge y \in P_2$, then the
invariance of $\omega$ under $G$ implies that $\iota(v_x\wedge
v_y)\omega$ is a closed 1-form. A multi-moment map  is an equivariant map $\nu \colon M\to P_2^*$ such that $\iota(v_p)\omega=\sum_i\iota(v_{x_i}\wedge v_{y_i})\omega$ is exact with primitive $\nu^*p$, for any $p=\sum_ix_i\wedge y_i \in P$.

\subsection{Actions on Courant algebroids}
 Recall that a \textbf{Courant algebroid} consists of a vector bundle $E\to M$ with a non-degenerate symmetric pairing on the fibers, a bilinear bracket $[\![ \cdot,\cdot ]\!]$ on $\Gamma(E)$, and a bundle map $\rho \colon E\to TM$ satisfying certain conditions, see for instance \cite[Def.\  4.2]{Dima}. Courant algebroids appear naturally in the study of Dirac and generalized complex structures.
%\cite[Section 2.1]{BCG}

Let $(M,\omega)$ be a  pre-$2$-plectic manifold. There is a  Courant algebroid associated to the closed 3-form $\omega$, namely
 $TM\oplus T^*M$ with the natural pairing $\langle  X+\xi, Y+\eta \rangle=\xi(Y)+\eta(X)$, the projection $\rho$ onto the first factor, and 
  bracket  $$[\![ X+\xi, Y+\eta ]\!]_{\omega}=[X,Y]+\cL_X\eta-\iota_Y d\xi+\iota_Y\iota_X \omega.$$ This Courant algebroid  is sometimes called the $\omega$-twisted Courant algebroid, and 
we will denote  it by
   $(TM\oplus T^*M)_\omega$.

Bursztyn-Cavalcanti-Gualtieri \cite{BCG}  defined the notion of extended action on a Courant algebroid. We spell out only the case of a \textbf{trivially extended action} \cite[Def.\ 2.12]{BCG} on the above Courant algebroid: it is an action of a connected Lie group $G$ on $M$ 
together with a linear map $\xi \colon \g \to \Gamma(T^*M)$ such that
$\g \to \Gamma(TM\oplus T^*M)_{\omega}, x \mapsto v_x+\xi(x)$ is
bracket-preserving and the  Lie algebra morphism $x\mapsto
\ad_{v_x+\xi(x)} =[\![ v_x+\xi(x),\cdot ]\!]_{\omega}$
into the infinitesimal automorphisms of the Courant algebroid 
 integrates to an action of $G$ on $TM\oplus T^*M$.

The following lemma is essentially contained in   \cite[Sec.\ 2.2]{BCG}.
\begin{lemma}\label{fromBCG}
 Let $G$ act on pre-$2$-plectic manifold $(M,\omega)$ and $\mu \in (\g^{\vee}\otimes \Omega^1(M))^G$.
 $\omega-\mu$ is an equivariant extension of $\omega$ if and only
 if
$$\Psi \colon \g \to \Gamma(TM\oplus T^*M)_{\omega},\;\; x \mapsto
v_x-\mu(x)$$ is a trivially extended action on the Courant algebroid
${(TM\oplus T^*M)_\omega}$ 
integrating to the action of $G$  by  tangent-cotangent lifts
with isotropic image:
\[
\innerprod{\im \Psi}{\im \Psi}=0.
\]
\end{lemma}
\begin{proof} To simplify the notation, denote $\xi=-\mu$. Suppose  $\omega+\xi$ is an equivariant extension (so   Eq.\  \eqref{eq:dgclosed} holds). The following two statements are contained in the text after Prop.\ 2.11 of \cite{BCG}. First, the condition $\iota_{v_x}\omega-d(\xi(x))=0$ and the equivariance of $\xi$ are equivalent to the fact that $\Psi$ preserves brackets
and  $TM\oplus\{0\}$ is a $\g$-equivariant splitting.
Second, $\iota_{v_x}\omega-d(\xi(x))=0$ is also equivalent to 
$\ad_{v_x+\xi(x)}$ 
    being the Lie derivative  $\cL_{v_x}$, acting on vector fields and 1-forms (hence $\Psi$ integrates to the action of $G$ on $TM\oplus T^*M$ is by  tangent-cotangent lifts). The condition $\iota_{v_x}(\xi(x))=0$ clearly means that the image of $\Psi$ is isotropic.

The converse implication is proven by reversing the above argument.
\end{proof}

%%% 

On the other hand, we know that an equivariant extension $\omega-\mu$ delivers a moment map $(f_k) \colon \g \to \li(M,\omega)$, by Thm.\  \ref{main1}. When $\omega$ is non-degenerate, the relation between $f$ and the extended action $\Psi$ of Lemma \ref{fromBCG} is simply
\begin{equation}\label{pif}
\Psi= i\circ f.
\end{equation}
Here
$$(i_k) \colon \li(M,\omega) \to \li((TM\oplus T^*M)_\omega)$$
is the embedding  of Lie 2-algebras given in \cite[Thm.\ 5.2]{RogersCou}
\footnote{The same theorem appears also as \cite[Thm.\ 7.1]{RogersPre} but with different sign conventions.} (it is defined only when $\omega$ is non-degenerate).  $\li((TM\oplus T^*M)_\omega)$ denotes the Lie 2-algebra associated 
to the  Courant algebroid $(TM\oplus T^*M)_\omega$ by
Weinstein-Roytenberg \cite{rw}\cite[Thm.\ 6.5]{RogersPre}; its  underlying graded vector space is $C^{\infty}(M)$ in degree $-1$ and $\Gamma(TM\oplus T^*M)_{\omega}$ in degree zero, and the binary bracket in degree zero is the twisted Courant bracket, i.e. the skew-symmetrization of $[\![\cdot,\cdot ]\!]_{\omega}$.
In summary, the composition of the two Lie 2-algebra morphisms on the r.h.s. of Eq.\  \eqref{pif} happens to be a strict morphism from $\g$  to  $\li((TM\oplus T^*M)_\omega)$, whose only non-trivial component is $\Psi$.

\begin{remark}
There is a notion of moment map for extended actions on Courant
algebroids \cite[Def.\ 2.14]{BCG}. In the case of a trivially extended action, however,
the only such moment map is the zero map. 
\end{remark}

\begin{remark}
Whenever the action of $G$ on $(M,\omega)$ is by Hamiltonian vector fields,  i.e. there is a linear map $f \colon \g\to \Omega^1_{\Ham}(M)$ satisfying 
$\iota_{v_x}\omega=-df(x)$ for all $x$,  one has $\ad_{v_x+\xi(x)}=\cL_{v_x}$. Hence, given an arbitrary moment map $(f_k) \colon \g \to \li(M,\omega)$, the Lie algebra morphism into the infinitesimal automorphisms of the Courant algebroid  suggested by Lemma \ref{fromBCG}
$$\g\to \mathfrak{aut}((TM\oplus T^*M)_\omega), x\mapsto \ad_{v_x-f_1(x)}=\cL_{v_x}$$ 
 does not encode neither $\omega$ nor the moment map. This is a clear indication  the moment map   $(f_k)$ can not be encoded naturally in terms on a map of $\g$ into the infinitesimal automorphisms of the Courant algebroid.

\end{remark}

\subsection{Actions on differential graded manifolds}\label{action}
%!!!
Let $M$ be a manifold endowed with a closed $n+1$-form $\omega$.
Uribe \cite{UribeDGman} considers the graded manifold $P=T[1]M\oplus\RR[n]$, 
whose graded algebra of function is $\Omega(M)\otimes S[t]$,
where $S[t]$ denotes polynomials in a variable $t$ of degree $n$.
The graded manifold $P$ is endowed with the homological  vector field
$Q=d_{\dR}+\omega\partial_t$, where $d_{\dR}$ is the de Rham vector field on $T[1]M$. 
%$Q$ is a homological vector field: 
%it has degree $1$, and $[Q,Q]=0$. 
Assume that a connected Lie group $G$ acts on $M$, and assume for simplicity that  $\omega$ is preserved by the action.

Denote  $$\s\y\m(P,Q)=\X_{<0}(P)\oplus \{Y\in \X_0(P):[Q,Y]=0\}$$ (the vector fields of negative degree on $P$, together with the degree zero vector fields commuting with $Q$). It is a DGLA with the usual bracket of vector fields and differential $[Q,\bullet]$.
The main motivation for Uribe to study these objects is that, when $n=2$, $\s\y\m(P,Q)$ is isomorphic to the DGLA of symmetries of the $\omega$-twisted exact Courant algebroid over $M$.
$\s\y\m(P,Q)$ contains a sub-DGLA 
$\g\s\y\m(P,Q)$, whose degree zero component consists of vector fields preserving the function $\omega$, and with degree $-1$ component $$\{\iota_X+\alpha\partial_t: X\in \X(M),\alpha\in \Omega^{n-1}(M), d\alpha=-\iota_X\omega\}$$(see \cite[Sec.\ 3.2]{UribeDGman}).
Further, there is a DGLA associated to the Lie algebra $\g$ of $G$: it is\footnote{It is a DGLA concentrated in degrees $-1$ and $0$, which corresponds to the natural structure of Lie algebra crossed module of $\g$ over itself.} $\g[1]\oplus \g$, with bracket given by the bracket and adjoint action of
$\g$, and differential $\id_{\g}$.

We are now ready to reproduce two statements.
First, \cite[Lemma 3.7]{UribeDGman} states that strict morphisms of
DGLAs  
\begin{equation} \label{uribe_eq1}
\g[1]\oplus \g \to \g\s\y\m(P,Q) 
\end{equation}
lifting the action of $\g$ on $M$ are in bijective correspondence with cocycles $\omega-\mu$ in the Cartan model, where 
$\mu\in \Omega^{n-1}(M)\otimes S^1\g^{\vee}[-2]$.

Second, by   \cite[Prop.\ 2.15]{UribeDGman}, $L_{\infty}$-morphisms 
\begin{equation} \label{uribe_eq2}
\g[1]\oplus \g \to  \s\y\m(P,Q)
\end{equation}
lifting the map
$\g[1]\oplus \g \to \X(M)[1]\oplus \X(M)$ induced by the action of $\g$ on $M$ are in bijective correspondence with closed extensions of $\omega$ in the BRST model $\wedge\g^{\vee}\otimes S\g^{\vee}\otimes \Omega(M)$.

Notice that, as $\g\s\y\m(P,Q)$ is a sub-DGLA of $\s\y\m(P,Q)$, the
morphisms appearing in \eqref{uribe_eq1}
are particular cases of those appearing in \eqref{uribe_eq2}.

The relevance of the above to our work is as
follows. A cocycle $\omega-\mu$ in the Cartan model induces two kinds
of objects: by Thm.\ \ref{main1} it induces a moment map of a
particular form (equivariant, and determined by its unary component);
by \cite[Lemma 3.7]{UribeDGman} it induces an infinitesimal action of
$\g[1]\oplus \g$ on $(P,Q)$ of a particular form (strict, and
preserving $\omega$). Proposition 2.15 in \cite{UribeDGman}
suggests that there is a relation between arbitrary homotopy moment maps and
closed extensions of $\omega$ in the BRST model. We are pursuing such
ideas in future work.

\section{Concluding remarks} \label{end_sec}
This work raises many questions and suggests a variety of possible directions for future research.
We mentioned some of these throughout the text. Here we give a few more.

First, several of the examples introduced in this paper deserve
more thorough investigation. In particular:
\begin{itemize}
\vspace{.5cm}
\item{In Sec.\ \ref{sub:exact}, we consider moment maps
    arising from exact $n$-plectic forms. Manifolds equipped such
    $n$-plectic structures naturally arise in certain models of
    classical field theory \cite{GIMM}. How do homotopy moment maps
    relate to the conservation laws and symmetries studied within
    these models? 

    {Along these lines, let us
      just mention that if a Hamiltonian $(n-1)$-form $H$ is invariant
      under the $G$-action, the existence of a (not necessarily
      equivariant) moment map implies the existence of many conserved
      quantities, i.e., differential forms for which the Lie
      derivative by the Hamiltonian vector field $v_H$ is exact. In
      particular, the ``dynamics'' given by $v_H$ is constrained to the
      level sets of certain functions.}

}
\vspace{.5cm}
\item{
    In Sec.\ \ref{subs:conj}, we note a relationship between the
    homotopy moment map arising from a Lie group acting on itself via
    conjugation and the theory of quasi-Hamiltonian
    $G$-spaces. Moreover, if $G$ acts on a pre-2-plectic manifold
    $(M,\omega)$, with $\omega$ both $G$-invariant and representing a
    degree 3 integral cohomology class, then we also have reasons to suspect
    that there is a relationship between homotopy moment maps lifting the
    $G$-action, and $G$ equivariant $U(1)$-gerbes on $M$. 
   Indeed, the results in \cite{FRS}
      imply that a homotopy moment map lifts the $\g$-action on
      $(M,\omega)$ to a $\g$-action on any
      $U(1)$-gerbe whose 3-curvature is $\omega$.}
    Some relationships have already been established
    between trivializations of $G$-equivariant gerbes and
    quasi-Hamiltonian $G$-spaces (e.g.\ \cite{Gomi:2004}). What is the precise relationship
    which intertwines these formalisms with homotopy moment maps?
 
 \end{itemize}

Further development of the general theory of homotopy moment maps
would also be desirable. For example:

\begin{itemize}
\item{ 
    When should two moment maps be
    considered equivalent?  Indeed, there are well-known uniqueness
    results for moment maps in symplectic geometry and, for example,
    abstract moment maps in Hamiltonian cobordism theory. For us, the
    question is particularly relevant since some of our constructions
    (e.g.\ Thm.\ \ref{thm:momapyes}) depend on a number of choices. We
    briefly discussed uniqueness issues in Sec.\ \ref{unique_subsec}
    for the case of pre-2-plectic manifolds. Also, in Section 7 of \cite{FLRZ},
    various candidates for equivalences are proposed. One of these is closely
    related to the notion of homotopy equivalence between
    $L_{\infty}$-morphisms (e.g., \cite[Def.\ 4.7]{DHR:2015}).
} 
\vspace{.5cm}
\item {Can one perform reduction of pre-$n$-plectic forms using an equivariant moment map $(f_{k}) \maps  \g \to L_{\infty}(M,\omega)$, analogously
to the Marsden-Weinstein-Meyer reduction in symplectic geometry?
It is easily checked that, if $$S=\{p \in M: f_1(x)|_p=0 \text{ for all }x\in \g\}$$ satisfies certain regularity conditions, then it is $G$-invariant and the pullback of $\omega$ descends to $S/G$. However $S$ is the vanishing set of a family of $(n-1)$-forms, and it is hard to control its smoothness properties. The only instance we are aware of where the above-mentioned reduction procedure works, apart from the one discussed in Subsection \ref{subsec:red}, is   Ex.\ \ref{ex:ctlifts}, namely
the (free and proper) cotangent lift action of $G$ on $\Lambda^n
T^*N$: in this case $S/G$ is canonically isomorphic to $\Lambda^n
T^*(N/G)$ with its canonical $n$-plectic form. This procedure is
  probably too na\"ive in general, and it only uses a small part of the information provided by the moment map.

}
\end{itemize}
 
\appendix
\section{Explicit formulas for $L_{\infty}$-morphisms}
Here we recall the relationship between
$L_{\infty}$-morphisms and morphisms of dg coalgebras. We use
this to prove Prop.\ \ref{Lie_alg_P_prop1} from Section
\ref{Linfty_sec} and other results needed throughout the text.
The facts presented here are well-known to experts, but it is
difficult to find explicit formulas in the literature. 
For more details on coalgebras, we suggest Sections 3d and 22a of \cite{FHT}, 
and Appendix B of \cite{Quillen:1969}. The description of $L_{\infty}$-algebras 
as coalgebras is also reviewed in Section 2 of \cite{Lada-Markl}.
Unlike the aforementioned references, we use cohomological conventions
i.e.\ our (co)differentials have degree $+1$.

% Throughout this section, graded means $\Z$-graded and all vector
% spaces are over $\R$. If $V$ is a graded vector space, then $\bs V$
% (resp.\ $\ds V$) denotes the \textbf{suspension} (resp.\
% \textbf{desuspension}) of $V$ i.e.
% \[
% (\bs V)_{i} = V_{i-1}, \quad  (\ds V)_{i} = V_{i+1}.
% \]

\subsection{The coalgebra {\boldmath $S(V)$}}\label{cofree_sec}
Given a graded vector space $V$, the graded symmetric algebra
\begin{equation*}
\begin{split}
S(V) &= \R \oplus V \oplus S^{2}(V) \oplus S^{3}(V) \oplus \hdots\\
& =\R \oplus \S(V)
\end{split}
\end{equation*}
is naturally a cocommutative coalgebra. This means that it is equipped with
a linear map $\Delta \maps S(V) \to S(V) \tensor S(V)$
(the comultiplication) such that $\left( \Delta \tensor \id \right) \Delta
= \left(\id \tensor \Delta \right) \Delta$ (coassociativity) and $T
\circ \Delta = \Delta$ (cocommutativity), where $T(x\tensor y) =
(-1)^{\deg{x} \deg{y}} y \tensor x$. In this case, $\Delta$ is the
unique morphism of algebras such that $\Delta(v) = v \tensor 1 + 1
\tensor v$ for all $v\in V$.

From $\Delta$ we also obtain a cocommutative coalgebra structure on
the reduced symmetric algebra:
\[
\rDelta \maps \S(V) \to \S(V) \tensor \S(V)
\]
where $\rDelta c = \Delta c - c \tensor 1 - 1\tensor c$ is the
reduced comultiplication. Explicitly,
\begin{equation*}
 \label{redcomult}
\begin{split}
 \rDelta(v_{1} \odot v_{2} \odot \cdots \odot v_{n})=&
  \sum_{1 \leq p \leq n-1} \sum_{\sigma \in \Sh(p,n-p)} \epsilon(\sigma)
  \left(v_{\sigma(1)} \odot v_{\sigma(2)} \odot \cdots \odot v_{\sigma(p)} \right)\\
  & \tensor \left ( v_{\sigma(p+1)} \odot v_{\sigma(p+2)} \odot \cdots \odot v_{\sigma(n)}
  \right).
\end{split}
\end{equation*}

The reduced diagonal $\rdDelta{n}$ is recursively defined by the formulas:
\begin{equation*} 
%\label{red_diag}
\begin{split}
\rdDelta{0}&=\id\\
\rdDelta{1}&=\rDelta\\
\rdDelta{n} &= \bigl( \rDelta \tensor \id^{\tensor(n-1)} \bigr) \circ
\rdDelta{n-1} \maps \S(V) \to \S(V)^{\tensor(n+1)}.
\end{split}
\end{equation*}
A simple induction argument shows that we can rewrite $\rdDelta{n}$
as
\begin{equation} \label{red_diag2}
\rdDelta{n} = \bigl( \rdDelta{n-1} \tensor \id \bigr) \circ \rDelta.
\end{equation}
Using Eq.\  \eqref{red_diag2}, it is easy to see that
$\S^{\bullet \leq  k}(V) \subseteq \ker \rdDelta{k}$.
The following lemma will be useful in the proceeding sections. It
follows straightforwardly via induction.
\begin{lemma} \label{red_diag_lemma}
 If $v_{1} \odot v_{2} \odot \cdots \odot v_{n} \in \S(V)$, and $ 1 \leq p \leq n-1$ then
\begin{align*}
\rdDelta{p}(v_{1} \odot \cdots \odot v_{n}) &=\sum^{k_{1} + k_{2} +
  \cdots + k_{p+1}=n}_{k_{1},k_{2},\hdots,k_{p+1} \geq 1} \sum_{\sigma
  \in \Sh(k_{1},k_{2},\hdots,k_{p+1})} \epsilon(\sigma) v_{\sigma(1)}
\odot \cdots \odot v_{\sigma(k_{1})} \\
&\quad \tensor v_{\sigma(k_{1} + 1)} \odot \cdots \odot
v_{\sigma(k_{1}+k_{2})} \tensor v_{\sigma(k_{1} +k_{2} + 1)} \odot \cdots \odot
v_{\sigma(k_{1}+k_{2} +k_{3})} \tensor \cdots \\
& \quad \tensor v_{\sigma(n-k_{p+1} + 1)} \odot \cdots \odot v_{\sigma(n)}. 
\end{align*}
In particular, we have
\[
\rdDelta{n-1}(v_{1} \odot \cdots \odot v_{n}) = \sum_{\sigma \in
  \Sn_{n}} \epsilon(\sigma) v_{\sigma(1)} \tensor v_{\sigma(2)}
  \tensor \cdots \tensor v_{\sigma(n)}. 
\]
\end{lemma}

Note that the above lemma implies that $\ker \rdDelta{k} =\S^{\bullet
  \leq  k}(V)$ for  $k \geq 0$ and hence
\begin{equation} \label{prim_cogen}
\S(V)=\bigcup_{n} \ker \rdDelta{n}. 
\end{equation}
% For $n \geq 1$, the canonical filtration $F_{n}S(V)$ corresponds to the filtration on $\S(V)$ given by $\ker \rdDelta{n}$.
% Hence, $S(V)$ is a connected coalgebra.

\subsection{Coalgebra morphisms} \label{coalgebra_morphisms}
A \textbf{morphism} between the reduced coalgebras $(\S(V),\rDelta)$ and $(\S(V'),\rDelta')$
is a degree 0 linear map $F \maps \S(V) \to \S(V')$ such that
\begin{equation*} 
%\label{coalg_morphism_def}
\rDelta' \circ F = (F \tensor F) \circ \rDelta.
\end{equation*}
Given such a morphism, we define the restriction-projections:
\begin{equation}\label{Fproj}
F^p_{n}= \pr_{\S^{p}(V')} \circ F \vert_{\S^{n}(V)} \maps
\S^{n}(V) \to \S^{p}(V').
\end{equation}
The following proposition implies that the linear map
\[
F^{1} \maps \S(V) \to V', \quad F^{1} := F^{1}_{1} + F^{1}_{2} + \cdots
\]
uniquely determines the coalgebra morphism $F$.

\begin{prop}\label{FHT_prop}
If $V$ and $V'$ are graded vector spaces, and $F^{1} \maps \S(V)
\to  V'$  is a degree zero linear map, then there exists a unique
morphism of coalgebras
\[
F \maps \S(V) \to \S(V')
\]
lifting $F^{1}$ such that $\pr_{ V'} \circ F = F^{1}$.
\end{prop}

\begin{proof}
The statement is a special case of Prop.\ 4.1 in Sec.\ B3 of
\cite{Quillen:1969} or Lemma 22.1 in \cite{FHT}.  Here we just recall the construction of the coalgebra morphism $F$. First, define for each $p >0$: 
\[
\psi^{(p)} \maps \underset{p}{\underbrace{\S(V) \tensor \S(V) \tensor \cdots \tensor \S(V)}}
\to \S ^{p} (V')
\]
\begin{equation*}
%\label{morph_eq2}
\psi^{(p)}(c_{1} \tensor c_{2} \cdots \tensor c_{p}) = \frac{1}{p!}
F^{1}_{k_{1}}(c_{1}) \odot \cdots \odot F^{1}_{k_{p}}(c_{p}),
\end{equation*}
where $c_{1},\hdots,c_{p}$ are simple tensors  with $c_{i} \in
\S^{k_{i}}(V)$ and  $F^{1}_{k_{i}} =F^{1}\vert_{ \S^{k_{i}}(V)}$.
Define $F$ to be: 
\begin{equation}\label{morph_eq3}
F(c) =
\sum_{p=0}^{\infty} \psi^{(p+1)} \circ \rdDelta{p}(c), \quad  c \in \S(V).
\end{equation}
Note the infinite sum is well-defined since Eq.\  \eqref{prim_cogen} holds.
\end{proof}

We use Lemma \ref{red_diag_lemma} to write out the formula for $F$
explicitly in terms of the maps $F^{1}_{k}$. Given
$v_{1},\hdots,v_{n} \in V$,  we have 
% Eq.\  \eqref{morph_eq3} implies that
% \[
% F(v_{1} \odot \cdots \odot v_{n}) =
% \sum_{p=0}^{n-1} \psi^{(p+1)} \circ \rdDelta{p}(v_{1} \odot \cdots \odot v_{n})
% \]
% since  $\ker \rdDelta{p} = \S^{\bullet \leq p}(\ds V)$. Therefore
\begin{equation} \label{big_morphism_eq_1}
\begin{split}
F(v_{1} \odot \cdots \odot v_{n}) &=  F^{1}_{n}(v_{1} \odot \cdots
\odot v_{n}) + \sum_{p=1}^{n-1}
\sum^{k_{1} + k_{2} +
  \cdots + k_{p+1}=n}_{k_{1},k_{2},\hdots,k_{p+1} \geq 1} \sum_{\sigma
  \in \Sh(k_{1},k_{2},\hdots,k_{p+1})} \frac{\epsilon(\sigma)}{(p+1)!}\\
&\quad \times F^{1}_{k_{1}}(v_{\sigma(1)}
\odot \cdots \odot v_{\sigma(k_{1})}) 
\odot F^{1}_{k_{2}}(v_{\sigma(k_{1} + 1)} \odot \cdots \odot
v_{\sigma(k_{1}+k_{2})}) \odot \cdots \\
& \quad \odot F^{1}_{k_{p+1}}(v_{\sigma(m-k_{p+1} + 1)} \odot \cdots \odot v_{\sigma(n)}).
\end{split}
\end{equation}
This gives explicit formulas for the projections $F^{p}_{n}$ defined in \eqref{Fproj}:
% We also define projections $F^{p}_{n}$ for $p >1$ by
% \begin{equation}
% F^p_{n}= \pr_{\S^{p}(\ds L')} \circ F \vert_{\S^{n}(\ds L')} \maps
% \S^{n}(\ds L) \to \S^{p}(\ds L'),
% \end{equation}
% and then Eq.\  \eqref{big_morphism_eq_1} implies that
\begin{equation} \label{big_morphism_eq_2}
\begin{split}
F^{p}_{n}(v_{1} \odot \cdots \odot v_{n}) &=
\sum^{k_{1} + k_{2} +
  \cdots + k_{p}=n}_{k_{1},k_{2},\hdots,k_{p} \geq 1} \sum_{\sigma
  \in \Sh(k_{1},k_{2},\hdots,k_{p})} \frac{\epsilon(\sigma)}{p!}
F^{1}_{k_{1}}(v_{\sigma(1)} \odot \cdots \odot v_{\sigma(k_{1})}) \\
& \quad \odot F^{1}_{k_{2}}(v_{\sigma(k_{1} + 1)} \odot \cdots \odot
v_{\sigma(k_{1}+k_{2})}) \odot \cdots 
\odot F^{1}_{k_{p}}(v_{\sigma(m-k_{p} + 1)} \odot \cdots \odot v_{\sigma(n)}).
\end{split}
\end{equation}
In particular, 
\[
F^{n}_{n}(v_{1} \odot \cdots \odot v_{n}) = F^{1}_{1}(v_{1}) \odot F^{1}_{1}(v_{2}) \odot \cdots \odot F^{1}_{1}(v_n),
\]
and
\[
F^{p}_{n}(v_{1} \odot \cdots \odot v_{n})=0 \quad \text{for $p >n$}.
\]

\subsection{{\boldmath $L_{\infty}$}-algebras as dg coalgebras} \label{coalg-Linfty}
A \textbf{codifferential} on the coalgebra $(\S(V),\rDelta)$ is 
a degree $+1$ linear map $Q \maps \S(V) \to \S(V)$ satisfying
\[
\rDelta Q = \left (Q \tensor \id \right) \rDelta +  (\id \tensor Q ) \rDelta.
\]
and
\[
Q \circ Q =0.
\]
% and the coLeibniz identity
% Such a codifferential on $C$ is uniquely determined by its restriction
% to $\bar{C}$ which satisfies the coLeibniz identity with respect to $\rDelta$.
% Recall that in Def.\ \ref{Linfty}, we defined an $L_{\infty}$-algebra
% structure on $L$ to be a collection of skew-symmetric maps $\{l_{k}
% \maps L^{\tensor k} \to L \}_{k=1}^{\infty}$ with $\deg l_{k} =k-2$ which satisfy a
% rather complicated generalization of the Jacobi identity. In contrast,
% the following theorem provides a more elegant description.

If $V$ is a graded vector space, then $\bs V$
 (resp.\ $\ds V$) denotes the suspension (resp.\
 desuspension) of $V$ i.e.
 \[
 (\bs V)_{i} = V_{i-1}, \quad  (\ds V)_{i} = V_{i+1}.
 \]
Recall that in Def.\ \ref{Linfty}, we defined an $L_{\infty}$-algebra
structure on $L$ to be a collection of skew-symmetric maps $\{l_{k}
\maps L^{\tensor k} \to L \}_{k=1}^{\infty}$ with $\deg l_{k} =2-k$ which satisfy a
rather complicated generalization of the Jacobi identity. In contrast,
the following theorem provides a more elegant description.

\begin{thm}[Thm.\ 2.3 \cite{Lada-Markl}] \label{coalg_Linfty_thm}
An $L_{\infty}$-structure $(l_{k})$ on a graded vector space $L$
uniquely determines a degree $1$ codifferential $Q$ on the coalgebra
$$C(L)=\S(\ds L).$$ Conversely, any such codifferential on $C(L)$
uniquely
determines an $L_{\infty}$-structure on $L$.
\end{thm}

% \mccomment{in the following, the use of the indices $m$ and $k$ was slightly confusing, for example in the first equation $k$ is a lower index and in the following equation, its role is taken by $m$ and $k$ is the upper index. Is there a reason for this?}
We will need to briefly describe the correspondence given by the theorem. 
We define the restrictions
\begin{equation*}
 \label{restrict}
Q_{m}=Q \vert_ {\S^{m}(\ds L)} \maps \S^{m}(\ds L) \to \S(\ds L)
\end{equation*}
so that $Q = Q_{1} + Q_{2} +Q_{3} + \hdots$, and the projections
\begin{equation} \label{Qproj}
Q^{k}_{m} = \pr_{\S^{k}(\ds L)} \circ Q_{m} \maps \S^{m}(\ds L) \to
\S^{k}(\ds L).
\end{equation}
It follows from Lemma 2.4 in \cite{Lada-Markl} that $Q$
is uniquely determined by the collection of maps
\begin{equation*} 
\label{struct_maps}
\sQ_{m} = \pr_{\ds L} \circ Q_{m} \maps \S^{m}(\ds L) \to \ds L, \quad
m \geq 1.
\end{equation*}
These are related to the skew-symmetric ``structure maps'' $l_{m}
\maps L^{\tensor m} \to L$ via the formula
\begin{equation}\label{struc_skew}
\sQ_{m} = (-1)^{\frac{m(m-1)}{2}} \ds \circ
l_{m} \circ \bs^{\tensor m},
\end{equation}
while the entire coderivation $Q$ can be expressed as
\begin{multline} \label{coder_eq}
Q_{m}(\ds x_{1} \odot \cdots \odot \ds x_{m}) = 
\sQ_{m}(\ds x_{1} \odot \cdots \odot \ds x_{m}) + \\
\sum^{m-1}_{i=1} \sum_{\sigma \in \Sh(i,m-i)}
\epsilon(\sigma) \sQ_{i}(\ds x_{\sigma(1)} \odot \cdots \odot
\ds x_{\sigma(i)} )\odot \ds x_{\sigma(i+1)} \odot \cdots \odot \ds x_{\sigma(m)},
\end{multline}
for all $x_{i} \in L$.
The condition $Q \circ Q =0$ is equivalent to the generalized
Jacobi identity \eqref{gen_jacobi} for the collection $(l_{k})$. In particular, it implies that $l_{1}$ is degree $+1$ differential on $L$.

\subsection{{\boldmath $L_\infty$}-Morphisms: General case}\label{Linfty_morph_sec}
% A morphism $F \maps (C,Q) \to (C',Q')$ between connected dg coalgebras is a coalgebra morphism such that
% \[
% FQ = Q'F.
% \]
Thanks to Thm\ \ref{coalg_Linfty_thm}, it is now clear what an
$L_{\infty}$-morphism should be.
\begin{defi}\label{morphism_def}
A \textbf{morphism} between $L_{\infty}$-algebras $(L,l_{k})$ and $(L',l'_{k})$ 
is a coalgebra morphism
\[
F \maps C(L) \to C(L')
\]
such that
\[
FQ = Q'F.
\]
\end{defi}
The following proposition says that `strict morphisms' in the sense of
Def.\ \ref{strict_morph_def_1} are precisely those coalgebra morphisms
that satisfy
\[
\forall k \geq 2 \quad F^{1}_{k} =0.
\]
We leave the proof to the reader.
\begin{prop}\label{strict_morphism_prop}
If $(L,l_{k})$ and $(L',l'_{k})$ are $L_{\infty}$-algebras, 
and $f \maps L \to L'$ is a degree zero linear map satisfying
\[
l'_{k} \circ f^{\tensor k} = f \circ l_{k} \quad \forall k \geq 1,
\]
then the linear map $F \maps C(L) \to C(L')$ given by
\[
F( \ds x_{1} \odot \cdots \ds x_{k}) = \ds f(x_{1}) \odot \cdots \odot
\ds f(x_{k})
\]
is a strict $L_{\infty}$-morphism.
\end{prop}

More generally, if $F \maps (C(L),Q) \to (C(L'), Q')$ is any $L_{\infty}$-morphism,
then the projections defined in Eq.\ \eqref{Fproj} and Eq.\
\eqref{Qproj} allow us to write the equality $F Q = Q' F$ as
\begin{equation} \label{codiff_compat_eq}
\sum^{m}_{k=1}\sF_{k} Q^{k}_{m}=\sum^{m}_{k=1} Q'^{1} _{k} F^{k}_{m} 
\quad \forall m  \geq 1.
\end{equation}
Every such $F$ is of the form \eqref{morph_eq3}, since by Prop.\
\ref{FHT_prop} it is the unique lift of its projection
$F^{1}=F^{1}_{1} + F^{1}_{2} + F^{1}_{3} + \cdots$. Hence, $F$ is
uniquely determined by its corresponding collection of
``structure maps'' $\{ f_{k} \maps L^{\tensor k} \to L' \}^{\infty}_{k \geq 1}$ 
which satisfy
\begin{equation} \label{morph_eq1}
 F^{1}_{k} = (-1)^{\frac{k(k-1)}{2}} \ds \circ
 f_{k} \circ \bs^{\tensor k}.
\end{equation}
Hence each $f_{k}$ is graded
skew-symmetric with $\deg{f_{k}} =1-k$.
Note that the equality $FQ=Q'F$ implies that the degree zero map
\[
f_{1} \maps (L,l_1) \to (L',l'_{1})
\]
is a morphism of cochain complexes. This leads us to the notion of
$L_{\infty}$-quasi-isomorphism given in Def.\ \ref{Linfty_qiso_def}.

\begin{remark} \label{rmk:general_morphism}
The defining equations \eqref{codiff_compat_eq} for a morphism 
$(L,l_{k}) \to (L',l'_{k})$ between $L_\infty$-algebras
can be given explicitly in terms of the structure maps
$f_k$, $l_k$, and $l'_k$. First, one uses  Eqs.\ \eqref{big_morphism_eq_2} and
\eqref{coder_eq} to rewrite  \eqref{codiff_compat_eq} in terms of the
multi-linear maps $F^{1}_{k}$, $Q^{1}_{k}$, and $Q'^{1}_{k}$. 
As Eqs.\ \eqref{struc_skew} and \eqref{morph_eq1} show, these
maps are nothing but the structure maps composed with the appropriate
suspensions and desuspensions. (See Cor.\ \ref{Lie_alg_P_cor} for a special case.)
\end{remark}

\subsection{Lie algebras} \label{lie_alg_sec}
Any  differential graded Lie algebra (DGLA) can be thought of as a $L_{\infty}$-algebra 
by associating to $(\g,d,[\cdot,\cdot])$ the coalgebra $\S(\ds \g)$ with codifferential $D$ defined by the equations
\begin{equation} \label{coder_liealg}
\begin{split}
D_{1}(\ds x) &= \ds d x \\
D^{1}_{2}(\ds x, \ds y) &= (-1)^{\deg{x}}\ds [x,y]\\
\quad D^{1}_{k}&= 0, ~ k \geq 3.
\end{split}
\end{equation}
A DGLA morphism $f \maps \g \to \g'$ induces a unique strict $L_{\infty}$-morphism between 
$(\S(\ds \g),D)$ and $(\S(\ds \g'),D')$. 
We treat ordinary Lie algebras as DGLAs concentrated in degree zero with differential $d=0$.

Now we consider $L_{\infty}$-algebra morphisms whose sources
are just Lie algebras $(\g,[\cdot,\cdot])$.
Since the projections $D^{k}_{m}$ are built from the structure maps
$D^{1}_{m}$ via Eq.\  \eqref{coder_eq},  we have
\begin{equation*}
D^{k}_{m} =0 \quad \text{whenever $k \neq m-1$.}
\end{equation*} 
Therefore, Eq.\  \eqref{codiff_compat_eq}, which a coalgebra morphism $F
\maps \S(\ds \g)
\to \S(\ds L)$ must satisfy to be an $L_{\infty}$-morphism, simplifies to 
\begin{equation} \label{Lie_alg_morph}
\begin{split}
Q^{1}_{1}F^{1}_{1} &= 0,\\
\sF_{m-1} D^{m-1}_{m} &=\sum^{m}_{k=1} Q^{1} _{k} F^{k}_{m}  \quad \forall m  \geq 2.
\end{split}
\end{equation}

In particular, homotopy moment maps (Def.\ \ref{main_def}) are 
 $L_{\infty}$-morphisms from a Lie algebra to a
Lie $n$-algebra $(L,l_{k})$ satisfying Property \eqref{property}, which we defined in Sec.\ \ref{secP} as being:
\[
\forall k \geq 2 \quad  l_{k}(x_{1},\ldots,x_{k}) = 0 \quad
\text{whenever $\sum_{i=1}^{k} \deg{x_{i}} < 0.$}
\]
Equation \eqref{struc_skew} implies that this is equivalent to the
corresponding codifferential $Q$ on $\S(\ds L)$ satisfying
\[
\forall k \geq 2 \quad  Q^{1}_{k}(\ds x_{1} \odot \cdots \odot \ds x_{k}) = 0 \quad
\text{whenever $\deg{\ds x_{1} \odot \cdots \odot \ds x_{k}} < k.$}
\]

\begin{prop}\label{Lie_alg_prop_P2}
If $(\g,[\cdot,\cdot])$ is a Lie algebra and $(L,l_{k})$ is a Lie $n$-algebra
satisfying Property \eqref{property}, then a coalgebra morphism 
$F \maps \S(\ds \g) \to \S(\ds L)$ is an $L_{\infty}$-algebra morphism if and
only if
\begin{equation}\label{Lie_alg_prop_P_eq1}
\sF_{m-1} D^{m-1}_{m} = Q^{1}_{1} \sF_{m} + Q^{1}_{m}F^{m}_{m}
\end{equation}
for $2 \leq m \leq n$, and
\begin{equation} \label{Lie_alg_prop_P_eq2}
F^{1}_{n}D^{n}_{n+1} = Q^{1}_{n+1}F^{n+1}_{n+1},
\end{equation}
where $D$ and $Q$ are the codifferentials determined by
$[\cdot,\cdot]$, and $(l_{k})$, respectively.
\end{prop}
\begin{proof}
We will show the conditions given in Eqs.\ \eqref{Lie_alg_prop_P_eq1}
and \eqref{Lie_alg_prop_P_eq2} are equivalent to
those in \eqref{Lie_alg_morph}. First, note that for any coalgebra
morphism $F \maps \S(\ds \g) \to \S(\ds L)$
\[
Q^{1}_{1}F^{1}_{1} = 0 
\]
holds trivially, since $F^{1}_{1}$ is a degree 0 map and $\ds \g$ is in degree -1, while $Q^{1}_{1}$ has
degree $+1$ and $\ds L$ is concentrated in degrees $-n,\ldots,-1$.
Next, we observe that Property \eqref{property} and Eq.\  \eqref{big_morphism_eq_2} imply that
\[
\sum^{m}_{k=1} Q^{1} _{k} F^{k}_{m}=Q^{1}_{1} \sF_{m} + Q^{1}_{m}F^{m}_{m}  \quad \forall m  \geq 2.
\]

When $m \geq n+1$, the degree condition on $\ds L$ implies that
 $F^{1}_{m} =0$ and hence
\begin{align*}
 Q^{1}_{1} \sF_{m}&=0 \quad \forall m \geq n+1, \\
F^{1}_{m-1} D^{m-1}_{m} &=0 \quad \forall m \geq n+2.
\end{align*}
For the same reason, $Q^{1}_{m}=0$ whenever $m \geq n+2$. Therefore
\[
Q^{1}_{m}F^{m}_{m} =0  \quad \forall m  \geq n+2.
\]
Hence, satisfying Eqs.\ \eqref{Lie_alg_prop_P_eq1} and \eqref{Lie_alg_prop_P_eq2} is both necessary
and sufficient for $F$ to be an $L_{\infty}$-morphism.
\end{proof}

We now prove Prop.\ \ref{Lie_alg_P_prop1} as a corollary of the above.
\begin{cor}[Prop.\ \ref{Lie_alg_P_prop1}] \label{Lie_alg_P_cor}
If $(\g,[\cdot,\cdot])$ is a Lie algebra and $(L,l_{k})$ is a Lie $n$-algebra
satisfying property \eqref{property}, then a collection of $n$
skew-symmetric maps
\[
f_{m} \maps \g^{\tensor m} \to L, \quad \deg{f_{m}} = 1-m, \quad 1
\leq m \leq n
\]
determine an $L_{\infty}$-morphism $\S(\ds \g) \to \S(\ds L)$ via Eq.\  \eqref{morph_eq1} if and
only if $\forall x_{i} \in \g$
\begin{multline} \label{acor_eq1}
\sum_{1 \leq i < j \leq m}
(-1)^{i+j+1}f_{m-1}([x_{i},x_{j}],x_{1},\ldots,\widehat{x_{i}},\ldots,\widehat{x_{j}},\ldots,x_{m})\\
=l_{1} f_{m}(x_{1},\ldots,x_{m}) + l_{m}(f_{1}(x_{1}),\ldots,f_{1}(x_{m})).
\end{multline}
for $2 \leq m \leq n$ and
\begin{multline} \label {acor_eq2}
\sum_{1 \leq i < j \leq n+1}
(-1)^{i+j+1}f_{n}([x_{i},x_{j}],x_{1},\ldots,\widehat{x_{i}},\ldots,\widehat{x_{j}},\ldots,x_{n+1})
=l_{n+1}(f_{1}(x_{1}),\ldots,f_{1}(x_{n+1})).
\end{multline}
\end{cor}
\begin{proof}
Assume we are given such maps $f_{1},\hdots,f_{n}$ satisfying the above
equalities. Using Eq.\  \eqref{morph_eq1}, we construct the corresponding
degree 0 maps $F^{1}_{1},\hdots,F^{1}_{n}$, and set $F^{1}_{k}=0$ for
$k \geq n +1$. By Prop.\ \ref{FHT_prop}, these give a unique
coalgebra morphism $F \maps \S(\ds \g) \to \S(\ds L)$. To show $F$ is
an $L_{\infty}$-morphism, Prop.\ \ref{Lie_alg_prop_P2} implies it is
sufficient to show Eqs.\
\eqref{Lie_alg_prop_P_eq1} and \eqref{Lie_alg_prop_P_eq2} hold. 
From Eq.\  \eqref{coder_liealg}, we have the equality
$D^{1}_{2}(\ds x \odot \ds y) = \ds [x,y]$, while Eq.\  \eqref{coder_eq} implies that
\begin{align*}
D^{m-1}_{m}(\ds x_1 \odot \cdots \odot \ds x_m) &= \sum_{\sigma \in \Sh(2,m-2)}
\epsilon(\sigma) D^{1}_{2}(\ds x_{\sigma(1)} \odot \ds x_{\sigma(2)})
\odot \cdots \odot \ds x_{\sigma(m)}\\
&=\sum_{1 \leq i < j \leq m}
(-1)^{i+j+1}\ds[x_{i}, x_{j}] \odot \ds x_{1} \odot \cdots
\widehat{\ds x_{i}} \cdots \widehat{\ds x_{j}} \cdots \odot \ds x_{m}).
\end{align*}
The signs in the last equality above are due to the fact that $\g$ is in degree 0.
It follows from Eq.\  \eqref{morph_eq1} that 
\begin{equation} \label{foo_eq1}
F^{1}_{m}(\ds x_1 \odot \cdots \odot \ds x_m) = \ds f_{m} (x_1, \hdots, x_m).
\end{equation}
Therefore, the left-hand sides of Eqs.\ \eqref{Lie_alg_prop_P_eq1}
and \eqref{Lie_alg_prop_P_eq2} are the desuspension of the left-hand sides
of Eqs.\ \eqref{acor_eq1} and \eqref{acor_eq2}, respectively.

Now we consider the right-hand sides. First, note that Eq.\ 
\eqref{foo_eq1} also implies that
\begin{equation} \label{foo_eq2}
Q^{1}_{1} F^{1}_{m} = \ds l_{1} \circ f_{m}.
\end{equation}
Recall Eq.\  \eqref{big_morphism_eq_2} gives
\[
F^{m}_{m}(\ds x_{1} \odot \cdots \odot \ds x_{m})  =F^{1}_{1}(\ds x_{1}) \odot F^{1}_{1}(\ds x_{2}) \odot \cdots \odot F^{1}_{1}(\ds x_{m}).
\]  
For each $x_{i}$, we have $\deg{F^{1}_{1}(\ds x_{i})}=-1$ and
$F^{1}_{1}(\ds x_{i}) = \ds f_{1}(x_{i})$. Therefore,
\[
Q^{1}_{m}F^{m}_{m}(\ds x_{1} \odot \cdots \odot  \ds x_{m}) = \ds  l_{m}(f_{1}(x_{1}),\ldots,f_{1}(x_{m})).
\]
Combining the above equality with Eq.\  \eqref{foo_eq2}, we see that the
right-hand sides of Eqs.\ \eqref{Lie_alg_prop_P_eq1} and \eqref{Lie_alg_prop_P_eq2} are the desuspension of
the right-hand sides of Eqs.\ \eqref{acor_eq1} and \eqref{acor_eq2}, respectively.
Hence, $F$ is a $L_{\infty}$-morphism. 

It is easy to see to see that the converse follows by reversing the above arguments.
\end{proof}

% \begin{remark} \label{rem:infty-morph_general}

% \end{remark}

\subsection{{\boldmath $L_{\infty}$}-morphisms and central {\boldmath $n$}-extensions}
Let $(\g,[\cdot,\cdot])$ be a Lie algebra and $c \maps \Lambda^{n+1}
\g \to \R$ a degree $n+1$ cocycle in the Chevalley-Eilenberg complex
associated to $\g$. A theorem of Baez and Crans \cite[Thm.\ 55]{hd6} implies
that this data gives a Lie $n$-algebra $\widehat{\g}_{c}$
whose underlying vector space is
\begin{align*}
L_{0}&=\g,\\
L_{i}&=0 \quad 2-n \leq i \leq  -1,\\
L_{1-n}&=\R,
\end{align*}
and whose only non-trivial multibrackets are
\begin{align*}
l_{2}(x_{1},x_{2}) &= 
\begin{cases}
[x_{1},x_{2}] & \text{if $x_{1},x_{2} \in \g$}\\
0 & \text{otherwise}
\end{cases}\\
l_{n+1}(x_{1},\hdots,x_{n+1}) &=
\begin{cases}
c(x_{1},\hdots,x_{n+1}) & \text{if $x_{1},\hdots,x_{n+1} \in \g$}\\
0 & \text{otherwise}
\end{cases} \\
\end{align*}
We call such Lie $n$-algebras {\bf central \boldmath ${n}$-extensions} of $\g$.
\begin{prop} \label{ext_morph_prop}
Let $\g$ be a Lie algebra, $c \in \Hom(\Lambda^{n+1}\g,\R)$ a
$(n+1)$-cocycle, and $\widehat{\g}_{c}$ the corresponding central $n$-extension.
If $(L,l_{k})$ is a Lie $n$-algebra
satisfying property \eqref{property}, then a collection of $n$
skew-symmetric maps
\begin{equation*}
\begin{split}
f_{1} &\maps \g \oplus \R[n-1] \to L\\
f_{m} &\maps \g^{\tensor m} \to L, \quad \deg{f_{m}} = 1-m, \quad 2 \leq m \leq n
\end{split}
\end{equation*}
determine an $L_{\infty}$-morphism $\S(\ds \widehat{\g}_{c}) \to \S(\ds L)$ if and
only if 
\begin{equation} \label{ext_eq0}
l_{1}f_{1}(r)=0 \quad \forall r \in \R,
\end{equation}
and $\forall x_{i} \in \g$
\begin{multline} \label{ext_eq1}
\sum_{1 \leq i < j \leq m}
(-1)^{i+j+1}f_{m-1}([x_{i},x_{j}],x_{1},\ldots,\widehat{x_{i}},\ldots,\widehat{x_{j}},\ldots,x_{m})\\
=l_{1} f_{m}(x_{1},\ldots,x_{m}) + l_{m}(f_{1}(x_{1}),\ldots,f_{1}(x_{m})).
\end{multline}
for $2 \leq m \leq n$ and
\begin{multline} \label{ext_eq2}
\sum_{1 \leq i < j \leq n+1}
(-1)^{i+j+1}f_{n}([x_{i},x_{j}],x_{1},\ldots,\widehat{x_{i}},\ldots,\widehat{x_{j}},\ldots,x_{n+1})
+ f_{1}c(x_{1},\ldots,{x}_{n+1}) \\
=l_{n+1}(f_{1}(x_{1}),\ldots,f_{1}(x_{n+1})).
\end{multline}
\end{prop}
\begin{proof}
  Observe the similarity between the above formulas
  and those given in Cor.\ \ref{Lie_alg_P_cor} for a
  $L_{\infty}$-morphism from $\g$ to $(L,l_{k})$.  
%The only differences  involve the map $f_{1}$ and $f_{n}$.
  Let $D$ denote the codifferential on $\S(\ds \widehat{\g}_{c})$.
  We proceed as we did in the proof of Prop.\ \ref{Lie_alg_prop_P2} and conclude
  that a coalgebra morphism $F \maps \S(\ds \widehat{\g}_{c}) \to \S(\ds
  L)$ is an $L_{\infty}$-morphism iff
\begin{equation*}
\begin{split}
Q_{1}F_{1}(\ds r) &=0 \quad \forall r \in \R[n-1], \\
\sF_{m-1} D^{m-1}_{m} &= Q^{1}_{1} \sF_{m} + Q^{1}_{m}F^{m}_{m} \quad 2 \leq m \leq n,\\
\end{split}
\end{equation*}
and
\[
F^{1}_{n}D^{n}_{n+1} + F^{1}_{1}D^{1}_{n+1} = Q^{1}_{n+1}F^{n+1}_{n+1}.
\]
Rewriting these in terms of structure maps $(f_{k})$ (cf.\ the proof
of Cor.\ \ref{Lie_alg_P_cor}), we obtain Eqs.\ \eqref{ext_eq0},
\eqref{ext_eq1}, and \eqref{ext_eq2}.
\end{proof}

Note that a central $n$-extension itself satisfies Property
\eqref{property}, so we have the following corollary:
\begin{cor} \label{ext_qiso}
If $[c]=[c'] \in H^{n+1}_{\CE}(\g,\R)$, then the central
$n$-extensions $\widehat{\g}_{c}$ and $\widehat{\g}_{c'}$ are quasi-isomorphic.
\end{cor}
\begin{proof}
Let $b \maps \Lambda^{n} \g \to \R$ such that $c' = c + \delta_{\CE}b$.
Consider the collection of skew-symmetric maps:
$f_{1}=\id_{\g \oplus \R[n-1]}$, $f_{k}=0$ for $2 \leq k \leq n-1$
and $f_{n}=b$. Using Prop.\ \ref{ext_morph_prop}, it's easy to see
these give an $L_{\infty}$-morphism $\widehat{\g}_{c} \to
\widehat{\g}_{c'}$. Since $f_{1}$ is the identity, it is clearly a quasi-isomorphism.
\end{proof}

\section{Proof of Theorem \ref{thm:CartanMM}} \label{sec:proof_CartanMM}

In this appendix, we prove Thm.\ \ref{thm:CartanMM}, which provides an
explicit formula for contructing a moment map from a cocycle in the
Cartan complex.  For this, we need to construct a natural chain map
$\Phi \maps C_G^*(M)\to \bOmega(G \ltimes M_{\bu})$, which is a
quasi-isomorphism when $G$ is compact. This chain map seems to be well
known; see for example \cite[Appendix C]{Meinrenken}. 
However, we need an explicit formula.
%  we recall the details of its
% construction.

The rationale for this construction of $\Phi$ is the following. 
We wish to identify $C^{\ast}_{G}(M)$ with differential forms on the
base space, i.e.\ the (homotopy) quotient $G \ltimes M_{\bu}$. So, 
in essence, we require a connection. Hence, we replace $M$ with a
equivalent space  $E_\bu G \times M_\bu$, 
which is the total space of a principal $G$ bundle over $G \ltimes
M_{\bu}$, and furthermore equipped with a canonical connection. 
The connection induces a well-known chain map (the ``Cartan
map'') between the Cartan model of the total space, and forms on the
base. This map is an extension of the usual Chern-Weil homomorphism,
and therefore we require  a workable theory of 
Lie algebra-valued forms in the simplicial setting. 
In particular, we would like a commutative product on differential
forms. Unfortunately, the product on the $\bOmega(G \ltimes M_{\bu})$
is only homotopy commutative. So we temporarily replace this complex by
an equivalent one which is strictly commutative, namely
Dupont's model for the de Rham complex of $G \ltimes M_{\bu}$.

\subsection{Differential forms on simplicial manifolds}

If $\bX$ is a simplicial manifold with face maps $d_i: X_n \to
X_{n-1}$, $i=0,\ldots,n$, then the simplicial differential $\del_n \maps
\Omega^{*}(X_{n}) \to \Omega^{*}(X_{n +1})$ is
\begin{equation} \label{eq:simp_dif}
\del_n = \sum_{i=0}^{n+1} (-1)^{i} d^{\ast}_i.
\end{equation}
The Bott-Shulman-Stasheff complex is the total complex of the
double complex of differential forms on $\bX$:
% with differentials $\del$ and $d$, where $\del = \sum_i \pm d^{\ast}_i$ is the usual simplicial differential, and 
\begin{equation}\begin{split}\label{eq:diffFormsSimplMfd}
 \Omega^{j,k}(\bX)&:=\Omega^k(X_j), \\
\Omega^*(\bX)&:=\bigl(\Tot(\Omega^{*,*}(\bX)), D\bigr), \\
D&:= \del + (-1)^j d, 
\end{split}\end{equation}
where $d$ is the de Rham differential.

\begin{ep}
Let $M$ be a manifold and $M_\bu$ the simplicial manifold $M_n = M$,
whose face and degeneracy maps are $\id_M$. Since all $\partial_n$ are either zero or isomorphisms, the inclusion 
\begin{equation}\label{eq:mfdSimpMfd}
(\Omega^{*}(M),d) =(\Omega^{*}(M_0),d) \stackrel{\iota}{\hookrightarrow} (\Omega^{*}(M_\bu),D)
\end{equation}
is an quasi-isomorphism. %TODO: mention spectral sequence?
\end{ep}

\begin{ep}
 If $M$ is a $G$-manifold, let $\EGM$ denote the product $E_{\bu}G
\times M_{\bu}$, i.e.,  the simplicial manifold 
\[
[n] \mapsto E_{n}G \times M =G^{n+1} \times M
\]
with the ``usual'' face and degeneracy maps.
\end{ep}

\begin{prop}
The projection $\pi_M\colon \EGM \to M_{\bu}$ induces a quasi-isomorphism
\begin{equation}\label{eq:piStar}
 \pi_M^{\ast} \maps \Omega^*(M_{\bu}) \to \Omega^* \bigl(\EGM \bigr).
\end{equation}
\end{prop}

\begin{proof}
% % TODO: say something about $G$-actions?
% Let us first show that $\pi^{\ast}\colon \bOmega(M_{\bu}) \to \bOmega(\EGM)$ is indeed a quasi-isomorphism.
Recall, that if $\bX$ is a simplicial manifold which is paracompact in
each dimension, then
the de Rham theorem of Bott-Shulman-Stasheff \cite{BSS:1976} implies
that there exists a natural isomorphism
\[
H \bigl( \bOmega(\bX) \bigr) \xto{\cong}
H \bigl(\fat{\bX} \bigr),
\]
where $H \bigl(\fat{\bX} \bigr)$ is the singular cohomology with $\R$ coefficients of the fat
realization of $\bX$. We denote by $|\bX|$  the thin geometric realization of $\bX$. Since $|\cdot|$ preserves products, and since both $G$ and $M$ are manifolds, it follows from \cite[Prop.\ A1]{Segal:1974} that we have a commuting diagram
\begin{equation} \label{diag1}
\xymatrix{
 H \bigl( \bOmega(M_{\bu}) \bigr) \ar[r]^{\cong}
 \ar[d]_{\pi_M^{\ast}} & H \bigl(\fat{M_{\bu}} \bigr) \ar[r]^{\cong}
 \ar[d]^{\fat{\pi_M}^{\ast}} & H\bigl(|M_{\bu}| \bigr)\ar[r]^{=} \ar[d]^{|\pi_M|^{\ast} } &  H\bigl(|M_{\bu}| \bigr)\ar[d]^{\left(\pi_{|M_\bu|}\right)^{\ast} }\\
 H \bigl( \bOmega(\EGM) \bigr) \ar[r]^-{\cong} & 
H \bigl(\fat{\EGM} \bigr)\ar[r]^-{\cong} & H \bigl(|\EGM| \bigr) \ar[r]^-{\cong} & H \bigl(|E_\bu G|\times|M_\bu| \bigr).
}
\end{equation}

Since $|E_\bu G|$ is contractible, the K\"{u}nneth formula implies that the right vertical arrow in the diagram
\eqref{diag1} is an isomorphism. Hence $\pi_M^{\ast}$ is also an isomorphism.
\end{proof}

\begin{remark}(cf. \cite[Appendix C.2]{Meinrenken})
Equip $\EGM$ with the diagonal $G$-action
\begin{equation}\begin{split}\label{eq:actionEGM}
G\times E_nG\times M &\to E_nG\times M, \\
(h,g_0,\ldots, g_n,p)&\mapsto (g_0h^{-1},\ldots, g_nh^{-1}, hp).
\end{split}\end{equation}
The projection
\[
\pi_M \maps  \EGM \to M_{\bu}
\]
is a morphism of simplicial $G$-manifolds.

Note that the map 
\begin{equation*}\begin{split}\label{eq:isoBG}
G^{n+1}\times M&\to G^n\times M, \\ 
(g_0,\ldots,g_n,p)&\mapsto (g_0 g_1^{-1},\ldots,g_{n-1}g_n^{-1},g_n p)
\end{split}\end{equation*}
induces an isomorphism of simplicial manifolds \[E_\bu G\times_G M\cong G\ltimes M_\bu,\] where $E_\bu G\times_G M$ is the quotient of $E_{\bu}G\times M$ by the diagonal $G$-action.
% The de Rham theorem of Bott-Shulman-Stasheff \cite{BSS:1976} implies that the cohomology of $\bigl(C^*(G \ltimes M_{\bu}),\bd_{BS}\bigr)$ is the equivariant cohomology of $M$. % since M paracompact
Futhermore, the bundle $E_n G\times M \to E_n G\times_G M\cong G^n\times M$ is trivial with section $s$
\begin{equation}\begin{split}\label{eq:trivSection}
 s\colon G^n\times M&\to G^{n+1}\times M = E_n G\times M,\\
 (g_1,\ldots,g_n,p) &\mapsto (e,g_1^{-1},\ldots, (g_1\cdots g_n)^{-1},g_1\cdots g_n p).
\end{split}\end{equation} 
\end{remark}

\subsection{Simplicial differential forms}
We recall  the notion of simplicial differential forms introduced by Dupont \cite[Def 2.1]{Dupont}:

Let $\bX$ be a simplicial manifold with face maps $d_i\colon X_q\to X_{q-1}$ for $i=0,\ldots,q$. Let $\Delta^q\subset\R^{q+1}$ be the standard $q$-simplex and $\varepsilon_i\colon \Delta^{q-1}\to\Delta^q$ the inclusion of the $i$-th face.

A \textbf{simplicial differential {\boldmath $n$}-form} $\varphi$ on $\bX$ consists of a sequence of forms \[\varphi^{(q)}\in\Omega^n(\Delta^q\times X_q),~q=0,1,\ldots \] satisfying 
\begin{equation*}\label{eq:simplForms}
 (\varepsilon_i\times \id)^*\varphi^{(q)} = (\id\times d_i)^*\varphi^{(q-1)}
\end{equation*}
for all $q$ and all $i=1,\ldots,q$.

The set of all simplicial $n$-forms on $\bX$ is denoted
$\Omega^n_{\spl}(\bX)$.  Equipped with the usual de Rham differential
$d$, $\Omega^*_{\spl}(\bX)$ is a differential graded commutative
algebra, which is also the total complex of the following double
complex:
\begin{equation}\label{eq:simplFormsDoubleComplex}
 \Omega^n_{\spl}(\bX) = \bigoplus_{j+k = n}\Omega_{\spl}^{j,k}(\bX).
\end{equation}
Here, similar to \eqref{eq:prodDoubleCx}, $\Omega_{\spl}^{j,k}(\bX)$
consists of those simplicial differential $n$-forms $\varphi =
(\varphi^{(q)})$ with the property
\[\varphi^{(q)}\in\Gamma(\Delta^q\times X_q,\Lambda^j T^*\Delta^q\otimes\Lambda^kT^*X_q)\subset\Omega^{j+k}(\Delta^q\times X_q).\]
The de Rham differential $d$ on $\Omega_{\spl}^*(\bX)$ is 
\[d = d^\Delta + (-1)^j d^X, \]
 where $d^\Delta$ and $d^X$ denote the de Rham differentials in the $\Delta^q$ and $X_q$-directions, respectively. 

Dupont proved that $\Omega^*(\bX)$ and $\Omega_{\spl}^*(\bX)$ are quasi-isomorphic.
\begin{thm}[{\cite[Thm.\ 2.3]{Dupont}}]\label{thm:Dupont}
 There are natural maps of doubles complexes 
\[\xymatrix{\bigl(\Omega_{\spl}^{*,*}(\bX), d^\Delta, d^X\bigr)\ar@<.5ex>[r]^(.55){\mathscr{I}}&\ar^(.45){\mathscr{C}}@<.5ex>[l]\bigl(\Omega^{*,*}(\bX), \del, d\bigr),}\]
 which give natural chain homotopy equivalences between $\bigl(\Omega_{\spl}^{*,k}(\bX),d^\Delta\bigr)$ and $\bigl(\Omega^{*,k}(\bX),\del\bigr)$. 

In particular, $\mathscr{C}$ and $\mathscr{I}$ induce quasi-isomorphisms between the total complexes $\bigl(\Omega^*_{\spl}(\bX),d\bigr)$ and $\bigl(\Omega^*(\bX), D\bigr)$.
\end{thm}

The map $\mathscr{I}$ in Dupont's theorem is integration over the
fiber: if $\varphi \in \Omega_{\spl}^{j,k}(\bX)$, then
\begin{equation}\label{eq:dupontI}
\mathscr{I}(\varphi) := \int_{\Delta^j}\varphi^{(j)}\in\Omega^k(X_j)
\end{equation}
Going the other direction, if $ \beta \in  \Omega^k(X_j)$, 
then the simplicial form $\mathscr{C}(\beta)\in\Omega^{j,k}_{\spl}(\bX)$ is:
\begin{equation}\label{eq:dupontC}\begin{split}
\mathscr{C}(\beta)^{(q)} :=&
\begin{cases}
 j!\sum_{|I|=j}\sum_{\ell=0}^j(-1)^\ell t_{i_\ell} dt_{i_0}\wedge\ldots\wedge \widehat{dt_{i_\ell}}\wedge\ldots\wedge dt_{i_j}\wedge \mu_I^*\beta & \textrm{if }q\geq j,\\
0 & \textrm{if }q<j.
\end{cases}
\end{split}\end{equation}
Here $I=(i_0,\ldots,i_j)$ is a multi-index with $0\leq i_0<\cdots<
i_j\leq q$, $|I|:=j$, and $\mu_I = d_{\tilde{\iota}_{q-j}}\circ\ldots\circ d_{\tilde{\iota}_1}\colon X_q\to X_j$ is the face map corresponding to the complementary sequence $0\leq \tilde{\iota}_1<\cdots< \tilde{\iota}_{q-j}\leq q$ of $I$. 

Composing the inclusion $\iota$ \eqref{eq:mfdSimpMfd}, $\pi_M^*$ \eqref{eq:piStar} and Dupont's map $\mathscr{C}$ \eqref{eq:dupontC} for $\bX = \EGM$, we obtain a quasi-isomorphism

\begin{equation}
\Omega^*(M)\xto{\iota} \Omega^*(M_\bu)\xto{\pi_M^*} \Omega^*(\EGM)\xto{\mathscr{C}} \Omega_{\spl}^*(\EGM).
\end{equation}

Note that a $k$-form $\alpha\in\Omega^{k}(M)$ is mapped to the simplicial $k$-form $\mathscr{C}(\pi_M^*\alpha)\in\Omega_{\spl}^{0,k}(\EGM)$ which is given by the sequence
\[ \mathscr{C}(\pi_M^*\alpha)^{(q)} = \pi_M^*\alpha\in\Omega^{0,k}(\Delta^q\times E_q G\times M).\]

The map induced by $\mathscr{C}\circ\pi_M^*\circ \iota$ on the corresponding Cartan complexes, which we consider next, provides the first step in the construction of the chain map $C_G^*(M)\to \bOmega(G\ltimes M)$.

\subsection{Cartan complexes}
If $A$ is a $G^{\star}$-module in the sense of \cite[Def.\ 2.3.1]{G-S}, with differential $d^A$ and insertion operation $\iota_\g^A$, let
\begin{equation}\begin{split}\label{eq:CartanComplexGstarModule}
C_{G}(A):=& \bigl(S(\g^{\vee}) \tensor A \bigr)^{G}, \\
 d_G =& \delta + d^A
\end{split}\end{equation}
denote the Cartan complex (\cite[section 6.5]{G-S}),  where $\delta = - \Sym\circ\iota_\g^A$ the composition of $-\iota_\g^A$ and the symmetrization $\Sym\colon \g^\vee\otimes S^*(\g^\vee)\to S^{*+1}(\g^\vee)$.
This is also the total complex of the double complex 
\begin{equation*}\begin{split}
 C_G^{i,j}(A) :=& \bigl(S^i(\g^\vee)\otimes A^{j-i}\bigr)^G,\\
 C_G^*(A) =& \Tot( C_G^{*,*}(A)).
\end{split}\end{equation*}
Define the decreasing filtration on $C_G(A)$:
\begin{equation} \label{eq:filt}
 F_p C_G(A):=\bigoplus_{i\geq p}\bigoplus_j C_G^{i,j}(A)
\end{equation}
If $A$ is bounded below, then the associated spectral sequence clearly converges.

\begin{lemma}\label{lem:inducedQuiso}
Let $G$ be a compact Lie group, $A$ and $B$ two $G^\star$-modules
which are bounded below as complexes and let $\phi\colon A\to B$ be a quasi-isomorphism of $G^\star$-modules, i.e., a morphism of $G^\star$-modules, which induces an isomorphism of $G$-modules on total cohomology. Then the induced map of Cartan complexes
\[\id_{S(\g^\vee)}\otimes\phi\colon C_G(A)\to C_G(B) \]
is a quasi-isomorphism.
\end{lemma}
\begin{proof}
 The induced map $\id_{S(\g^\vee)}\otimes\phi$ respects the filtrations defined in \eqref{eq:filt}. Since
$\phi$ is a quasi-isomorphism and $G$ is compact, $\id_{S(\g^\vee)} \tensor \phi$ induces an isomorphism
between the $E_1$ pages
\[ E_1^{p,q}(A) = \bigl( S^p(\g^\vee)\otimes H^{q-p}(A)\bigr)^G \to \bigl( S^p(\g^\vee)\otimes H^{q-p}(B)\bigr)^G = E_1^{p,q}(B)\]
 of the associated spectral
sequences (e.g. \cite[Thm.\ 6.5.1]{G-S}). Since $A$ and $B$ are bounded below, the filtrations are bounded in each degree. Therefore, $\id_{S(\g^\vee)} \tensor \pi_M^{\ast}$ is a quasi-isomorphism (e.g. \cite[Thm.\ 3.5]{McCleary}).
\end{proof}

\begin{ep}
For $M$ a $G$-manifold the Cartan complex of $\Omega^*(M)$ with the usual $G^\star$-module structure is the usual Cartan complex for $M$:
\begin{equation}\label{eq:cartanM}
 C_G(M)= C_G(\Omega^*(M)).
\end{equation}
\end{ep}

\begin{ep} For a simplicial $G$-manifold $\bX$, the complex $\Omega^*(\bX) = \Tot(\Omega^{*,*}(\bX))$ \eqref{eq:diffFormsSimplMfd} with differential $D=\del + (-1)^j d$ and the insertion operation $\iota_\g^{\Omega^*(\bX)}:=(-1)^j\iota_\g$ is a $G^\star$-module. Note that
\[ D \iota_\g^{\Omega^*(\bX)} + \iota_\g^{\Omega^*(\bX)}D = d\iota_\g + \iota_\g d\]
is still the usual Lie derivative. 
Its Cartan complex % reference this?
\begin{equation}\label{eq:cartanSimp}
\begin{split}
C_G^*(\bX)&:= \Bigl (C_G(\Tot(\Omega^{*,*}(X_\bu)),D_G \Bigr),\\
D_G &:=(-1)^j\delta + D =  \del + (-1)^j\delta + (-1)^j d
\end{split}
\end{equation}
is also the total complex of the tricomplex
\begin{equation*} \label{eq:tricomplex}
C_G^{i,j,k}(\bX) := %\Bigl( S^{i}(\g^{\vee}) \tensor \Omega[i]^{k}(X_j) \Bigr)^{G} =
\Bigl( S^{i}(\g^{\vee}) \tensor \Omega^{k-i}(X_j) \Bigr)^{G}.
\end{equation*}
Note the grading is such that if $x \in C_G^{i,j,k}(\bX)$, then $\deg x = i + j+ k$. 
\end{ep}

\begin{ep}
For a simplicial $G$-manifold $\bX$, consider the Cartan complex of $\Omega^*_{\spl}(\bX)$ \eqref{eq:simplFormsDoubleComplex} with the usual $G^\star$-module structure:
\begin{equation}\begin{split}
C_{G,\spl}^*(\bX) :=& C_G(\Omega^*_{\spl}(\bX)) \\
 D_G :=& \delta + d^\Delta + (-1)^j d^X.
\end{split}\end{equation}
Note that this is the total complex of the tricomplex
\begin{equation}\begin{split}\label{eq:simplCartanComplex}
C_{G,\spl}^{i,j,k}(\bX) :=& \Bigl( S^{i}(\g^{\vee}) \tensor \Omega^{j,k-i}_{\spl}(\bX) \Bigr)^{G} \subset  \prod_{q=0}^\infty \Bigl(S^i(\g^\vee)\tensor\Omega^{j,k-i}(\Delta^q\times X_q)\Bigr)^G. \\
% C_{\spl}^{i,j,k}(\bX) := \Bigl( S^{i}(\g^{\vee}) \tensor \Omega[i]^{j,k}_{\spl}(\bX) \Bigr)^{G} \subset  \prod_{p=0}^\infty \Bigl(S^i(\g^\vee)\tensor\Omega^{j,k-i}(\Delta^p\times X_p)\Bigr)^G
\end{split}\end{equation}
\end{ep}

 Note that the $G^\star$-module structures on $\Omega(M)$, $\bOmega(\EGM)$ and $\bOmega_{\spl}(\EGM)$ are chosen in such a way, that the quasi-isomorphisms $\iota$ \eqref{eq:mfdSimpMfd}, $\pi_M^*$ \eqref{eq:piStar} and Dupont's map $\mathscr{C}$ \eqref{eq:dupontC} are maps of $G^\star$-modules. 
%  In particular, $\mathscr{C}\iota_\g^{\Omega^*(\bX)} = (-1)^j\mathscr{C}\iota_\g = \iota_\g\mathscr{C}$. 
Therefore, Lemma \ref{lem:inducedQuiso} now implies

\begin{prop} \label{prop:2}
If $G$ is compact, then the map induced by $\mathscr{C}\circ\pi_M^*\circ \iota\colon\bOmega(M)\to \Omega_{\spl}^*(\EGM)$ on the total Cartan complexes 
\[
\jmath:=\id_{S(\g^{\vee})} \tensor (\mathscr{C}\circ\pi_M^*\circ \iota) \maps C_{G}^*(M) \to C_{G,\spl}^* \bigl(\EGM \bigr).
\]
is a quasi-isomorphism.
\end{prop}

\subsection{The Cartan map}
Given a principal $G$-bundle $\pi\colon P\to B$ with connection
$A\in\Omega^1(P,\g)^G$, Cartan \cite{Cartan} constructed a chain map
$C_G^*(P)\to \Omega^*(B)$ now known as the \textbf{Cartan map}:
\begin{equation}
\Car^A\colon C_G^*(P) \to \Omega^*(P)^G_{\hor}\cong\Omega^*(B).
\end{equation}
If $\beta \in \bigl(S^i(\g^\vee)\otimes\Omega^*(P)\bigr)^G$, then
\[
\Car^A(\beta):= \hor_A\bigl(\langle (F_{A})^{i}
,\beta\rangle\bigr)\in\Omega^{*+2i}(P)^G_{\hor}.
\]
Here, $F_A= dA+ \frac{1}{2}[A,A]\in\Omega^2(P,\g)^G_{\hor}$ is the
curvature of $A$, 
\[ (F_{A})^{i} = \underset{i}{\underbrace{F_A\wedge\cdots\wedge F_A}}\in \Omega^{2i}(P,\g^{\otimes i})^G_{hor},\]
$\langle\cdot,\cdot\rangle$ denotes the natural pairing induced by $\g\otimes\g^\vee \to \R$, 
and $\hor_A\colon\Omega^*(P)^G\to\Omega^*(P)^G_{\hor}$ denotes the projection to horizontal forms defined by $A$.

\begin{remark}
 Recall that $\pi^*\colon \Omega^*(B)\to \Omega^*(P)^G_{\hor}$ is an
 isomorphism. If $G$ is compact, then $\pi^\ast$ composed with the inclusion 
\[\Omega^*(B)\xto{\pi^*}\Omega^*(P)^G_{\hor}\hookrightarrow C_G^*(P)\]
 induces an isomorphism in cohomology, with homotopy inverse $\Car^A$.
\end{remark}

There is also a simplicial version of this construction: Let $P_\bu\to B_\bu$ be a simplicial principal $G$-bundle with a simplicial connection $A\in\Omega_{\spl}^1(P_\bu,\g)^G$. The connection $A$ is given by a sequence of connections $A^{(q)}\in\Omega^1(\Delta^q\times P_q,\g)^G$ on the principal $G$-bundles $\Delta^q\times P_q\to \Delta^q\times B_p$. Applying the degree-wise Cartan maps 
\[\Car^{A^{(q)}}\colon C_G^*(\Delta^q\times P_q)\to \Omega^*(\Delta^q\times B_q)\]
gives a chain map between the total complexes \eqref{eq:simplCartanComplex} and \eqref{eq:simplFormsDoubleComplex}
\begin{equation}\label{eq:CartanMap}
\Car^A\colon C_{G,\spl}^*(P_\bu)\to \Omega_{\spl}^*(B_\bu).
\end{equation}

If $G$ is compact, then $\Car^A$ is a quasi-isomorphism.

\begin{ep}
 Let $M$ be a $G$-manifold and $\EGM\to E_\bu G\times_G M$ the simplicial principal $G$-bundle. For $i=0,\ldots,q$, let
\[\pi_i\colon E_q G\times M = G^{q+1}\times M\to G\]
 denote the projections and let $\theta_L\in\Omega^1(G,\g)^G$ denote the left-invariant Maurer-Cartan form on $G$.
Following Dupont \cite{Dupont}, we consider the canonical simplical connection %TODO: universal? canonical?
\begin{equation}\begin{split}\label{eq:conn}
 \theta=&\bigl(\theta^{(q)}\bigr)\in\Omega^{0,1}_{\spl}(\EGM,\g)^G,\\
 \theta^{(q)} :=& \sum_{i=0}^q t_i \pi_i^*\theta_L\in\Omega^1(\Delta^q\times E_q G\times M,\g)^G,
\end{split}\end{equation}
where $t_i$, $i=0,\ldots,q$ are barycentric coordinates on $\Delta^q$.
The curvature of $\theta$ is 
\begin{equation*}
 F_{\theta^{(q)}} = \underbrace{d^\Delta\theta^{(q)}}_{\textrm{type }1,1} + \underbrace{d^{E_q G\times M}\theta^{(q)}+ \tfrac{1}{2}[\theta^{(q)},\theta^{(q)}]}_{\textrm{type }0,2} %= \sum_{i=0}^q (dt_i\wedge\pi_i^*\theta_L + t_i d\pi_i^*\theta_L) + \sum_{i,j=0}^q t_it_j [\pi_i^*\theta_L,\pi_j^*\theta_L]
\in\Omega^{1,1}(\Delta^q\times E_qG \times M,\g)^G_{\hor}\oplus\Omega^{0,2}(\Delta^q\times E_qG\times M,\g)^G_{\hor}.
\end{equation*}

Making use of the trivializing section $s \maps G^q \times M \to
G^{q+1} \times M$ \eqref{eq:trivSection}, we can write the Cartan map for the connection $\theta^{(q)}$ on $\Delta^q \times E_q G\times M\to \Delta^q \times G^q\times M$ as
\begin{equation}\label{eq:CartanThetaq}\begin{split}
\Car^{\theta^{(q)}}\colon C_G^{i,j}(\Delta^q \times E_q G\times M)&\to \Omega^{i+j}(\Delta^q\times G^q\times M), \\
\beta&\mapsto \langle s^*F_{\theta^{(q)}}^i,s^*\hor_{\theta^{(q)}}\beta\rangle.
\end{split}\end{equation}

For example, for $q=0$, we have $\theta^{(0)} = \pi_0^*\theta_L$, $F_{\theta^{(0)}} = 0$ and $s^*\hor_{\theta^{(0)}} = s^*$, and hence
\begin{equation}\label{eq:CartanTheta0}
\Car^{\theta^{(0)}}(\beta) = 
\begin{cases}
 s^*\beta & \textrm{if }\beta\in C_G^{0,*}(G\times M), \\
 0 & \textrm{else}.
\end{cases}
\end{equation}

For $q=1$, we have 
\begin{equation}\label{eq:curvature}\begin{split}
\theta^{(1)} &=  t_0\pi_0^*\theta_L + t_1\pi_1^*\theta_L,\\ 
F_{\theta^{(1)}} &= -dt_1\wedge(\pi_0^*\theta_L-\pi_1^*\theta_L) - \tfrac{t_0t_1}{2}[\pi_0^*\theta_L-\pi_1^*\theta_L,\pi_0^*\theta_L-\pi_1^*\theta_L].
\end{split}\end{equation}

Let $I\colon G\to G$ denote the inversion on $G$, $L_g$ the diffeomorphism of $M$ corresponding to $g$ and let $\theta_R\in\Omega^1(G,\g)$ be the right-invariant Maurer-Cartan form. The differential of the section $s\colon G\times M\to G^2\times M$ %,(g,x) \mapsto (e,g^{-1},gx)$
 from \eqref{eq:trivSection} is
\begin{equation}\label{eq:ds} s_*|_{(g,p)}(\tilde{x},w) = \bigl(0,I_*(\tilde{x}),(L_g)_*(w)- v_{\theta_R(\tilde{x})}|_{gp}\bigr)~\textrm{ for } \tilde{x}\in T_gG, w\in T_pM.
\end{equation}
Therefore, 
\begin{equation}\label{eq:pushDownCurvature}\begin{split} s^* F_{\theta^{(1)}} &= -dt_1\wedge\pi_G^*\theta_R - \tfrac{t_0t_1}{2}\pi_G^*[\theta_R,\theta_R], \\
s^* F_{\theta^{(1)}}^i  &= (-1)^i i \tfrac{(t_0t_1)^{i-1}}{2^{i-1}} dt_1\wedge\pi_G^*\bigl(\theta_R\wedge[\theta_R,\theta_R]^{i-1}\bigr) + \bigl(-\tfrac{t_0t_1}{2} \pi_G^*[\theta_R,\theta_R]\bigr)^i.
\end{split}\end{equation}

The horizontal projection for the connection $\theta^{(1)}$ on $\Delta^1\times E_1G \times M = \Delta^1\times G^2\times M$ is given by 
\begin{equation*}\begin{split}
 T_{(t,g_0,g_1,p)}(\Delta^1\times G^2\times M)\to& T_{(t,g_0,g_1,p)}(\Delta^1\times G^2\times M), \\
 (a,\tilde{x}_0,\tilde{x}_1,w' )\mapsto& %(a,\tilde{x}_0 ,\tilde{x}_1,w')- v^{G^2\times M}_{\theta^{(1)}(\tilde{x}_0,\tilde{x}_1,w')} = 
 (a,\tilde{x}_0,\tilde{x}_1,w')  - v^{G^2\times M}_{t_0 \theta_L(\tilde{x}_0) + t_1\theta_L(\tilde{x}_1)},
\end{split}\end{equation*}
where $v^{G^2\times M}$ is the fundamental vector field for the action $G\curvearrowright G^2\times M$ \eqref{eq:actionEGM}. In particular, its component in $T_pM$ is 
\begin{equation}\label{eq:dpiHor} (\pi_M)_*\bigl(\hor_{\theta^{(1)}}(a,\tilde{x}_0,\tilde{x}_1,w')\bigr) = w' - t_0 v_{\theta_L(\tilde{x}_0)}|_p - t_1 v_{\theta_L(\tilde{x}_1)}|_p. \end{equation}
\end{ep}

\begin{remark}
 Note that, since $F_{\theta^{(q)}}$ is of homogeneous total degree $2$, but not of homogeneous bi-degree, the chain map 
 \begin{equation}
\Car^\theta\colon C^*_{G,\spl}(\EGM)\to \bOmega_{\spl}(G\ltimes M)
\end{equation}
only preserves the grading on the total complexes, but maps elements of homogeneous tri-degree to elements of inhomogeneous bi-degree.
\end{remark}

\begin{remark}
 Note that $\theta = \mathscr{C}(\pi_0^*\theta_L)$, with $\pi_0^*\theta_L\in\Omega^1(E_0G\times M,\g)^G = \Omega^1(G\times M,\g)^G$ the pullback of the left-invariant Maurer-Cartan form. 
\end{remark}

\subsection{Cartan complex and Bott-Shulman-Stasheff complex}

We can now compose the chain maps defined above and, if $G$ is compact, obtain a %\nrt{explicit?} 
quasi-isomorphism between the Cartan complex and the Bott-Shulman-Stasheff complex: 
\begin{prop}\label{prop:CartanBS}
 Let $G$ be a Lie group and $M$ a $G$-manifold. There is a natural chain map from the Cartan model to the Bott-Shulman-Stasheff model 
% \[C_G^*(M)\to C^*(G \ltimes M_{\bu}),\] 
% which is defined as the composition
 \begin{equation}\label{eq:CartanBS}
\Phi \maps   C_G^*(M)\xto{\jmath} C_{G,\spl}^*(\EGM)\xto{\Car^{\theta}}\Omega_{\spl}^*(G\ltimes M_\bu)\xto{\mathscr{I}} \bOmega(G \ltimes M_{\bu}),
%    C_G^*(M)\xto{\jmath}C_G^*(\EGM)\xto{\mathscr{C}}  C_{G,\spl}^*(\EGM)\xto{\Car^{\theta}}\Omega_{\spl}^*(M\ltimes G_\bu)\xto{\mathscr{I}} C^*(G \ltimes M_{\bu}),
% differentials are \delta + d, \del + (-1)^j d_G, \del + (-1)^j\delta + (-1)^j d^{\EGM},  \delta + d^\Delta + (-1)^j d^{\EGM}, d^\Delta + (-1)^jd^X,  \del + (-1)^j d^X, \del + (-1)^j d^X
 \end{equation}
 where 
\begin{itemize}
 \item $\jmath$ is the chain map from Prop.\ \ref{prop:2},
%  \item $\jmath$ is the chain map from Corollary \ref{cor:quisoCartanModels},
%  \item $\mathscr{C}$ is the chain map induced on Cartan complexes by the map \eqref{eq:dupontC} defined by Dupont in Theorem \ref{thm:Dupont}, 
 \item $\Car^\theta$ is the simplicial Cartan map \eqref{eq:CartanMap} for the simplical connection $\theta$ \eqref{eq:conn} on $\EGM \to E_\bu G\times_G M\cong G\ltimes M_\bu$, 
 \item $\mathscr{I}$ is the quasi-isomorphism \eqref{eq:dupontI} defined by Dupont in Thm.\ \ref{thm:Dupont}.
\end{itemize}

If $G$ is compact, then all of the above are quasi-isomorphisms and
hence $\Phi \maps C_G^*(M)\to \bOmega(G \ltimes M_{\bu})$ is a quasi-isomorphism.
\end{prop}

\begin{remark}
 Note that if $G$ is not compact, $\jmath$ and $\Car^\theta$ can fail to be quasi-isomorphisms. 

In the case of $\jmath$, taking invariants will in general not commute with taking cohomology, \cite[Thm.\ 6.5.1]{G-S} and hence Lemma \ref{lem:inducedQuiso} fails in general.

 If $G$ is non-compact, then, in general, the Cartan model does not compute the equivariant cohomology. Since the equivariant cohomology of a principal $G$-bundle equals the cohomology of the base, the Cartan map $\Car^\theta$ is not an isomorphism in general. 
\end{remark}

\subsection{Homotopy moment maps from Cartan cocycles} \label{sec:CartanMM_proof}
We are now in the position to prove Thm. \ref{thm:CartanMM} by applying Thm.\ \ref{thm:BSmm} to the image of a cocycle in the Cartan complex under the chain map \eqref{eq:CartanBS} and to obtain an explicit homotopy moment map:

\begin{proof}[Proof of Thm. \ref{thm:CartanMM}]
Let $\omega + \sum_{i=1}^{\lfloor\frac{n+1}{2}\rfloor} P_i\in
C_G^{n+1}(M)$ be a cocycle. If $G$ is compact, then using
the chain map $\Phi$ \eqref{eq:CartanBS} and the second part of Thm.\
\ref{thm:BSmm}, we immediately obtain a homotopy moment map. 

If $G$ is an arbitrary Lie group, we will show that the
$\Omega^{0,n+1}(G \ltimes M_{\bu}) = \Omega^{n+1}(M)$ and
$\Omega^{1,n}(G \ltimes M_{\bu}) = \Omega^{n}(G\times M)$-components
of the image of $\omega + \sum_{i=1}^{\lfloor\frac{n+1}{2}\rfloor}
P_i$ in $\Omega^{n+1}(G \ltimes M_{\bu})$ are, in fact,
$G$-invariant. This will then allow us to use the first part of Thm.\ \ref{thm:BSmm} to construct a homotopy moment map.

% First note, that 
% \begin{equation*}\begin{split}
% \jmath(\omega) &= \pi_M^*\omega \in \Omega^{n+1}(G\times M)^G =C_G^{0,0,n+1}(\EGM), \\
% \jmath(P_i) &=\pi_M^* P_i\in \bigl(S^i(\g^\vee)\otimes\Omega^{n-2i+1}(G\times M)\bigr)^G =C_G^{i,0,n-i+1}(\EGM),
% \end{split}\end{equation*}
% where $\pi_M\colon G\times M \to M$ is the projection.
% Therefore, 
% \begin{equation*}
%  \jmath\Bigl(\omega + \sum_{i=1}^{\lfloor \frac{n+1}{2}\rfloor} P_i\Bigr) = \pi_M^*\omega + \sum_{i=1}^{\lfloor \frac{n+1}{2}\rfloor} \pi_M^*P_i.
% \end{equation*}
The elements
\begin{equation*}\begin{split}
 \jmath(\omega)&\in C_{G,\spl}^{0,0,n+1}(\EGM)\subset \prod_{q=0}^\infty \Omega^{0,n+1}(\Delta^q\times E_q G\times M)^G, \\
 \jmath(P_i)&\in C_{G,\spl}^{i,0,n-i+1}(\EGM) \subset %\prod_{q=0}^\infty \bigoplus_{\ell=0}^i\Bigl(S^i(\g^\vee)\tensor\Omega^{-\ell,n-2i+\ell}(\Delta^q\times E_qG\times M)\Bigr)^G  \\
%&= 
\prod_{q=0}^\infty \Bigl(S^i(\g^\vee)\tensor\Omega^{0,n-2i+1}(\Delta^q\times E_qG\times M)\Bigr)^G,
\end{split}\end{equation*}
are given by the sequences $\jmath(\omega)^{(q)}=\pi_M^*\omega$ and $\jmath(P_i)^{(q)} = \pi_M^*P_i$, respectively.

The images of $\jmath(\omega)$ and $\jmath(P_i)$ under the simplicial Cartan map $\Car^\theta$ are given by the sequences %$\Car^{\theta^{(q)}}(\pi_M^*\omega)$ and $\Car^{\theta^{(q)}}(\pi_M^*P_i)$ for all $q$. 
% Recall that the Cartan map was defined by inserting the curvature, taking a horizontal component and then pushing the resulting $G$-invariant horizontal form down to the base. 
\begin{equation}\begin{split}\label{eq:CarOmegaPi}
 \Car^{\theta^{(q)}}(\pi_M^*\omega) &= s^*\hor_{\theta^{(q)}}\pi_M^*\omega, \\
 \Car^{\theta^{(q)}}(\pi_M^*P_i)    &= %s^*\hor_{\theta^{(q)}}\langle F_{\theta^{(q)}}^i,\pi_M^*P_i\rangle = 
 \langle s^* F_{\theta^{(q)}}^i, s^*\hor_{\theta^{(q)}} \pi_M^*P_i\rangle,
\end{split}\end{equation}
respectively. Here $s$ is the trivializing section \eqref{eq:trivSection}.

Keeping in mind that Thm.\ \ref{thm:BSmm} only uses the components in $\Omega^{0,n+1}(G \ltimes M_{\bu})$ and $\Omega^{1,n}(G \ltimes M_{\bu})$, we only need to compute the $(0,n+1)$- and $(1,n)$-components of $\mathscr{I}\bigl(\Car^\theta(\jmath(\omega))\bigr)$ and $\mathscr{I}\bigl(\Car^\theta(\jmath(P_i))\bigr)$. 
Since $\mathscr{I}$ is a map of bicomplexes, we only need the Cartan maps for $\theta^{(0)}$ and $\theta^{(1)}$. 

Equation \eqref{eq:CartanTheta0} implies that the $\Omega^{0,n+1}(G \ltimes M_{\bu})$-component of $\mathscr{I}\bigl(\Car^\theta(\jmath(\omega + \sum_i P_i))\bigr)$ is indeed 
\begin{equation*}
\int_{\Delta^0}\Car^{\theta^{(0)}}\Bigl(\pi_M^*\omega+\sum_{i=1}^{\lfloor \frac{n+1}{2}\rfloor} P_i\Bigr) = s^*\pi_M^*\omega = \omega, % \in C^{0,n}(G \ltimes M_{\bu})^G, \\
\end{equation*}
and, in particular, $G$-invariant.

We now turn to the $\Omega^{1,n}(G \ltimes M_{\bu})$-component. Since $\Car^{\theta^{(1)}}(\pi_M^*\omega)\in\Omega^{0,{n+1}}(\Delta^1\times E_1 G\times M)^G$, we have 
\begin{equation*}\begin{split}
 \int_{\Delta^1}\Car^{\theta^{(1)}}(\pi_M^*\omega) &= 0.
\end{split}\end{equation*}

Thus, the  $\Omega^{1,n}(G \ltimes M_{\bu})$-component of $\mathscr{I}\bigl(\Car^\theta(\jmath(\omega + \sum_i P_i))\bigr)$, from which the homotopy moment map is constructed, is 
\[\sum_{i=1}^{\lfloor \frac{n+1}{2}\rfloor}\int_{\Delta^1} \Car^{\theta^{(1)}}(\pi_M^*P_i) = \sum_{i=1}^{\lfloor \frac{n+1}{2}\rfloor}\int_{\Delta^1} \langle s^* F_{\theta^{(1)}}^i, s^*\hor_{\theta^{(1)}} \pi_M^*P_i\rangle.\]
We will now compute this explicitly, and also show that it defines a $G$-invariant $n$-form on $G\times M$, so that we can apply the first part of Thm.\ \ref{thm:BSmm}.

Combining \eqref{eq:ds} and \eqref{eq:dpiHor}, and using $I^*\theta_L = -\theta_R$ as well as $t_0=1-t_1$ and $(L_{g^{-1}})_* (v_{\theta_R(\tilde{x})}) = v_{\theta_L(\tilde{x})}$, we have
%$dL_g(di(\tilde{x})) = -dR_{g^{-1}}(\tilde{x})$, hence $\theta_L(di(\tilde{x})) = -\theta_R(\tilde{x})$
\begin{equation*}\begin{split}
 (\pi_M)_*\bigl(\hor_{\theta^{(1)}} s_*(\tilde{x},w)\bigr) =&  (\pi_M)_*\bigl(\hor_{\theta^{(1)}}(0,I_*(\tilde{x}),(L_g)_*(w)- v_{\theta_R(\tilde{x})}|_{g p})\bigr)\\
 =& (L_g)_*(w)- v_{\theta_R(\tilde{x})}|_{gp} %- t_0 v_{\theta_L|_{e}(0)}|_{gp} 
- t_1 v_{I^*\theta_L(\tilde{x})}|_{gp} 
 =(L_g)_*(w) - t_0 v_{\theta_R(\tilde{x})}|_{gp} \\
 =&(L_g)_*(w - t_0 v_{\theta_L(\tilde{x})}|_{p})
\end{split}\end{equation*}
 for all $(\tilde{x},w)\in T_{(g,p)}(G\times M)$. Using the $G$-invariance of $P_i$, i.e. $L_g^*P_i = \Ad_g^\vee P_i$, we have
% \begin{equation*}\begin{split}
% &s^*\hor_{\theta^{(1)}} \pi_M^*P_i|_{(g,p)}\bigl((\tilde{x}_1, w_1),\ldots, (\tilde{x}_{n-2i+1}, w_{n-2i+1})\bigr) \\
% % =& P_i|_{gp}\bigl(dL_g(w_1 - t_0 v_{\theta_L(\tilde{x}_1)}|_{p}),\ldots,dL_g(w_{n-2i+1} - t_0 v_{\theta_L(\tilde{x}_{n-2i+1})}|_{p})\bigr)\\
% =& L_g^*P_i|_{p}\bigl(w_1 - t_0 v_{\theta_L(\tilde{x}_1)}|_{p},\ldots,w_{n-2i+1} - t_0 v_{\theta_L(\tilde{x}_{n-2i+1})}|_{p}\bigr)\\
% =& \bigl(\Ad_g^\vee\bigr)^{\otimes i}P_i|_{p}\bigl(w_1 - t_0 v_{\theta_L(\tilde{x}_1)}|_{p},\ldots,w_{n-2i+1} - t_0 v_{\theta_L(\tilde{x}_{n-2i+1})}|_{p}\bigr).
% \end{split}\end{equation*}
\begin{equation}\label{eq:horPi}
 s^*\hor_{\theta^{(1)}}\pi_M^*P_i = (\Ad_g^\vee)^{\otimes i}P_i \circ \phi_{t_0}^{\otimes(n-2i+1)},
\end{equation}
where $\phi_{t_0}$ denotes the map $T_gG\oplus T_pM \ni (\tilde{x},w) \mapsto w-t_0 v_{\theta_L(\tilde{x})}|_{p}\in T_p M$.
% TODO: rephrase using $(\tilde{x}_1,w)\mapsto(t_0\tilde{x}_1,w_1)$?
Combining \eqref{eq:CartanThetaq}, \eqref{eq:pushDownCurvature}, \eqref{eq:horPi} and $\Ad_{g^{-1}}\theta_R = \theta_L$, we have 
\begin{equation*}
 \Car^{\theta^{(1)}}(\pi_M^*P_i) = \bigl\langle(-1)^i i \tfrac{(t_0t_1)^{i-1}}{2^{i-1}}   dt_1\wedge\pi_G^*\bigl(\theta_L\wedge[\theta_L,\theta_L]^{i-1}\bigr)+ \bigl(-\tfrac{t_0t_1}{2} \pi_G^*[\theta_L,\theta_L]\bigr)^i, P_i \circ \phi_{t_0}^{\otimes(n-2i+1)} \bigr\rangle.
\end{equation*}
Since $\phi_{t_0}$ is invariant under the left action $(g,(t,h,p))\mapsto (t,gh,p)$, this also proves that \[\Car^{\theta^{(1)}}(\pi_M^*P_i)\in\Omega^{n}(\Delta^1\times G\times M)^G,\]
 where $G$ acts by the same action. Hence also
\begin{equation}\label{eq:CarPi}
 \int_{\Delta^1} \Car^{\theta^{(1)}}(\pi_M^*P_i) =\int_{\Delta^1} (-1)^i i \tfrac{(t_0t_1)^{i-1}}{2^{i-1}}   dt_1\wedge\bigl\langle\pi_G^*\bigl(\theta_L\wedge[\theta_L,\theta_L]^{i-1}\bigr), P_i \circ \phi_{t_0}^{\otimes(n-2i+1)} \bigr\rangle\in \Omega^{1,n}(G \ltimes M_{\bu})^G.
\end{equation}
However, recall that $\phi_{t_0}$ depends on $t_0 =1- t_1$. For $x_1,\ldots, x_k\in\g$ and $w_1,\ldots,w_{n-k}\in T_pM$ we have 
\begin{equation*}\begin{split}
&\bigl\langle\pi_G^*\bigl(\theta_L\wedge[\theta_L,\theta_L]^{i-1}\bigr), P_i \circ \phi_{t_0}^{\otimes(n-2i+1)} \bigr\rangle \bigl((x_1,0),\ldots,(x_k,0),(0,w_1),\ldots,(0,w_{n-k})\bigr) \\
=&  \sum_{\sigma\in Sh(2i-1,k-2i+1)}(-1)^\sigma(-t_0)^{k-2i+1}\langle\theta_L\wedge[\theta_L,\theta_L]^{i-1}(x_{\sigma(1)},\ldots,x_{\sigma(2i-1)}), P_i(v_{x_{\sigma(2i)}},\ldots,v_{x_{\sigma(k)}},w_1,\ldots,w_{n-k}) \rangle \\
=& \tfrac{k!t_0^{k-2i+1}}{(k-2i+1)!}\Alt_k\bigl(\iota_\g^{k-2i+1} P_i(\cdot,\underbrace{[\cdot,\cdot],\ldots,[\cdot,\cdot]}_{i-1})\bigr)(x_1,\ldots,x_k)(w_1,\ldots,w_{n-k}).
\end{split}\end{equation*}
Combining this with \eqref{eq:CarPi}, and since $\int_0^1 t_0^{k-i} t_1^{i-1}dt_1 =  \frac{(i-1)!(k-i)!}{k!}$, we see that the image of $\omega+ \sum_{i=1}^{\lfloor\frac{n+1}{2}\rfloor}P_i$ in $C_\g^{*}(M)$ is 
% $\int_0^1 t_0^{a-1} t_1^{b-1}dt_1 =  \frac{(a-1)!(b-1)!}{(a+b-1)!}$
\begin{equation*}\begin{split}
  f^\vs:=  \sum_{k=1}^{n}\sum_{i=1}^{\lfloor\frac{n+1}{2}\rfloor}r\Bigl(\int_{\Delta^1}\Car^{\theta^{(1)}}(\pi_M^*P_i)\Bigr)_k
=&\sum_{k=1}^{n}\underbrace{\sum_{i=1}^{\lfloor\frac{k+1}{2}\rfloor} \tfrac{(-1)^i i!(k-i)!}{2^{i-1}(k-2i+1)!}  \Alt_k\bigl(\iota_\g^{k-2i+1}P_i(\cdot,\underbrace{[\cdot,\cdot],\ldots,[\cdot,\cdot]}_{i-1})\bigr)}_{ f^\vs_k}.
\end{split}\end{equation*}
With $f_k = \vs(k)f^\vs_k$, this completes the proof.
\end{proof}

%  \appendix

 %%%%%%%%%

% \bibliographystyle{habbrv}
% \bibliography{Homomap}

\begin{thebibliography}{10}

\bibitem{AMM}
A.~Alekseev, A.~Malkin, and E.~Meinrenken.
\newblock Lie group valued moment maps.
\newblock {\em J. Differential Geom.}, 48(3):445--495, 1998.
\newblock Also available as
  \href{http://arxiv.org/abs/dg-ga/9707021}{arXiv:math/dg-ga/9707021}.

\bibitem{A-B:1983}
M.~F. Atiyah and R.~Bott.
\newblock The {Y}ang-{M}ills equations over {R}iemann surfaces.
\newblock {\em Philos. Trans. Roy. Soc. London Ser. A}, 308(1505):523--615,
  1983.

\bibitem{hd6}
J.~C. Baez and A.~S. Crans.
\newblock Higher-dimensional algebra. {VI}. {L}ie 2-algebras.
\newblock {\em Theory Appl. Categ.}, 12:492--538, 2004.
\newblock Also available as
  \href{http://arxiv.org/abs/math/0307263}{arXiv:math/0307263}.

\bibitem{Baez-Rogers:2010}
J.~C. Baez and C.~L. Rogers.
\newblock Categorified symplectic geometry and the string {L}ie 2-algebra.
\newblock {\em Homology, Homotopy Appl.}, 12(1):221--236, 2010.
\newblock Also available as \href{
  http://arxiv.org/abs/0901.4721}{arXiv:0901.4721}.

\bibitem{Brylinski_book}
J.-L. Brylinski.
\newblock {\em Loop spaces, characteristic classes and geometric quantization}.
\newblock Modern Birkh{\"a}user Classics. Birkh{\"a}user Boston Inc., Boston,
  MA, 2008.
\newblock Reprint of the 1993 edition.

\bibitem{BSS:1976}
R.~Bott, H.~Shulman, and J.~Stasheff.
\newblock On the de {R}ham theory of certain classifying spaces.
\newblock {\em Advances in Math.}, 20(1):43--56, 1976.

\bibitem{BCG}
H.~Bursztyn, G.~R. Cavalcanti, and M.~Gualtieri.
\newblock Reduction of {C}ourant algebroids and generalized complex structures.
\newblock {\em Adv. Math.}, 211(2):726--765, 2007.
\newblock Also available as
  \href{http://arxiv.org/abs/math/0509640}{arXiv:0509640}.

\bibitem{Ana}
A.~Cannas~da Silva.
\newblock {\em Lectures on symplectic geometry}, volume 1764 of {\em Lecture
  Notes in Mathematics}.
\newblock Springer-Verlag, Berlin, 2001.

\bibitem{IbortCarCrampMultimomap}
J.~F. Cari{\~n}ena, M.~Crampin, and L.~A. Ibort.
\newblock On the multisymplectic formalism for first order field theories.
\newblock {\em Differential Geom. Appl.}, 1(4):345--374, 1991.

\bibitem{Cartan}
H.~Cartan.
\newblock La transgression dans un groupe de {L}ie et dans un espace fibr\'e
  principal.
\newblock In {\em Colloque de topologie (espaces fibr\'es), {B}ruxelles, 1950},
  pages 57--71. Georges Thone, Li\`ege; Masson et Cie., Paris, 1951.

\bibitem{Collier:2011}
B.~Collier.
\newblock Infinitesimal symmetries of {D}ixmier-{D}ouady gerbes.
\newblock Available as \href{http://arxiv.org/abs/1108.1525}{arXiv:1108.1525}.

\bibitem{DHR:2015}
V.~A. Dolgushev, A.~E. Hoffnung, and C.~L. Rogers.
\newblock What do homotopy algebras form?
\newblock {\em Adv.\ Math.}, 274: 562--605, 2015.
\newblock \href{http://arxiv.org/abs/1406.1751}{arXiv:1406.1751}.

\bibitem{Dupont}
J.~L. Dupont.
\newblock Simplicial de {R}ham cohomology and characteristic classes of flat
  bundles.
\newblock {\em Topology}, 15(3):233--245, 1976.

\bibitem{FHT}
Y.~F{\'e}lix, S.~Halperin, and J.-C. Thomas.
\newblock {\em Rational homotopy theory}, volume 205 of {\em Graduate Texts in
  Mathematics}.
\newblock Springer-Verlag, New York, 2001.

\bibitem{FRS}
D.~Fiorenza, C.~L. Rogers, and U.~Schreiber.
\newblock {$L_\infty$-algebras of local observables from higher prequantum
  bundles}.
\newblock {\em Homology, Homotopy Appl.}, 16: 107--142, 2014.
\newblock \href{http://arxiv.org/abs/1304.6292}{arXiv:1304.6292}.

\bibitem{Forger}
M.~Forger, C.~Paufler, and H.~R{{\"o}}mer.
\newblock The {P}oisson bracket for {P}oisson forms in multisymplectic field theory.
\newblock {\em Rev. Math. Phys.}, 15(7):705--744, 2003.
\newblock Also available as
  \href{http://arxiv.org/abs/math-ph/0202043}{arXiv:0202043}.

% \bibitem{Forger}
% M.~Forger, C.~Paufler, and H.~R{{\"o}}mer.
% \newblock Hamiltonian multivector fields and {P}oisson forms in multisymplectic
%   field theory.
% \newblock {\em J. Math. Phys.}, 46(11):112903, 29, 2005.
% \newblock Also available as
%   \href{http://arxiv.org/abs/math-ph/0407057}{arXiv:0407057}.
  
  
\bibitem{FLRZ}
Y.~Fr{{\'e}}gier, C.~Laurent-Gengoux, and M.~Zambon.
\newblock A cohomological framework for homotopy moment maps.
\newblock {\em J. Geom. Phys.}, 97:119--132, 2015.  
\newblock Available as \href{http://arxiv.org/abs/1409.3142}{arXiv:1409.3142}
   
\bibitem{ccc2}
W.~Greub, S.~Halperin, and R.~Vanstone.
\newblock {\em Connections, curvature, and cohomology. {V}ol. {II}: {L}ie
  groups, principal bundles, and characteristic classes}.
\newblock Academic Press [A subsidiary of Harcourt Brace Jovanovich,
  Publishers], New York-London, 1973.
\newblock Pure and Applied Mathematics, Vol. 47-II.

\bibitem{ccc3}
W.~Greub, S.~Halperin, and R.~Vanstone.
\newblock {\em Connections, curvature, and cohomology. {V}ol. {III}:
  {C}ohomology of {P}rincipal {B}undles and {H}omogeneous {S}paces}.
\newblock Academic Press [A subsidiary of Harcourt Brace Jovanovich,
  Publishers], New York-London, 1976.
\newblock Pure and Applied Mathematics, Vol. 47-III.

\bibitem{Gomi:2004}
K.~Gomi.
\newblock Reduction of strongly equivariant bundle gerbes with connection and
  curving.
\newblock Available as
  \href{http://arxiv.org/abs/math/0406144}{arXiv:math/0406144}.

\bibitem{GIMM}
M.~Gotay, J.~Isenberg, J.~Marsden, and R.~Montgomery.
\newblock Momentum maps and classical relativistic fields. part i: covariant
  field theory.
\newblock Available as
  \href{http://arxiv.org/abs/physics/9801019}{arXiv:physics/9801019}.

\bibitem{G-S}
V.~W. Guillemin and S.~Sternberg.
\newblock {\em Supersymmetry and equivariant de {R}ham theory}.
\newblock Mathematics Past and Present. Springer-Verlag, Berlin, 1999.

% \bibitem{Hinich:2001}
% V.~Hinich.
% \newblock D{G} coalgebras as formal stacks.
% \newblock {\em J. Pure Appl. Algebra}, 162(2-3):209--250, 2001.
% \newblock Also available as
%   \href{http://arxiv.org/abs/math/9812034}{arXiv:9812034}.

\bibitem{Kostant:1970}
B.~Kostant.
\newblock Quantization and unitary representations. {I}. {P}requantization.
\newblock In {\em Lectures in {M}odern {A}nalysis and {A}pplications, {III}},
  pages 87--208. Lecture Notes in Math., Vol. 170. Springer, Berlin, 1970.

\bibitem{Lada-Markl}
T.~Lada and M.~Markl.
\newblock Strongly homotopy {L}ie algebras.
\newblock {\em Comm. Algebra}, 23(6):2147--2161, 1995.
\newblock Also available as
  \href{http://arxiv.org/abs/hep-th/9406095}{arXiv:hep-th/9406095}.

\bibitem{LoVa} Jean--Louis Loday and Bruno Vallette, \emph{Algebraic operads}, Grundlehren Math. Wiss. 346, Springer, Heidelberg, 2012. 
\newblock Also available as
  \href{http://www-irma.u-strasbg.fr/~loday/PAPERS/LodayVallette.pdf}.

\bibitem{MadsenSwannClosed}
T.~Madsen and A.~Swann.
\newblock Closed forms and multi-moment maps.
\newblock {\em Geometriae Dedicata}, pages 1--28, 2012.
\newblock Also available as
  \href{http://arxiv.org/abs/1110.6541}{arXiv:math/1110.6541}.

\bibitem{MadsenSwannMultimomap}
T.~B. Madsen and A.~Swann.
\newblock Multi-moment maps.
\newblock {\em Adv. Math.}, 229(4):2287--2309, 2012.
\newblock Also available as
  \href{http://arxiv.org/abs/1012.2048}{arXiv:1012.2048}.

\bibitem{McCleary}
J.~McCleary.
\newblock {\em A user's guide to spectral sequences}, volume~58 of {\em
  Cambridge Studies in Advanced Mathematics}.
\newblock Cambridge University Press, Cambridge, second edition, 2001.

\bibitem{Meinrenken}
E.~Meinrenken.
\newblock Witten's formulas for intersection pairings on moduli spaces of flat
  {$G$}-bundles.
\newblock {\em Adv. Math.}, 197(1):140--197, 2005.
  
\bibitem{Pre-Seg}
A.~Pressley and G.~Segal.
\newblock {\em Loop groups}.
\newblock Oxford Mathematical Monographs. The Clarendon Press Oxford University
  Press, New York, 1986.
\newblock Oxford Science Publications.

\bibitem{Quillen:1969}
D.~Quillen.
\newblock Rational homotopy theory.
\newblock {\em Ann. of Math. (2)}, 90:205--295, 1969.

\bibitem{RogersCou}
C.~L. Rogers.
\newblock Courant algebroids from categorified symplectic geometry.
\newblock Available as \href{http://arxiv.org/abs/1001.0040}{arXiv:1001.0040}.

\bibitem{RogersL}
C.~L. Rogers.
\newblock {$L_\infty$}-algebras from multisymplectic geometry.
\newblock {\em Lett. Math. Phys.}, 100(1):29--50, 2012.
\newblock Also available as
  \href{http://arxiv.org/abs/1005.2230}{arXiv:1005.2230}.

\bibitem{RogersPre}
C.~L. Rogers.
\newblock 2-plectic geometry, {C}ourant algebroids, and categorified
  prequantizaion.
\newblock {\em J.\ Symplectic Geom.}, 11:53--91, 2013.
\newblock Also available as
  \href{http://arxiv.org/abs/1009.2975}{arXiv:1009.2975}.

\bibitem{Dima}
D.~Roytenberg.
\newblock On the structure of graded symplectic supermanifolds and {C}ourant
  algebroids.
\newblock In {\em Quantization, Poisson brackets and beyond (Manchester,
  2001)}, volume 315 of {\em Contemp. Math.}, pages 169--185. Amer. Math. Soc.,
  Providence, RI, 2002.
\newblock Also available as
  \href{http://arxiv.org/abs/math/0203110}{arXiv:0203110}.

\bibitem{rw}
D.~Roytenberg and A.~Weinstein.
\newblock Courant algebroids and strongly homotopy {L}ie algebras.
\newblock {\em Lett. Math. Phys.}, 46(1):81--93, 1998.
\newblock Also available as
  \href{http://arxiv.org/abs/math/9802118}{arXiv:/9802118}.

\bibitem{S-P:2011}
C.~J. Schommer-Pries.
\newblock Central extensions of smooth 2-groups and a finite-dimensional string
  2-group.
\newblock {\em Geom. Topol.}, 15(2):609--676, 2011.
\newblock Also available as
  \href{http://arxiv.org/abs/0911.2483}{arXiv:0911.2483}.

\bibitem{Segal:1974}
G.~Segal.
\newblock Categories and cohomology theories.
\newblock {\em Topology}, 13:293--312, 1974.

\bibitem{Souriau:1967}
J.-M. Souriau.
\newblock Quantification g\'eom\'etrique. {A}pplications.
\newblock {\em Ann. Inst. H. Poincar\'e Sect. A (N.S.)}, 6:311--341, 1967.

\bibitem{Stacey}
A.~Stacey.
\newblock The differential topology of loop spaces.
\newblock Available as \href{http://arxiv.org/abs/math/0510097}{arXiv:0510097}.

\bibitem{UribeDGman}
B.~Uribe.
\newblock Group {A}ctions on {DG}-{M}anifolds and {E}xact {C}ourant
  {A}lgebroids.
\newblock {\em Comm. Math. Phys.}, 318(1):35--67, 2013.
\newblock Also available as
  \href{http://arxiv.org/abs/1010.5413}{arXiv:1010.5413}.

\bibitem{Wockel:2007}
C.~Wockel.
\newblock Lie group structures on symmetry groups of principal bundles.
\newblock {\em J. Funct. Anal.}, 251(1):254--288, 2007.
\newblock Also available as
  \href{http://arxiv.org/abs/math/0612522}{arXiv:math/0612522}.

\bibitem{HDirac}
M.~Zambon.
\newblock {$L_{\infty}$-algebras and higher analogues of Dirac structures and
  Courant algebroids}.
\newblock {\em J. Symplectic Geom.}, 10(6), 2012.
\newblock Also available as
  \href{http://arxiv.org/abs/1003.1004}{arXiv:1003.1004}.


\end{thebibliography}
\end{document}